\documentclass[11pt, reqno]{amsart}
\usepackage{amssymb}
\usepackage{amsmath}
\usepackage{amsthm}
\usepackage{braket}
\usepackage{pdfpages}
\usepackage{graphicx}
\usepackage{caption}
\usepackage{subcaption}
\usepackage{mathrsfs} 
\usepackage{tikz-cd} 
\usepackage{bm}

\usepackage[utf8]{inputenc}
\usepackage[english]{babel}
\usepackage{hyperref}

\numberwithin{equation}{section}
\theoremstyle{plain}
\newtheorem{theorem}{Theorem}[section]

\theoremstyle{theorem}
\newtheorem{prop}[theorem]{Proposition}
\newtheorem{lem}[theorem]{Lemma}
\newtheorem{cor}[theorem]{Corollary}

\newtheorem{question}[theorem]{Question}
\newtheorem*{question*}{Question}

\theoremstyle{definition}
\newtheorem{defn}[theorem]{Definition}
\newtheorem{rmk}[theorem]{Remark}
\newtheorem*{remark}{Remark}

\newcommand{\R}{\mathbb{R}}
\newcommand{\C}{\mathbb{C}}
\newcommand{\M}{\mathcal{M}}
\newcommand{\HC}{\mathcal{H}}
\newcommand{\Com}{\mathbb{C}}
\newcommand{\Proj}{\mathbb{P}}
\newcommand{\Q}{\mathbb{Q}}
\newcommand{\Z}{\mathbb{Z}}
\newcommand{\Hyp}{\mathbb{H}}
\newcommand{\PS}{\mathbb{L}}
\newcommand{\N}{\mathbb{N}}

\newcommand{\E}{\mathscr{E}}

\newcommand{\Aff}{\mathbb{A}}

\DeclareMathOperator{\dist}{d}
\DeclareMathOperator{\Rat}{Rat}
\DeclareMathOperator{\Poly}{Poly}
\DeclareMathOperator{\PSL}{PSL}

\DeclareMathOperator{\Isom}{Isom}

\DeclareMathOperator{\Res}{Res}

\DeclareMathOperator{\Affine}{Aff}

\DeclareMathOperator{\Int}{Int}
\DeclareMathOperator{\hull}{hull}

\newcommand{\powerset}{\raisebox{.15\baselineskip}{\Large\ensuremath{\wp}}}

\numberwithin{figure}{section}

\title{Trees, length spectra for rational maps via barycentric extensions and Berkovich spaces}
\author{Yusheng Luo}
\address{Institute of Mathematical Science, Dept. of Mathematics, Stony Brook University, Stony Brook, NY 11794 USA}
\email{yusheng.s.luo@gmail.com}
\date{\today}
\begin{document}

\begin{abstract}
In this paper, we study the dynamics of degenerating sequences of rational maps on Riemann sphere $\hat \C$ using $\R$-trees.
As an analogue of isometric group actions on $\R$-trees for Kleinian groups, we give two constructions for limiting dynamics on $\R$-trees: one geometric and one algebraic.
The geometric construction uses the limit of rescalings of barycentric extensions of rational maps, while the algebraic construction uses the Berkovich space of complexified Robinson's field.
We show that the two approaches are equivalent.
As an application, we use it to give a classification of hyperbolic components of rational maps that admit degeneracies with bounded multipliers.
\end{abstract}

\maketitle

\setcounter{tocdepth}{1}
\tableofcontents

%%%%%%%%%%
%% Section 1
%% %%%%%%%%

\section{Introduction}
The study of dynamics of rational maps $f:\hat\C \longrightarrow \hat\C$ has been a central topic in mathematics.
The space $\Rat_d(\C)$ of degree $d$ rational maps is not compact, so it is interesting and useful to understand the dynamics as rational maps degenerate.
Let $f_n \to \infty$ in $\Rat_d(\C)$.
In \cite{L19}, we started the investigation of degenerating rational maps using the barycentric method. In particular, we constructed a limiting branched covering on $\R$-trees.
The construction is done in $2$ steps (see \S \ref{LimitingMapViaBarycentricExt} and \S \ref{UltralimitAC}):
\begin{enumerate}
\item Using the barycentric extension, we first extend the rational map $f_n$ to $\E f_n : \Hyp^3 \longrightarrow \Hyp^3$.
\item By choosing an appropriate rescalings 
$$
r_n = \max_{y\in \E f_n^{-1}(\bm 0)} d_{\Hyp^3} (\bm 0, y),
$$ 
we get a limiting map on the ultralimit ${^r\Hyp^3}$ (also known as the {\em asymptotic cone}) of pointed metric spaces $(\Hyp^3, \bm 0, d_{\Hyp^3}/r_n)$
$$
F: {^r\Hyp^3} \longrightarrow {^r\Hyp^3}.
$$
\end{enumerate}

Dynamics on $\R$-trees also arises naturally via Berkovich spaces (see \cite{BakerRumely10,Kiwi15}).
Using the Berkovich space of the complexified Robinson's field, we construct a limiting dynamical system on the $\R$-tree. This algebraic construction is done in $2$ steps as well (see \S \ref{LimitingMapViaBerk} and \S \ref{CRF}):
\begin{enumerate}
\item First we associate to the sequence $f_n$ a degree $d$ rational map $\mathbf{f}$ with coefficients in the complexified Robinson's field ${^\rho\Com}$ associated with the sequence $\rho_n = e^{-r_n}$.
\item Using the Berkovich extension, the rational map $\mathbf{f}$ naturally extends to a map on the Berkovich hyperbolic space, 
$$
B: \Hyp_{Berk}({^\rho\Com}) \longrightarrow \Hyp_{Berk}({^\rho\Com}).
$$ 
\end{enumerate}

We establish a connection between the two constructions:
\begin{theorem}\label{EquivTheorem}
There is a canonical isometric bijection
$$
\Phi: \Hyp_{Berk}({^\rho\Com}) \longrightarrow {^r\Hyp^3},
$$
such that
$$
\Phi\circ B = F \circ \Phi.
$$
\end{theorem}
\begin{remark}
Both the asymptotic cone ${^r\Hyp^3}$ and the complexified Robinson's field ${^\rho\Com}$ use non-principal ultrafilters in the construction. 
We remark that the same ultrafilter $\omega$ is used in both constructions.

We want to emphasize that each perspective of the limiting map brings its own unique advantages and benefits. 
The barycentric method allows us to work with degenerating sequence of conjugacy classes (see Theorem \ref{DGLC}) and gives geometric information for the dynamics of $f_n$. The Berkovich approach gives additional algebraic structure of the $\R$-tree.

We also give a version (see Theorem \ref{ConnectionToBerkovichDynamicsFamily}) for degenerating families of rational maps that shows our construction generalizes the use of the Berkovich space of the field of Puiseux series (cf. \cite{Kiwi15}).
\end{remark}

The limiting dynamics of $F$ (or equivalently $B$) on the $\R$-tree is useful in recovering the limiting ratios of the length spectra for rational maps.
Motivated by the theory of Teichm\"uller space and Kleinian groups, it is more natural to discuss length spectra for rational maps in a single hyperbolic component $\HC$ (cf. \cite{McMS98}), although most of the discussions also work without this assumption.

\subsection*{Markings and the length spectra.}
A conjugacy class of a rational map $[f]$ is called {\em hyperbolic} if the orbit of every critical point converges to some attracting periodic cycle. 
The space of hyperbolic rational maps is open in the moduli space $\M_d = \Rat_d/\PSL_2(\C)$, and a connected component $\HC$ is called a {\em hyperbolic component}.
For each hyperbolic component $\HC$, there is a topological model
$$
\sigma: J \longrightarrow J
$$
for the dynamics on the Julia set.
That is, for any $[f]\in \HC$, there is a homeomorphism
$$
\phi(f): J \longrightarrow J(f)
$$
which conjugates $\sigma$ and $f$.
A particular choice of such $\phi(f)$ will be called a {\em marking} of the Julia set.

Let $[f]\in \HC\subseteq \M_d$ be a hyperbolic rational map with a marking $\phi: J \longrightarrow J(f)$.
Let $\mathscr{S}$ be the space of periodic cycles of
$\sigma:J\longrightarrow J$.
We define the length of a periodic cycle $C\in \mathscr{S}$ for $[f]$ by
$$
L(C, [f]) = \log |(f^q)'(z)|,
$$
where $q = |C|$ and $z \in \phi(C)$. 
The collection $(L(C, [f]): C\in \mathscr{S}) \in \R_{+}^\mathscr{S}$ will be called the {\em marked length spectrum} of $[f]$.
As $[f]$ varies over the hyperbolic component, we aim to understand how the length spectrum changes.
In particular, we will investigate the behavior of the length spectrum for a degenerating sequence $[f_n]$ via the limiting dynamics $F$.

\subsection*{The ends of a tree and translation lengths}
Let $\alpha$ be an end of the tree ${^r\Hyp^3}$. The {\em translation length} of an end $\alpha$ measures the rate at which $F$ moves points of ${^r\Hyp^3}$ towards $x^0$; it is defined by
$$
L(\alpha, F) = \lim_{x_i \to \alpha} d(x_i, x^0) - d(F (x_i), x^0).
$$

Let $[f_n]\in \HC$ be a degenerating sequence with markings $\phi_n$, and
$$
r([f_n]) := \inf_{x\in \Hyp^3} \max_{y\in \E f_n^{-1}(x)} d_{\Hyp^3} (x, y).
$$
Then $r([f_n]) \to \infty$ (see Proposition \ref{LossOfMassC}), and we choose a sequence of representatives $f_n$ with
$$ 
\max_{y\in \E f_n^{-1}(\bm 0)} d_{\Hyp^3} (\bm 0, y) \leq r([f_n]) + 1.
$$
Let $F$ be the limiting map on ${^r\Hyp^3}$. 
The markings $\phi_n$ give a marking $\phi$ on the ends of ${^r\Hyp^3}$  (see \S \ref{MandL}).
If $C\in \mathscr{S}$ is a periodic cycle, then $\phi(C) = \{\alpha_1,..., \alpha_q\}$ is a cycle of periodic ends.
We define its {\em translation length}
$$
L(C, F) = \sum_{i=1}^q L(\alpha_i, F).
$$

It is a well-known principle that the translation lengths on the limiting $\R$-trees give information on the length spectrum (cf. \cite{DeM08, L19, McM09, MorganShalen84}). The precise statement that we will prove and use is the following:

%%%%
%Theorem 1.5
%%%%
\begin{theorem}\label{TL}
Let $[f_n]\in \HC$ be degenerating with markings $\phi_n$, and let $F$ be the associated limiting maps on ${^r\Hyp^3}$.
If $C\in \mathscr{S}$ is a periodic cycle, then
$$
L(C, F) = \lim_\omega L(C, [f_n])/r([f_n]).
$$
Moreover, every cycle of repelling periodic ends of $F$ is represented by some cycle $C \in \mathscr{S}$.
\end{theorem}

\subsection*{Degeneracy with bounded length spectra.}
Let $\HC$ be a hyperbolic component.
It is natural to ask whether the marked length spectrum is a proper map on $\HC$.
In other words, is there a degenerating sequence $[f_n]\in \HC$ with $L(C, [f_n])$ stays bounded for {\em every} $C\in \mathscr{S}$?
Hence, we define

\begin{defn}
A hyperbolic component $\HC$ is said to admit {\em bounded escape} if there exists a sequence $[f_n] \in \HC$ with markings $\phi_n$ so that
\begin{enumerate}
\item $[f_n]$ is degenerating;
\item For any periodic cycle $C\in \mathscr{S}$, the sequence of lengths $L(C, [f_n])$ is bounded.
\end{enumerate}
\end{defn}

Since there are only finitely many periodic points of a fixed period, we can formulate the definition without using the markings and replace the second condition by
\begin{enumerate}
\item[(2')] For any $p\in \N$ and any sequence of periodic points $x_n$ of $f_n$ with period $p$, the multipliers of $f_n$ at $x_n$ stay bounded.
\end{enumerate}

We remark that the analogue of such a phenomenon {\em cannot} happen for degenerating Kleinian groups (cf. \cite{MorganShalen84}), so it might already be surprising that such a hyperbolic component exists.

The first example comes from the McMullen family $f_n(z) = z^2+\frac{1}{n z^3}$ \cite{McM88}.
The sequence $f_n$ is contained in a hyperbolic component for all large $n$, and the Julia set $J$ for this hyperbolic component is homeomorphic to a Cantor set of circles.
In particular, any component of the Julia set separates the two points $0$, $\infty$, and the Julia set is disconnected.
We will show these two topological characteristics of the Julia set classify all examples admitting bounded escape.

\begin{defn}
Let $f\in \Rat_d(\C)$ be a hyperbolic rational map. We say $J(f)$ is {\em nested} if
\begin{enumerate}
\item There are two points $p_1, p_2\in \hat\C$ such that any component of $J(f)$ separates $p_1$ and $p_2$;
\item $J(f)$ contains more than one component.
\end{enumerate}
A hyperbolic component $\HC$ is said to have nested Julia sets if the Julia set of any rational map in $\HC$ is nested.
\end{defn}

%%%%
%Theorem 1.8
%%%%
\begin{theorem}\label{CantorCircle}
A hyperbolic component $\HC$ admits bounded escape if and only if it has nested Julia sets.
\end{theorem}

%%%%
%Figure 1.1
%%%%
\begin{figure}[h]
\centering
\includegraphics[width=5cm]{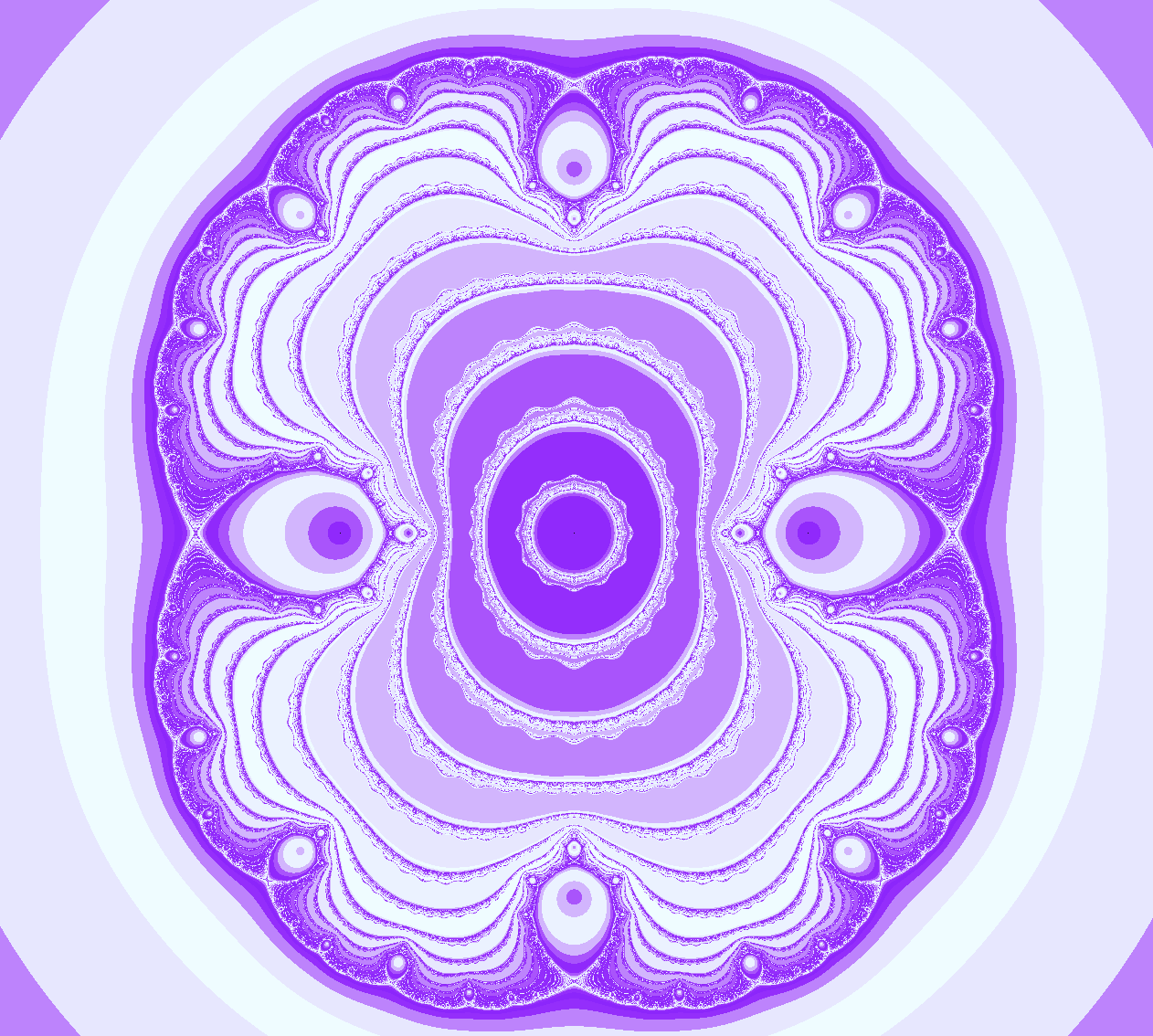}
\includegraphics[width=4.72cm]{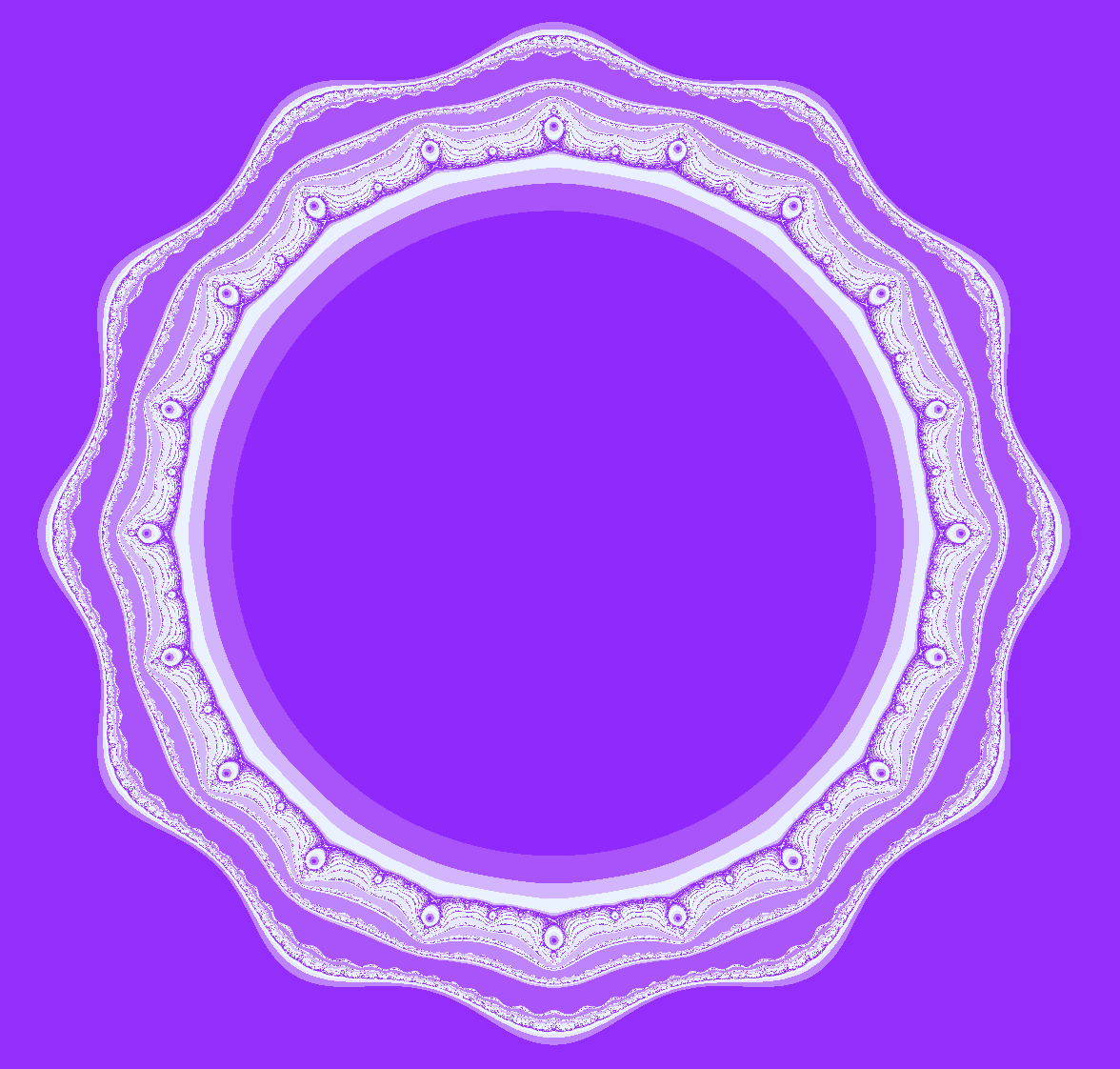}
\caption{The Julia set of $z^2/(1-z^2) + p/z^{10}$ with $p=10^{-7}$ on the left, and a zoom of the Julia set near $0$ on the right. The Julia set is a Cantor set of closed curves. Any `buried' closed curve is a circle. Any boundary component of the `gaps' is a covering of the Julia set of $z^2-1$ (which is conjugate via $z\mapsto 1/z$ to $z^2/(1-z^2)$).}
\label{NestedJuliaSetExample}
\end{figure}

\begin{figure}[h]
\centering
\includegraphics[width=10cm]{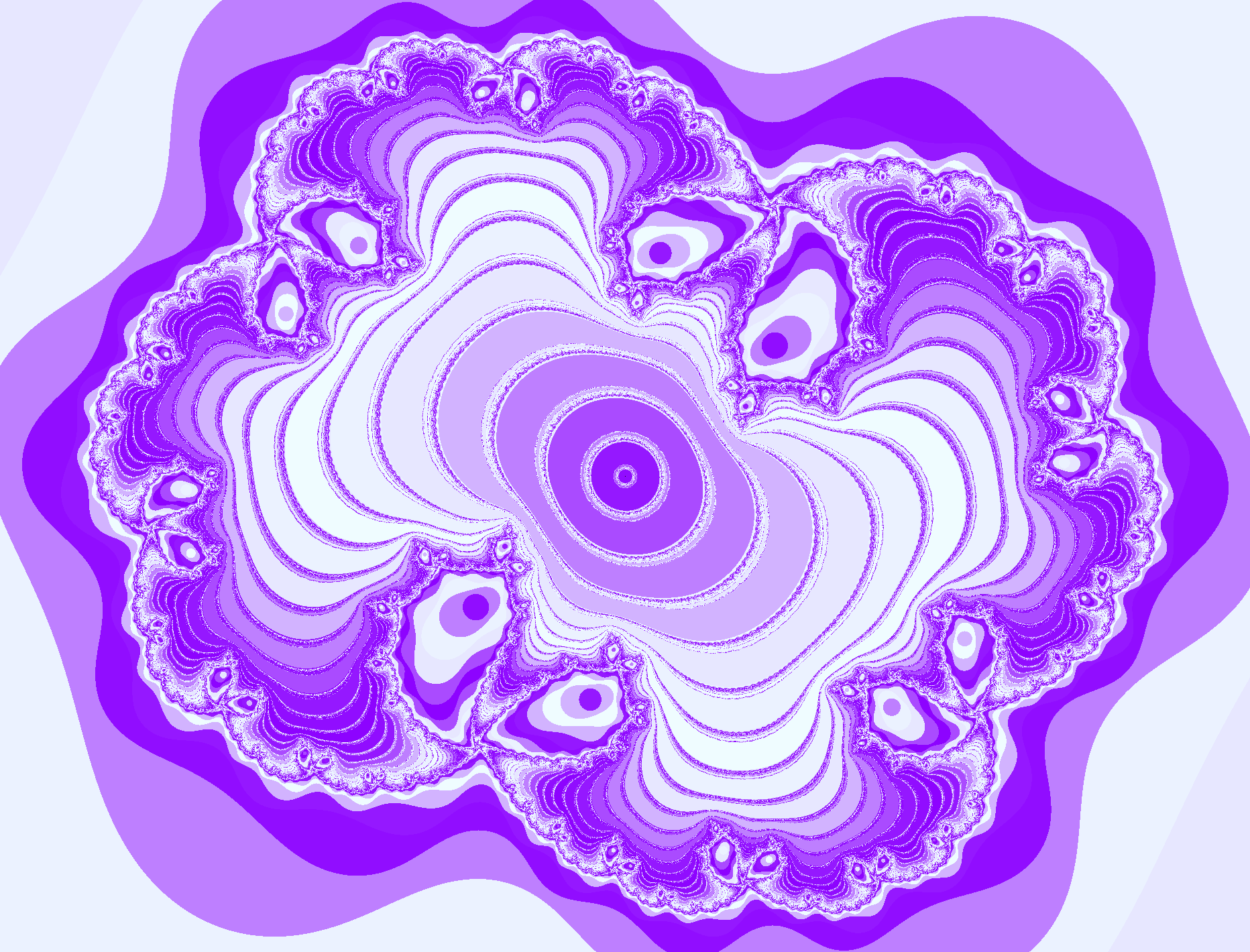}
\caption{The Julia set of $z^2/(1+cz^2) + p/z^{10}$ with $p=10^{-7}$ and $c$ in the `rabbit' component of the Mandelbrot set. Each Julia component is either a circle or a covering of the Julia set of the quadratic polynomial $z^2+c$.}
\label{NestedJuliaSetExample2}
\end{figure}

In \S \ref{HypComNestJ}, we associate a combinatorial invariant to a hyperbolic rational map with nested Julia sets (see Proposition \ref{NestJuliaSetInvariants}).
We also show any such combinatorial data is realizable (see Theorem \ref{NestedMatingPoly}) (cf. \cite{L20}).
Using this, we prove nested Julia sets implies bounded escape.

To prove the other direction, from the definition, we have a degenerating sequence $[f_n]\in \HC$ with bounded multipliers. This gives a limiting dynamics on an $\R$-tree with no repelling periodic ends. In \S \ref{ProofOfCantorCircle}, we classify such limiting dynamics and use our classification to derive topological properties of the Julia set.
In the course of the proof, we also get a quantitative result: 
if $\HC$ is not nested, then there exists some periodic cycle whose multiplier escapes to infinity comparable to $r([f_n])$ (see Theorem \ref{LimitingLengthSpectrum}).
We refer the readers to the discussion at the end of \S \ref{ProofOfCantorCircle} for open questions.

\subsection*{Comparison with Kleinian groups}
We conclude our discussion by comparing our results above with some well-known results for Kleinian groups.
\begin{enumerate}
\item {\em Hyperbolic components and markings.} A hyperbolic component of rational maps is an analogue of $AH(M)$, the space of conjugacy classes of discrete faithful representations of $\pi_1(M)$ in $\PSL_2(\C)$.
The repelling periodic cycles are in correspondence with the closed geodesics on the quotient hyperbolic 3-manifold.
For Kleinian groups, closed geodesics are naturally marked by the representation.
The markings of periodic cycles are, in general, not canonical.

\item {\em Limiting dynamics on $\R$-trees.} In \cite{MorganShalen84}, based on the study of valuations on the function field of the character variety, Morgan and Shalen showed how to compactify $AH(M)$ using isometric actions on $\R$-trees, which shed new lights and generalized part of Thurston's Hyperbolization Theorem. 
Bestvina and Paulin gave a more geometric perspective of this theory in \cite{Bestvina88, Paulin88}.
In the rational map setting, the construction of limiting map on $\R$-trees using barycentric extensions is in the same spirit as Bestvina and Paulin's construction, while the Berkovich dynamics, which studies valuations on polynomial rings, adapts the algebraic perspective as in Morgan and Shalen's formulation.
Theorem \ref{EquivTheorem} is analogous to the equivalence of various constructions of $\R$-trees for Kleinian groups.
We remark that the construction does not require the sequence to come from a single hyperbolic component in the rational map setting.

\item {\em Limiting ratios of length spectra.} The limiting ratios of lengths of marked geodesics for degenerating sequence of Kleinian groups are naturally recorded in the limiting isometric action on the $\R$-tree (see \cite{MorganShalen84, O96}).
Theorem \ref{TL} gives analogous results in complex dynamics.

\item {\em Bounded escape and nested Julia sets.} These phenomena are new in the complex dynamics setting:
\begin{itemize}
\item A sequence in $AH(M)$ is degenerate if and only if the length of some closed geodesic escapes to infinity (see \cite{MorganShalen84}).
\item Any {\em buried component} of the limit set of a finitely generated Kleinian group is a single point, while if the Julia set of a rational map is nested, it has uncountably many buried components which are non-degenerate (cf. \cite{McM88, AM77}).
\end{itemize}

\item {\em Blaschke products and Fuchsian groups/Teichm\"uller spaces.} Under our dictionary, the space of Blaschke products is analogous to the space of Fuchsian representations. Our construction of dynamics on $\R$-trees is a direct generalization of the Ribbon $\R$-tree construction in \cite{McM09} for Blaschke products. We also refer to \cite{McM08, McM09b, McM10} for more comparisons between Blaschke products and many other aspects of the Teichm\"uller spaces.
\end{enumerate}

\subsection*{Notes and references}
For more on $\R$-trees, degenerations of hyperbolic manifolds and rational maps, see e.g. \cite{MorganShalen84, Bestvina88, Bestvina01, Paulin88, O96, McM09}.
The use of asymptotic cone and the connection of $\R$-trees with the nonstandard analysis are developed and explained in \cite{KapovichLeeb95, Chiswell91}.
Other application of trees in complex dynamics can be found in \cite{Shishikura89, DeM08}.

The use of Berkovich space of formal Puiseux series and rescalings to understand asymptotic behaviors for a degenerating holomorphic family of rational maps was introduced and made precise in \cite{Kiwi15}.
Similar ideas have also been explored in \cite{Stimson93, Epstein00, DeM07, Arfeux17}.
In many situations, the study of parameter spaces leads us to consider sequences of rational maps.
Our constructions of limiting maps on $\R$-trees, which use sequences, are usually better suited to answer such questions.
For example, our theory easily gives a direct and uniform proof of the sequential version (compared to holomorphic family version) of the rescaling limit theorem in \cite{Kiwi15}, and also allows us to eliminate the smoothness assumptions of the boundary when studying hyperbolic components in \cite{NP18}.
Other applications of Robinson's field can be found in \cite{deFM09}.

Many examples of degenerating families of rational maps with bounded multipliers (Latt\`es family) without the hyperbolicity assumption are studied in \cite{FR10} using Berkovich dynamics.
Rational maps with disconnected Julia sets are studied extensively in \cite{PilgrimTan00}, and some examples of rational maps with nested Julia sets also appear there.

\subsection*{Acknowledgments}
The author thanks C. T. McMullen for his advice and helpful discussion on this problem.
The author gratefully thanks the anonymous referees for valuable comments and suggestions.

\section{$\R$-trees and branched coverings}\label{BCR}
In this section, we give a brief introduction of $\R$-trees and set up some notations.

\subsection*{$\R$-trees.}
An $\R$-tree is a nonempty metric space $(T,d)$ such that any two points $x, y\in T$ are connected by a unique topological arc $[x,y]\subseteq T$, and every arc of $T$ is isometric to an interval in $\R$.

We say $x$ is an endpoint of $T$ if $T-\{x\}$ is connected; otherwise $x$ is an interior point. If $T - \{x\}$ has three or more components, we say $x$ is a branch point. 
The set of branch points will be denoted $B(T)$. 
We say $T$ is a {\em finite tree} if $B(T)$ is finite.
We will write $[x, y)$ and $(x, y)$ for $[x, y]$ with one or both of its endpoints removed.

\subsection*{Ends and tangent vectors}
A ray $\alpha$ in the $\R$-tree $T$ is an isometric embedding of $[0, \infty)$.
We identify a ray as a map and its image in $T$.
Two rays are said to be {\em equivalent} if $\alpha_1 \cap \alpha_2$ is still a ray.
The collection $\epsilon(T)$ of all equivalence classes of rays forms the set of {\em ends} of $T$.

We will use $\alpha$ to denote both a ray and the end it represents.
We say a sequence of points $x_i$ converges to an end $\alpha$, denoted by $x_i\to \alpha$, if for all $\beta \sim \alpha$, $x_i \in \beta$ for all sufficiently large $i$.

Let $x\in T$. Two segments $[x,y_1]$ and $[x, y_2]$ are said to represent the same tangent vector at $x$ if $[x,y_1] \cap [x,y_2]$ is non-degenerate segment.
The set of equivalence classes of tangent vectors at $x$ is called the {\em tangent space} at $x$, and denoted by $T_xT$.
Equivalently, the tangent space $T_xT$ can be identified with the set of components of $T-\{x\}$.

Let $\vec{v} \in T_xT$. We will use $U_{\vec{v}}$ to denote the component of $T-\{x\}$ corresponding to $\vec{v}$.
More generally, if $x, y \in T$, we use $U^{x,y}$ to denote the component of $T-\{x,y\}$ with boundary $x, y$.

\subsection*{Convexity and subtrees.}
A subset $S$ of $T$ is called {\em convex} if $x, y\in S \implies [x,y]\subset S$.
The smallest convex set containing $E\subset T$ is called the {\em convex hull} of $E$, and is denoted by $\hull(E)$.
More generally, we can easily extend the above definitions to subsets of $T\cup \epsilon(T)$.

Note that subset $S\subset T$ is convex $\iff$ $S$ is connected $\iff$ $S$ is a subtree.
Moreover, $S$ is a {\em finite subtree} of $T$ $\iff$ $S$ is the convex hull of a finite set $E\subset T\cup \epsilon(T)$.

\subsection*{Branched coverings between $\R$-trees.}
Let $f:T_1 \longrightarrow T_2$ be a continuous map between two $\R$-trees, 
we say $f$ is a (tame) degree $d$ branched covering if 
there is a finite subtree $S\subseteq T_1$ such that
\begin{enumerate}
\item $S$ is nowhere dense in $T_1$, and $f(S)$ is nowhere dense in $T_2$.
\item For every $y\in T_2 - f(S)$, there are exactly $d$ preimages in $T_1$.
\item For every $x\in T_1 - S$, $f$ is a local isometry.
\item For every $x\in S$, and any sufficiently small neighborhood $U$ of $f(x)$, $f: V-f^{-1}(f(V\cap S)) \longrightarrow U-f(V\cap S)$ is an isometric covering, where $V$ is the component of $f^{-1}(U)$ containing $x$.
\end{enumerate}

%%%%%%%%%%
%% Section 2
%% %%%%%%%%

\section{Barycentric extensions for rational maps}\label{LimitingMapViaBarycentricExt}
In this section, we briefly review the theory of barycentric extensions for rational maps. 
We will summarize some properties and refer the readers to \cite{L19} for details.

The barycentric extension was extensively studied for circle homeomorphisms in \cite{DouadyEarle86}. 
The construction can be easily generalized to any continuous maps on the unit sphere $S^{n-1}$, (see \cite{L19, Petersen11}). For our purposes, we will focus on the barycentric extensions for rational maps on $\hat\C$, which is identified with $S^2$ under the stereographic projection.

We identify the hyperbolic space $\Hyp^3$ with the ball model $B(\bm 0,1)\subseteq \R^3$.
The conformal boundary of $\Hyp^3$ is naturally identified with $S^2$ in this way.
A measure $\mu$ on $S^2$ is said to be {\em balanced} at a point $y$ if 
$$
\int_{S^2} \zeta \, dM_*\mu(\zeta) = \vec{0} \in \R^3,
$$
where $M \in \PSL_2(\C) = \Isom(\Hyp^3)$ is any isometry sending $y$ to $\bm 0$.

Given a probability measure $\mu$ on $S^2$ with no atoms of mass $\geq 1/2$, there is a unique point $\beta(\mu) \in \Hyp^3$, called the {\em barycenter} of $\mu$, for which the measure is {\em balanced} (see \cite{DouadyEarle86, Hubbard06} or \cite{Petersen11}). 

Let $\mu_{S^2}$ be the probability measure coming from the spherical metric on $S^2$, 
and let $f : S^2 \longrightarrow S^2$ be a non-constant rational map. The {\em barycentric extension} 
$$
\E f: \Hyp^3 \longrightarrow \Hyp^3
$$ 
is a map sending the point $x\in \Hyp^3$ to the barycenter of the measure $f_*(M_{x})_*(\mu_{S^2})$, where $M_{x}$ is any isometry sending $\bm {0}$ to $x$.

The extension $\E(f)$ is {\em conformally natural}: if $M_1, M_2 \in \PSL_2(\C)$, then
$$
M_1 \circ \E(f) \circ M_2 = \E (M_1 \circ f \circ M_2).
$$

The following theorem concerning the regularities of the barycentric extensions is proved in \cite{L19} (see Theorem 1.1 in \cite{L19}):
\begin{theorem}\label{RationalLip}
For any rational map $f: \hat\C \longrightarrow \hat\C$ of degree $d$, the norm of the derivative of its barycentric extension $\E f: \Hyp^3 \longrightarrow \Hyp^3$ satisfies
$$
\sup_{x\in \Hyp^3} \| D \E f(x)\| \leq C \deg(f).
$$
Here the norm is computed with respect to the hyperbolic metric, and $C$ is a universal constant.
\end{theorem}

\subsection*{The space of rational maps}
The space $\Rat_d(\C)$ of rational maps of degree $d$ is an open variety of $\Proj^{2d+1}_\C$.
More concretely, fixing a coordinate system of $\Proj^1_\C$, then a rational can be expressed as a ratio of homogeneous polynomials $f(z:w) = (P(z,w): Q(z,w))$, where $P$ and $Q$ have degree $d$ with no common divisors. Using the coefficients of $P$ and $Q$ as parameters, then
$$
\Rat_d(\C) = \Proj^{2d+1}_\C - V(\Res)
$$
where $V(\Res)$ is the vanishing locus for the resultant of $P$ and $Q$.

The space $\Rat_d(\C)$ is not compact. 
A sequence $f_n\in\Rat_d(\C)$ is said to be {\em degenerating}, denoted by $f_n \to \infty$, if $f_n$ escapes every compact set of $\Rat_d(\C)$.
One natural compactification of $\Rat_d(\C)$ is $\overline{\Rat_d(\C)} = \Proj^{2d+1}_\C$. 

Every map in $f\in \overline{\Rat_d(\C)}$ determines the coefficients of a pair of homogeneous polynomials, and we write
$$
f= (P: Q) = (Hp:Hq) = H\varphi_f,
$$
where $H = \gcd (P, Q)$ and $\varphi_f = (p:q)$ is a rational map of degree at most $d$.
A zero of $H$ is called a {\em hole} of $f$ and the set of zeros of $H$ is denoted by
$\mathcal{H}(f)$.
We define the degree of $f\in \overline{\Rat_d(\C)}$ as the degree of $\varphi_f$.

The limit map describes the limiting dynamics away from the holes $\mathcal{H}(f)$ (see Lemma 4.2 in \cite{DeM05}):
\begin{lem}\label{comcon}
Let $f_n\in\Rat_d(\C)$ converge to $f \in \overline{\Rat_d(\C)}$. Then $f_n$ converges compactly to $\varphi_f$ on $\hat\C - \mathcal{H}(f)$.
\end{lem}

Note that if $f_n$ converges to a degree $0$ map, the measure $(f_n)_*\mu_{S^2}$ converges weakly to a delta measure. So the image of $\bm 0$
$$
\E f_n (\bm 0) \to \infty.
$$
On the other hand, if $f_n$ converges to $f = H\varphi_f$ with $\deg(\varphi_f)\geq 1$, $(f_n)_*\mu_{S^2}$ converges weakly to $(\varphi_f)_*\mu_{S^2}$, so
$$
\E f_n (\bm 0) \to \E \varphi_f (\bm 0).
$$
Therefore, we have the following proposition (see Proposition 5.5 in \cite{L19}):
\begin{prop}\label{bdd}
Let $f_n \in \Rat_d(\C)$. Then
$\E f_n (\bm 0)$ is uniformly bounded if and only if degree $0$ maps are not in the limit set of $\{f_n\}$ in $\overline{\Rat_d(\C)}$.
\end{prop}

The following proposition gives a criterion for degeneracy (see Proposition 8.1 in \cite{L19}):
\begin{prop}\label{LossOfMass}
Let $f_n\in \Rat_d(\C)$, and 
$r_n:=\max_{y\in \E f_n^{-1}(\bm 0)} d_{\Hyp^3}(\bm 0, y)$. 
Then $f_n$ is degenerating if and only if $r_n \to \infty$.
\end{prop}

\subsection*{The moduli space of rational maps.}
The group of M\"obius transformations $\PSL_2(\C)$ acts on $\Rat_d(\C)$ by conjugation. 
The quotient space $\M_d$ is known as the {\em moduli space} of rational maps.
A sequence $[f_n]\in \M_d$ is said to be {\em degenerating}, denoted by $[f_n] \to\infty$, if $[f_n]$ escapes every compact set of $\M_d$.
Equivalently, $[f_n] \to\infty$ if and only if {\em every} sequence of representatives $f_n \in \Rat_d$ is degenerating.

Since choosing a base point $\bm 0\in \Hyp^3$ is equivalent to choosing a representative of the conjugacy class $[f]$ up to the compact group $SO(3)$, we define
$$
r([f]) := \inf_{x\in \Hyp^3} \max_{y\in \E f^{-1}(x)} d_{\Hyp^3}(x, y).
$$

The following proposition follows immediately from the definition and Proposition \ref{LossOfMass}  (see also Proposition 8.10 in \cite{L19}):
\begin{prop}\label{LossOfMassC}
The sequence $[f_n] \to \infty$ in $\M_d$ if and only if $r([f_n]) \to \infty$.
\end{prop}
Note that it is possible that $r([f]) = 0$ as in the example of $f(z) = z^2$ (see Appendix A in \cite{L19}).
For degenerating sequences $[f_n]$, Proposition \ref{LossOfMassC} in particular gives that $r([f_n]) \neq 0$ for all large $n$.

\section{Ultralimits and asymptotic cones}\label{UltralimitAC}
In this section, we will review a standard construction of ultralimits for sequences of pointed metric spaces.
The construction uses a non-principal ultrafilter, which is an efficient technical device for simultaneously taking limits of all sequences without passing to subsequences and putting them together to form one object.
We refer the readers to \cite{Gromov92, KapovichLeeb95, Roe03} for more details.

\subsection*{Ultrafilters on $\N$.}
We begin by reviewing the theory of ultrafilter on $\N$. 
A subset $\omega\subset \powerset(\N)$ of the power set of $\N$ is called an ultrafilter if
\begin{enumerate}
\item If $A, B\in \omega$, then $A\cap B\in \omega$;
\item If $A\in \omega$ and $A\subset B$, then $B\in \omega$;
\item $\emptyset\notin \omega$;
\item If $A\subset \N$, then either $A \in \omega$ or $\N-A\in \omega$
\end{enumerate}

By virtue of the above $4$ properties, one can think of an ultrafilter $\omega$ as defining a {\em finitely additive $\{0,1\}$-valued probability measure} on $\N$. The sets of measure $1$ are precisely those belonging to the ultrafilter $\omega$. 

We will call a set in $\omega$ as {\em $\omega$-big} or simply {\em big}. Its complement is called {\em $\omega$-small} or simply {\em small}.
If a specific property is satisfied by a $\omega$-big set, we will also say this property holds $\omega$-almost surely.

Let $a\in \N$. We can construct an ultrafilter by
$$
\omega_a:=\{A\subset \powerset(\N): a\in A\}.
$$
Any ultrafilter of the above type is called a {\em principal ultrafilter}. 
It can be verified that an ultrafilter is principal if and only if it contains a finite set.
An ultrafilter that is not principal is called a {\em non-principal ultrafilter}.
The existence of a non-principal ultrafilter is guaranteed by Zorn's lemma.

Let $\omega$ be a non-principal ultrafilter on $\N$.
Let $x_n$ be a sequence in a metric space $(X,d)$ and $x\in X$.
We say $x$ is the {\em $\omega$-limit} of $x_n$, denoted by 
$$
\lim_\omega x_n = x
$$
if for every $\epsilon>0$, the set $\{n: d(x_n, x) < \epsilon\}$ is $\omega$-big.

It can be easily verified (see \cite{KapovichLeeb95}) that 
\begin{enumerate}
\item If the $\omega$-limit exists, then it is unique.
\item If $x_n$ is contained in a compact set, then the $\omega$-limit exists.
\item If $x = \lim_{n\to\infty} x_n$ in the standard sense, then $x = \lim_{\omega} x_n$.
\item If $x = \lim_{\omega} x_n$, then there exists a subsequence $n_k$ such that $x = \lim_{k\to\infty} x_{n_k}$ in the standard sense.
\end{enumerate}

From these properties, one should think of the non-principal ultrafilter $\omega$ as performing all the subsequence-selection in advance. All sequences in compact spaces will automatically converge without the need to pass to any further subsequences.

From now on and throughout the rest of the paper, we will fix a non-principal ultrafilter $\omega$ on $\N$.

\subsection*{Ultralimit of pointed metric spaces.}
Let $(X_n, p_n, d_n)$ be a sequence of pointed metric spaces with basepoints $p_n$. Let $\mathcal{X}$ denote the set of sequences $\{x_n\}, x_n \in X_n$ such that $d_n(x_n, p_n)$ is a bounded function of $n$. We also define an equivalence relation $\sim$ by
$$
(x_n) \sim (y_n) \Leftrightarrow \lim_\omega d_n(x_n, y_n) = 0.
$$
We use $[(x_n)]$ to denote the equivalence class of the sequence $(x_n)$.

Let $X_\omega = \mathcal{X} / \sim$, and we define
$$
d_\omega ([(x_n)], [(y_n)]) = \lim_\omega d_n(x_n, y_n).
$$

The function $d_\omega$ makes $X_\omega$ a metric space, and is called the {\em ultralimit} of $(X_n, p_n, d_n)$ with respect to the ultrafilter $\omega$, and is written as $\lim_\omega (X_n, p_n, d_n)$ or simply $\lim_\omega X_n$ for short.

\subsection*{Asymptotic cones of $\Hyp^3$.}
Given a positive sequence $r_n\to \infty$, we define the {\em asymptotic cone} 
$({^r\Hyp^3}, x^0, \dist)$ with rescaling $r_n$ as the ultralimit
$$
({^r\Hyp^3}, x^0, \dist) = \lim_\omega (\Hyp^3, \bm 0, \frac{1}{r_n}\dist_{\Hyp^3}).
$$
From the definition, a point $x\in {^r\Hyp^3}$ is represented by a sequence $x_n \in \Hyp^3$ with $\lim_\omega \frac{\dist_{\Hyp^3}(\bm 0, x_n)}{r_n} < \infty$, and $[(x_n)] = [(y_n)]\iff \lim_\omega \frac{\dist_{\Hyp^3}(x_n, y_n)}{r_n} = 0$.

It is well known that ${^r\Hyp^3}$ is an $\R$-tree that has uncountably many branches at every point (see \cite{Roe03,KapovichLeeb95}).

\subsection*{Ends of ${^r\Hyp^3}$.}
Let $z\in \hat\C$, viewed as a point on the conformal boundary of $\Hyp^3$. 
We denote $\gamma(t, z)\in \Hyp^3$ as the point of distance $t$ from $\bm 0$ in the direction of $z$.
Given any sequence $z_n \in \hat\C$, the ray
$$
s(t) = [(\gamma(t\cdot r_n, z_n))]
$$
is a geodesic ray parameterized by arc length in ${^r\Hyp^3}$.
Thus, $(z_n)$ corresponds to an end $\alpha \in \epsilon({^r\Hyp^3})$. We denote this by
$z_n \twoheadrightarrow_\omega \alpha$.

\subsection*{Limiting map on ${^r\Hyp^3}$.}
Let $f_n\in \Rat_d(\Com)$ be a degenerating sequence of rational maps. 
The appropriate rescaling to use is
$$
r_n := \max_{y\in \E f_n^{-1}(\bm{0})} d_{\Hyp^3}(\bm{0}, y),
$$
as it brings all the preimages of $\bm 0$ in view.
Let $({^r\Hyp^3}, x^0, d)$ be the asymptotic cone of $(\Hyp^3,\bm 0)$ with respect to the rescaling $r_n$.

The limiting map $F = \lim_\omega \E f_n : {^r\Hyp^3} \longrightarrow {^r\Hyp^3}$ is defined by
$$
F ([(x_n)]) = [(\E f_n (x_n))].
$$

Since all $\E f_n$ are $Cd$-Lipschitz by Theorem \ref{RationalLip} with a universal constant $C$, it follows that $F$ is well defined. 
In \cite{L19}, we showed that the limiting map is a degree $d$ branched covering (see Theorem 1.2 in \cite{L19}):
\begin{theorem}\label{DGL}
Let $f_n \to \infty$ in $\Rat_d(\C)$, with $r_n = \max_{y\in \E f_n^{-1}(\bm{0})} \dist_{\Hyp^3}(\bm{0}, y)$.
Then the limiting map 
$$
F = \lim_\omega \E f_n : {^r\Hyp^3} \longrightarrow {^r\Hyp^3}
$$ 
is a degree $d$ branched covering of the $\R$-tree ${^r\Hyp^3}$.
\end{theorem}

\subsection*{Degenerating sequences in $\M_d$.}
Let $[f_n] \to \infty$ in $\M_d$.
Recall that by Proposition \ref{LossOfMassC}, we have
$$
r([f_n]) := \inf_{x\in \Hyp^3}\max_{y\in \E f_n^{-1} (x)} d_{\Hyp^3} (x, y) \to \infty.
$$
We can choose representatives $f_n$ so that
$$
\max_{y\in \E f_n^{-1} (\bm 0)} d_{\Hyp^3} (\bm 0, y) \leq r([f_n]) + 1.
$$
The following theorem is the version for $\M_d$ (see Theorem 8.11 in \cite{L19}):

\begin{theorem}\label{DGLC}
Let $[f_n] \to \infty$ in $\M_d$ and $r_n:= r([f_n])$.
Let $f_n$ be a representative of $[f_n]$ with
$$
\max_{y\in \E f_n^{-1}(\bm{0})} \dist_{\Hyp^3}(\bm{0}, y) \leq r_n+1.
$$
Then the corresponding limiting map
$$
F = \lim_\omega \E f_n: {^r\Hyp^3} \longrightarrow {^r\Hyp^3}
$$
is a branched covering of degree $d$ with no totally invariant point.
\end{theorem}

%%%%%%%%%%
%% Section 8
%% %%%%%%%%
\section{Rational maps on Berkovich projective space $\Proj^1_{Berk}$}\label{LimitingMapViaBerk}
In this section, we give a brief review of the Berkovich projective space  $\Proj^1_{Berk}$ for a complete, algebraically closed non-Archimedean field $K$, and the dynamics of rational maps on it.
We refer the readers to \cite{BakerRumely10} for a more detailed exposition of this theory.

Let $K$ be a complete, algebraically closed non-Archimedean field. 
We use
$B(a,r) := \{z\in K: |z-a| \leq r\}$ and $B(a,r)^- :=\{z\in K: |z-a| < r\}$ to denote the {\em closed ball} and {\em open ball} centered at $a$ with radius $r$ respectively.
Recall that in a non-Archimedean field, any point $z\in B(a,r)$ (or $z\in B(a,r)^-$) is the center.
If two balls intersect, then one is contained in the other.

The valuation ring of $K$ will be denoted as $\mathfrak{D}_K = B(0, 1)$,
and its maximal ideal is $\mathfrak{M}_K = B(0, 1)^-$.
The residual field is $\widetilde{K} = \mathfrak{D}_K / \mathfrak{M}_K$.

Let $f\in\Rat_d(K)$ be a rational map with coefficients in $K$. After multiplying the denominator and numerator by a common factor, we may assume that the maximum norm of the coefficients is $1$.
The reduction map $\tilde{f}$ is given by taking the reduction on its coefficients.

\subsection*{The Berkovich affine space and the Berkovich projective space}
As a topological space, $\Aff^1_{Berk}$ can be defined as follows.
The underlying point set is the collection of all the multiplicative seminorms $[\,]_x$ on the polynomial ring $K[T]$ which extend the absolute value on $K$.
The topology on $\Aff^1_{Berk}$ is the weakest one for which $x\to [f]_x$ is continuous for all $f\in K[T]$.
The field $K$ can be thought of as a subspace of $\Aff^1_{Berk}$, via the evaluation map.
That is, we can associate to a point $x\in K$ the seminorm
$$
[f]_x = |f(x)|.
$$
The seminorms of this form will be called {\em classical points}.

The Berkovich projective space $\Proj^1_{Berk}$ is the one point compactification of $\Aff^1_{Berk}$.
The extra point, which is denoted, as usual, by $\infty$, can be regarded as the point $\infty\in\Proj^1_K$ embedded in $\Proj^1_{Berk}$.

\subsection*{Berkovich classification.}
Given closed ball $B(a, r)$, one can construct the {\em supremum norm}
$$
[f]|_{B(a, r)} = \sup_{z\in B(a, r)}|f(z)|.
$$
One of the miracles of the non-Archimedean universe is that this norm is {\em multiplicative}.
More generally, given any decreasing sequence of closed balls $x = \{B(a_i, r_i)\}$, we can consider the limiting seminorm
$$
[f]_x = \lim_{i\to\infty} [f]_{B(a_i, r_i)}.
$$
Berkovich's classification asserts that every point $x\in\Aff^1_{Berk}$ arises in this way, and we can classify them into $4$ types:
\begin{enumerate}
\item Type I: Points in $\Aff^1_K$, which we will also call the {\em classical points};
\item Type II: Points corresponding to a closed ball $B(a, r)$ with $r\in |K^\times|$;
\item Type III: Points corresponding to a closed ball $B(a, r)$ with $r\notin |K^\times|$;
\item Type IV: Points corresponding to a nested sequence $\{B(a_i, r_i)\}$ with empty 
intersection.
\end{enumerate}

Type I, II and III can all be thought of a special case of Type IV:
the classical points correspond to a nested sequence $\{B(a_i, r_i)\}$ with $\lim r_i = 0$;
the Type II points correspond to a nested sequence $\{B(a_i, r_i)\}$ with nonempty intersection and $r=\lim r_i >0$ belongs to the value group $|K^\times|$;
the Type III points correspond to a nested sequence $\{B(a_i, r_i)\}$ with nonempty intersection but $r=\lim r_i >0$ does not belong to the value group $|K^\times|$.

We will call the point corresponding to $B(0,1)$ the {\em gauss point} and is denoted by $x_g$.

\subsection*{Rational maps on $\Proj^1_{Berk}$}
There is a `Proj' construction of $\Proj^1_{Berk}$, which allows us to extend the natural action of $f\in \Rat_d(K)$ on $\Proj^1_K$ to $\Proj^1_{Berk}$, preserving the types of the points (see \S 2 in \cite{BakerRumely10}).

Given any Type II point $x$, there exists $M\in \PSL_2(K)$ such that $M(x_g) = x$. 
We will regard $M$ as a `change of coordinates'.
For our purposes, we can treat the following proposition as the definition of rational maps on Type II points (see Lemma 2.17 in \cite{BakerRumely10}).
\begin{prop}\label{ReductionGEQ1}
Let $f\in\Rat_d(K)$, and let $x, y\in\Proj^1_{Berk}$ be Type II points. Assume that $x = M(x_g)$ and $y = L(x_g)$ with $M, L \in \PSL_2(K)$. 
Then $f(x) = y$ if and only if $L^{-1} \circ f \circ M$ has non-constant reduction.
\end{prop}

\subsection*{The tree structure on $\Hyp_{Berk}$.}
The {\em Berkovich hyperbolic space} $\Hyp_{Berk}$ is defined by
$$
\Hyp_{Berk} = \Proj^1_{Berk} - \Proj^1_K = \Aff^1_{Berk} - \Aff^1_K
$$
Note that $\Hyp_{Berk}$ is also the space of Type II, III and IV points.

Given two Type II or III points $x, y$ corresponding to the balls $B(a,r)$ and $B(b, s)$ respectively, we let $B(a, R)$ be the smallest ball containing both $B(a,r)$ and $B(b, s)$. Note that $R = \max(r,s, |a-b|)$. 
We define the distance function
$$
d(x,y) = 2\log R - \log r - \log s
$$
Note that if $B(a,r)$ is contained in $B(b, s)$, then 
$$
d(x,y) = \log s - \log r = \log s/r,
$$ 
which should be interpreted roughly as the modulus of the open annulus $B(b,s)^- - B(a,r)$.
In general, the distance is the sum of modulus of $B(a,R)^- - B(a,r)$ and $B(b,R)^- - B(b, s)$.

One can extend this distance formula continuously to arbitrary points $x, y \in \Hyp_{Berk}$.
The metric space $(\Hyp_{Berk}, d)$ can be shown to be a complete $\R$-tree (see Proposition 2.29 in \cite{BakerRumely10}).
Moreover, the finite ends of the $\R$-tree correspond to the Type IV points, while the infinite ends of the $\R$-tree correspond to the classical (Type I) points.
The group $\PSL_2(K)$ acts isometrically on $\Hyp_{Berk}$ which is transitive on Type II points.

We remark that the topology generated by the metric $d$ is strictly finer than the subspace topology of the Berkovich topology on $\Hyp_{Berk}$.
In this paper, we shall mainly use the topology generated by this metric.

\subsection*{Tangent maps and multiplicities.}
The tangent space $T_x\Hyp_{Berk}$ at $x$ is the space of components of $\Hyp_{Berk} - \{x\}$.
If $x=x_g$ is the Gauss point, the tangent space $T_{x_g}\Hyp_{Berk}$ is identified with the projective line of the residual field $\Proj^1_{\widetilde{K}}$.
More generally, if $x$ is a Type II point, we can choose $M\in \PSL_2(K)$ so that $x = M(x_g)$.
The isometry $M$ allows us to identify $T_x\Hyp_{Berk}$ with $\Proj^1_{\widetilde{K}}$.

If $f\in \Rat_d(K)$, then $f$ induces a natural tangent map 
$$
D_x f: T_x \Hyp_{Berk} \longrightarrow T_{f(x)} \Hyp_{Berk}.
$$
The following theorem follows from Corollary 9.25 in \cite{BakerRumely10}:
\begin{theorem}\label{ReductionAsTangentMap}
Let $x\in \Hyp_{Berk}$ be a Type II point, and $y=f(x)$.
Choose $M, L\in \PSL_2(K)$ so that $x = M(x_g)$ and $y = L(x_g)$.
Let $g = L^{-1} \circ f \circ M$, and let $\tilde g$ be the reduction of $g$.
Then under the identification of $T_x\Hyp_{Berk}$ (and $T_y\Hyp_{Berk}$) with $\Proj^1_{\tilde K}$ by $M_*$ (and $L_*$ respectively), the tangent map $D_x f$ equals to the reduction map $\tilde g$.
\end{theorem}

There are many equivalent ways to extend the definition of local degrees for $f$ from $\Proj^1_K$ to $\Proj^1_{Berk}$ in the literature (see \S 2 in \cite{FR10} and Chapter 9 in \cite{BakerRumely10}).
Theorem \ref{ReductionAsTangentMap} allows us to define it for Type II points.
Let $x\in \Hyp_{Berk}$ be a Type II point. 
We define the {\em local degree}
$$
\deg_x f = \deg \tilde{g}.
$$ 
If we now count each point by its multiplicities, then every point has exactly $d$ preimages in $\Proj^1_{Berk}$.

%%%%%%%%%%
%% Section 9
%% %%%%%%%%

\section{Complexified Robinson's field}\label{CRF}
In this section, we study a complete, algebraically closed, non-Archimedean field ${^\rho\Com}$. This field is first introduced in the real setting by Robinson to formulate the non-standard analysis.

Recall that $\omega$ is a fixed non-principal ultrafilter $\N$.
Let $\Com^\N$ be the set of all sequences in $\Com$. We say two sequence $(z_n)$ and $(w_n)$ are equivalent if 
$$
z_n = w_n \;\;\; \omega-\text{almost surely}.
$$
The set of equivalence classes will be denoted by ${^*\Com}$.

The addition and multiplication are defined naturally:
if $x, y \in {^*\Com}$ are represented by $(x_n)$ and $(y_n)$, then $x+y$ and $x\cdot y$ are the class represented by $(x_n+y_n)$ and $(x_n\cdot y_n)$.
It can be checked that these are indeed well defined and make $ {^*\Com}$ a field.
This field is usually referred to as the {\em ultrapower construction} for $\Com$ (cf. Chapter 2 in \cite{LightstoneRobinson75}).

Given $x,y\in {^*\Com}$ represented by $(x_n)$ and $(y_n)$, we write $|x|\leq |y|$ or $|x| < |y|$ if $|x_n|\leq|y_n|$ or $|x_n|<|y_n|$ $\omega$-almost surely.

The field ${^*\Com}$ is usually too big to work with in our applications and is not equipped with a norm.
We will construct a more useful field ${^\rho\Com}$ as the quotient of a subspace of ${^*\Com}$.

Let $\rho_n \to 0$ be a positive sequence, which is regarded as $\rho\in {^*\Com}$. With the notations above, we construct
$$
M_0=\{t\in {^*\Com}: \text{There exists some }N\in \N \text{ such that } |t| < \rho^{-N} \}
$$
and
$$
M_1=\{t\in {^*\Com}: \text{For all } N\in \N, |t| < \rho^N \}.
$$

We remark that since $\rho_n \to 0$, $M_0$ consists of those (equivalence classes of ) sequences that are not growing to infinity too fast, while $M_1$ consists of those tending to $0$ very fast.
It is easy to show that both $M_0$ and $M_1$ form rings with respect to the addition and multiplication of ${^*\Com}$. It can also be shown that $M_1$ is a maximal ideal of ring $M_0$ (cf. Chapter 3.3 in \cite{LightstoneRobinson75}).
We define
$$
{^\rho\Com} = M_0/M_1
$$ 
as the quotient field.
Note that $\Com$ embeds into ${^\rho\Com}$ via constant sequences.

Intuitively, the field ${^\rho\Com}$ lies in between $\Com$ and ${^*\Com}$ consisting of those large infinitesimals and small infinite numbers.
We shall regard each member of $t\in M_1$ as a small infinitesimal, and its multiplicative inverse (provided that $t\neq 0$) a large infinite number.

We can define an equivalence relation on ${^*\Com}$: $x\sim y$ if $x-y \in M_1$.
Note that if $y\in M_0$, then $x\sim y$ if and only if $x\in [y]$ as a member of ${^\rho\Com}$.

\subsection*{Non-Archimedean norm on ${^\rho\Com}$}
One of the many desired properties of ${^\rho\Com}$ is that we can put a norm on it.
Let $x\in M_0-M_1$ and $i\in M_1$ represented by $(x_n)$ and $(i_n)$ respectively.
Note that there exist two integers $k_1, k_2$ such that $\rho^{k_1} \leq |x| < \rho^{k_2}$, hence the ultralimit
$$
\log_{\rho}|x| := \lim_\omega \log|x_n| / \log \rho_n
$$
is a finite number.
Note that
$$
\log_{\rho}|x+i| - \log_{\rho}|x| = \lim_\omega \log_{\rho_n} |\frac{x_n+i_n}{x_n}| = \lim_\omega \frac{\log|1+i_n/x_n|}{\log \rho_n}.
$$
Since $x_n\notin M_1$, $\lim_\omega i_n/x_n = 0$. Hence $\log_{\rho}|x+i| - \log_{\rho}|x| = 0$.

Therefore, we have a well-defined valuation for $[x]\in{^\rho\Com}$ by
$$
\nu([x]) = \log_{\rho}|x|
$$
where $x\in {^\rho\Com}$ is a representative of $[x]$.

Notice that by definition,
$\nu([\rho]) = 1$. More generally,
$\nu([\rho^t]) = t \text{ for } t\in \R$.

To simplify the notations, from now on, we will use a single roman letter to represent a number in ${^\rho\Com}$, and drop the square bracket.

It can be easily verified that for $x, y\in {^\rho\Com}$ (cf. Chapter 3 Lemma 3.1 and 3.2 in \cite{LightstoneRobinson75}), we have
\begin{align*}
\nu(x\cdot y) &= \nu(x) +\nu(y),\\
\nu(x+y) &\geq \min(\nu(x), \nu(y)).
\end{align*}
Hence, $\nu$ defines a non-Archimedean valuation on ${^\rho\Com}$, and this valuation naturally gives rise to a non-Archimedean norm via
$$
|x|_\nu = e^{-\nu(x)}.
$$
The distance function is given by
$$
d(x,y) = |x-y|_\nu.
$$

\subsection*{The field $({^\rho\Com}, d)$ is spherically complete}
A metric space $X$ is said to be {\em spherically complete} if for any nested sequence of (closed) balls $B_0 \supset B_1 \supset ...$, their intersection $\bigcap_j B_j$ is non-empty.
We show:

\begin{theorem}\label{SphericallyComplete}
The field $({^\rho\Com}, d)$ is spherically complete.
\end{theorem}
\begin{proof}
Let $B'_0 \supset B'_1 \supset ...$ be a decreasing sequence of closed balls.
We consider a decreasing sequence of open balls $B_i$ so that $B'_i \supset B_i \supset B'_{i+1}$.
We assume that $B_i$ has radius $r_i$, and denote $q_i = -\log r_i$.
Pick $\alpha_i \in B_i$, and assume that $\alpha_i$ is represented by $(a_{i,n})$.
Since $B_j\subset B_i$ for all $j\geq i$, we know
$$
|\alpha_i - \alpha_j| < r_i.
$$
Equivalently,
$$
\nu(\alpha_i- \alpha_j) = \lim_\omega \log|a_{i,n}- a_{j,n}|/ \log \rho_n > q_i.
$$

We can construct inductively a decreasing sequence $\N = N_0 \supset N_1\supset ...$ such that
\begin{enumerate}
\item $N_k$ is $\omega$-big;
\item $\bigcap_{k=1}^\infty N_k = \emptyset$;
\item For any $i\leq j \leq k$ and $l\in N_k$, we have
$$
\nu_l(a_{i,l}-a_{j,l}) := \log |a_{i, l} - a_{j, l}|/\log \rho_l > q_i.
$$
\end{enumerate}

Indeed, we can set $N_0 = \N$ as the base case.
Assume that $N_k$ is constructed.
To construct $N_{k+1}$, we note that for any $i\leq k+1$, 
$$
\nu(\alpha_i- \alpha_{k+1}) = \lim_\omega \log|a_{i,n}- a_{k+1,n}|/ \log \rho_n > q_i.
$$
Hence, there exists an $\omega$-big set $N$ so that for all $i\leq k+1$ and $l\in N$,
$$
\nu_l(a_{i,l}-a_{k+1,l}) > q_i.
$$
We define $N_{k+1} = N\cap N_k \cap \{n: n\geq k+1\}$, then $N_{k+1}\subseteq N_k$ is still $\omega$-big.
Property $(3)$ is satisfied by the induction hypothesis and by the definition of $N$.
Property $(2)$ holds as $N_k\subseteq \{n: n\geq k\}$ by construction.

We now define the sequence $a_n := a_{k, j}$ for $j\in N_k-N_{k-1}$, and let $\alpha = (a_n)$.
Note that for any $l\in N_i$, by Property $(2)$, $l \in N_k- N_{k-1}$ for some $k\geq i$.
Hence for any $i\in \N$ and $l \in N_i$,
$$
\nu_l(a_{i,l}-a_l) = \nu_l (a_{i,l} - a_{k, l}) > q_i.
$$
Therefore, $\nu(\alpha_i- \alpha) > q_i$. This means that $|\alpha_i - \alpha| < r_i$, so $\alpha\in B_i$.

Since this holds for any $i$, we conclude that $\alpha \in \bigcap_i B_i$, so $\bigcap_i B_i \neq \emptyset$. Therefore, $\bigcap_i B'_i \neq \emptyset$ as well.
\end{proof}

As an immediate corollary, we have (cf. Chapter 3 Theorem 4.1 in \cite{LightstoneRobinson75}):
\begin{cor}
The field $({^\rho\Com}, d)$ is complete.
\end{cor}

\subsection*{The field ${^\rho\Com}$ is algebraically closed}
We will show that not only does ${^\rho\Com}$ have good completion properties, it is also algebraically closed.
\begin{theorem}
${^\rho\Com}$ is algebraically closed.
\end{theorem}
\begin{proof}
Let $z^d+a_{d-1}z^{d-1}+...+a_0$ be a monic polynomial with coefficients $a_n = (a_{n,k})\in {^\rho\Com}$.
We assume that $M < \min(0,\nu(a_0),..., \nu(a_{d-1}))$. Hence there is an $\omega$-big set $N\subseteq \N$ so that for all $k\in N$ and $n=0,..., d-1$,
$$
|a_{n,k}| < \rho_k^M.
$$

Now let $f_k(z) = a_{d-1,k}z^{d-1}+...+a_{0,k}$ and $g(z) = z^d$. 
Note that for any $k\in N$, on the circle centered at $0$ of radius $d\cdot \rho_k^M$ (note that  $\rho_k^M > 1$ as $M<0$) that
\begin{align*}
|f_k(z)| &\leq |a_{d-1,k}|\cdot (d\cdot \rho_k^M)^{d-1} + ... + |a_{0,k}|\\
&<  \rho_k^M \cdot d\cdot (d\cdot \rho_k^M)^{d-1} = |g(z)|.
\end{align*}
By Rouch\'e's theorem, there are $d$ solutions of $g+f_k(z) = 0$ in the ball $B(0, d\cdot \rho_k^M)$.
Let $x_k$ be such a root. 
Note that $x_k$ is defined on an $\omega$-big set $N$, so $x= (x_k)$ represents a point in ${^\rho\Com}$ as $|x_k| < \rho_k^{M+1}$ for all $k\in N$.
Moreover, $x$ satisfies the equation $z^d+a_{d-1}z^{d-1}+...+a_0 = 0$.
Therefore, ${^\rho\Com}$ is algebraically closed.
\end{proof}

Since the value group of ${^\rho\Com}$ is $\R$, as an immediate corollary of the previous two theorems and the Berkovich classification theorem, we have
\begin{cor}\label{TypeIIOnly}
The Berkovich hyperbolic space $\Hyp_{Berk}({^\rho\C})$ consists of only Type II points.
\end{cor}

\subsection*{The residue field of ${^\rho\Com}$} 
For the complexified Robinson's field ${^\rho\Com}$, 
each non-zero element in $\widetilde{^\rho\Com}$ can be represented by a sequence $(z_n)$ with 
$$
\lim_\omega \log|z_n|/\log\rho_n = 0.
$$
Two sequences $(z_n)$ and $(w_n)$ represent the same element in $\widetilde{^\rho\Com}$ if $\lim_\omega \log|z_n-w_n|/ \log \rho_n > 0$.
Thus, the field $\C$ embeds into the residual field $\widetilde{^\rho\Com}$ by taking constant sequence.

\subsection*{Embedding of the field of Puiseux series $\PS$}
In this subsection, we will show how to embed the {\em field of formal Puiseux series} $\PS$ into the Robinson's field ${^\rho\Com}$ (cf. Chapter 3 \S 6 in \cite{LightstoneRobinson75}).

The field $\PS$ is the algebraic closure of the completion of the {\em field of formal Laurent series} $\Com((t))$.
An element in $\mathbf{a}\in\PS$ can be represented by a formal series
$$
\mathbf{a} = \sum_{j\geq 0} a_j t^{\lambda_j}
$$
where $a_j\in\Com$, $\lambda_j\in \Q$ so that if $a_j$ does not vanish for sufficiently large $j$, 
then $\lambda_j\to\infty$ as $j \to \infty$. 
The absolute value is given by 
$$
|\mathbf{a}| = \exp(-\min\{\lambda_j : a_j\neq 0\}) 
$$
provided $\mathbf{a}\neq \mathbf{0}$.

To show we have an embedding, we first prove the following lemma about the convergence of series in ${^\rho\Com}$.

\begin{lem}\label{pseries}
Let $a_j \in \Com$, and $\lambda_j$ be an unbounded increasing sequence of $\R$.
Then the series
$$
\sum_{j=0}^\infty a_j \rho^{\lambda_j}
$$
converges in ${^\rho\Com}$.

Moreover, $|\sum_{j=0}^\infty a_j \rho^{\lambda_j}| = \exp(-\min\{\lambda_j : a_j\neq 0\})$.
\end{lem}
\begin{proof}
Let $\alpha_j = a_j \rho^{\lambda_j}$.
If $a_j = 0$, then $\nu(\alpha_j) = \infty$.
Otherwise, $\nu(\alpha_j) = \nu(\rho^{\lambda_j}) = \lambda_j$.
Since $\lim \lambda_j = \infty$, so $\lim \nu(\alpha_j) = \infty$.
Hence, the series $\sum \alpha_j$ converges in ${^\rho\Com}$ by the convergence criterion in non-Archimedean field.

For the moreover part, let $\sigma_n = \sum_{j=0}^n a_j \rho^{\lambda_j}$ be the associated partial sums. Without loss of generality, we assume that $a_0 \neq 0$, then $\nu(\sigma_n) = \lambda_0$ for all $n$ by the strong triangle inequality. Therefore $|\sum_{j=0}^\infty a_j \rho^{\lambda_j}| = \exp(-\min\{\lambda_j : a_j\neq 0\})$.
\end{proof}

Let $\mathbf{a} = \sum_{j\geq 0} a_j t^{\lambda_j} \in \PS$.
We define $\Psi: \PS \longrightarrow {^\rho\Com}$ as follows.
$$
\Psi(\mathbf{a}) = \sum_{j\geq 0} a_j \rho^{\lambda_j}  \in {^\rho\Com}.
$$
Note that the series converges by Lemma \ref{pseries}.
One can easily verify that $\Psi(\mathbf{a} + \mathbf{b}) = \Psi(\mathbf{a}) + \Psi(\mathbf{b})$ and 
$\Psi(\mathbf{a}\cdot \mathbf{b}) = \Psi(\mathbf{a}) \cdot \Psi(\mathbf{b})$.
Hence we have
\begin{prop}\label{PSEmbedding}
The map
\begin{align*}
\Psi: \PS &\longrightarrow {^\rho\Com}\\
\mathbf{a} = \sum_{j\geq 0} a_j t^{\lambda_j} &\mapsto \sum_{j\geq 0} a_j \rho^{\lambda_j}
\end{align*}
is an embedding of fields and preserves the non-Archimedean norms.
\end{prop}

%%%%%%%%%%
%% Section 10
%% %%%%%%%%

\section{Equivalence of barycentric and Berkovich construction}\label{SectionProofOfConnection}
In this section, it is better to use the upper space model $H$ of the hyperbolic $3$-space $\Hyp^3$.
We can identify $H = \Com\times \R_{>0}$, and a linear map $M(z) = Az+B$ extends to an isometry on $H$ given by
\begin{align}\label{IsometryFormula}
M(z,h) = (Az+B, |A|h).
\end{align}
The distance between two points $(z_1, h_1)$ and $(z_2, h_2)$ is given by the formula
\begin{align}\label{DistanceFormula}
d((z_1, h_1), (z_2, h_2)) = 2\log\frac{\sqrt{|z_1-z_2|^2+(h_1-h_2)^2}+\sqrt{|z_1-z_2|^2+(h_1+h_2)^2}}{2\sqrt{h_1h_2}}.
\end{align}
We will identify $\bm 0$ as the point $(0,1)\in H$.

\subsection*{Construction of the isometric bijection $\Phi$.}
By Corollary \ref{TypeIIOnly}, $\Hyp_{Berk}({^\rho\C})$ consists of only Type II points.
Hence by the Berkovich classification theorem, every point $x\in \Hyp_{Berk}$ can be represented by a closed ball $B(p, R)$.
We consider a linear polynomial of the form 
$$
M(z) = az+ b \in \Affine({^\rho\Com}),
$$
with $M(B(0,1)) = B(p, R)$.
Representing $a$ and $b$ by the sequences $(a_n)$ and $(b_n)$, we get a sequence of M\"obius transformations
$$
M_n (z) = a_n z+ b_n.
$$
Let ${^r\Hyp^3}$ be the asymptotic cone of $\Hyp^3$ with respect to rescaling $r_n = -\log \rho_n$. We define 
\begin{align*}
\Phi: \Hyp_{Berk} &\longrightarrow {^r\Hyp^3}\\
x &\mapsto [(M_n (\bm 0))].
\end{align*}

\begin{prop}\label{IsometricBijection}
$\Phi$ is a well-defined isometric bijection.
\end{prop}
\begin{proof}
We will first check that this definition is well defined.
Let $L_n(z) = a'_nz+b'_n$ be a different representation, where $(a'_n), (b'_n)$ represent $a'$ and $b'$.
Since $M_n(B(0,1)) = L_n(B(0, 1)) = B(p, R)$, $|a| = |a'| = R$ and $|b-b'| \leq R$.
Without loss of generality, we assume $|a_n| \geq |a'_n|$ $\omega$-almost surely.

Let $\epsilon > 0$. By definition of the norm on ${^\rho\Com}$,
\begin{enumerate}
\item $\log |b_n-b'_n|/\log \rho_n > -\log R -\epsilon$, \;\;\;  $\omega$-almost surely;
\item $\log\frac{|a_n|}{|a'_n|}/\log \rho_n > -\epsilon$, \;\;\;  $\omega$-almost surely.
\end{enumerate}
Rearranging the inequalities and using the fact that $|a'| = R$, we have,
\begin{enumerate}
\item $|b_n-b'_n|/|a'_n| < \rho_n^{-\epsilon}$, \;\;\; $\omega$-almost surely;
\item $\frac{|a_n|}{|a'_n|} < \rho_n^{-\epsilon}$, \;\;\; $\omega$-almost surely.
\end{enumerate}
Consider $L^{-1}_n \circ M_n(z) = \frac{a_n}{a'_n}z + \frac{(b_n-b'_n)}{a'_n}$, then using equations \ref{IsometryFormula} and \ref{DistanceFormula}, we conclude that on an $\omega$-big set,
\begin{align*}
d(L_n(\bm 0), M_n(\bm 0)) &= d(\bm 0, L^{-1}_n \circ M_n(\bm 0))\\
&=d((0,1),(|\frac{(b_n-b'_n)}{a'_n}|,|\frac{a_n}{a'_n}|))\\
&< 2\log \frac{\sqrt{\rho_n^{-\epsilon}+(\rho_n^{-\epsilon}-1)^2}+\sqrt{\rho_n^{-\epsilon}+(\rho_n^{-\epsilon}+1)^2}}{2}\\
&< 2\log (2\rho_n^{-\epsilon}) \\
&= 2 \log 2 + 2\epsilon r_n = O(2\epsilon \cdot r_n).
\end{align*}
Since $\epsilon$ is arbitrary, we conclude $(L_n(\bm 0))$ and $(M_n(\bm 0))$ represent the same point in ${^r\Hyp^3}$.
Therefore, $\Phi$ is a well-defined map.

We will now show that $\Phi$ is bijective.
To show this, we will construct the inverse map $\Xi: {^r\Hyp^3} \longrightarrow \Hyp_{Berk}$.
Given a point $x\in {^r\Hyp^3}$, we can represent it as $x = [(M_n(\bm 0))]$, where $M_n(z) = a_nz+b_n$.
Using equations \ref{IsometryFormula} and \ref{DistanceFormula}, we conclude that
$|a_n| < \rho_n^{-N}$ and $|b_n| < \rho_n^{-N}$ for some $N \in \N$ $\omega$-almost surely.
Hence $(a_n), (b_n)$ represent $a, b\in {^\rho\Com}$, with $a\neq 0$.
Denote $M(z) = az+b \in \Affine({^\rho\Com})$, and
we define 
$$
\Xi(x) = M(B(0,1)) \in \Hyp_{Berk}.
$$
Similarly, we can easily check that $\Xi$ is well defined, and $\Phi \circ \Xi$, $\Xi \circ \Phi$ are identity maps.
Therefore $\Phi$ is bijective.

We will now show that $\Phi$ is an isometry.
Given $a, b\in {^\rho\Com}$ represented by $(a_n)$ and $(b_n)$, then $M(z) = az+b\in \Affine({^\rho\Com})$ and $[(x_n)]\mapsto [(M_n(x_n))] \in {^r\Hyp^3}$ where $M_n(z) = a_nz+b_n$ are isometries of $\Hyp_{Berk}$ and ${^r\Hyp^3}$ respectively.
Hence, it suffices to show $d(x_g, M(x_g)) = d([(\bm 0)], [(M_n(\bm 0))])$.

If $M(x_g)$ is represented by a closed ball either contained or containing $B(0,1)$, then we can choose $M(z) = az$, and $d(x_g, M(x_g)) = |\log|a||$.
A direct computation using equation \ref{DistanceFormula}, we have $d(\bm 0, M_n(\bm 0)) = |\log |a_n||$, so 
$$
d(x^0, (M_n(\bm 0))) = \lim_\omega -|\log|a_n|| / \log \rho_n = |\log|a||,
$$
where the last equality holds by the definition of norm on ${^\rho\Com}$.

More generally, if $M(x_g)$ is represented by a closed ball $B(p, R)$ disjoint from $B(0,1)$, one can construct a geodesic by connecting $B(0,1)$ to $B(0, |p|)$ and then connecting $B(0, |p|)$ to $B(p, R)$.
By the above argument, one can show that $\Phi$ is an isometry on either geodesic segment.
Since $\Phi$ is a bijection, and ${^r\Hyp^3}$ is a tree, this means $d(x_g, M(x_g)) = d([(\bm 0)], [(M_n(\bm 0))])$. Therefore, $\Phi$ is an isometry.
\end{proof}

\subsection*{The isometric bijection $\Phi$ is a conjugacy.}
Before proving the equivalence theorem, we need the following algebraic lemma.
\begin{lem}\label{ReductionAtGaussPoint}
Let $\mathbf{f} \in \Rat_d({^\rho\Com})$, and $f_n\in \Rat_d(\C)$ be a sequence representing $\mathbf{f}$.
Let $F=\lim_\omega \E f_n$ be the limiting map on ${^r\Hyp^3}$.
Then the reduction of $\mathbf{f}$ has degree $\geq 1$ if and only if $F (x^0) = x^0$.
\end{lem}
\begin{proof}
If the reduction of $\mathbf{f}$ has degree $\geq 1$, we can represent 
$$
\mathbf{f}(z) = \frac{a_d z^d+ ... +a_0}{b_dz^d+ ... +b_0}
$$ 
with $\max\{|a_d|, ..., |a_0|\} = 1$ and $\max\{|b_d|, ..., |b_0|\} = 1$.
We denote 
$$
f_n(z) = \frac{a_{d,n} z^d+ ... +a_{0,n}}{b_{d,n}z^d+ ... +b_{0,n}},
$$
where $(a_{k,n})$ and $(b_{k,n})$ represent $a_k$ and $b_k$ in ${^\rho\Com}$.

Let $i_{top}$ be the largest index $i$ so that 
$$
\lim_\omega |a_{j,n}| / |a_{i,n}| <\infty
$$
for all $j=0,...,d$.
Similarly, we define $i_{bot}$ accordingly.

If $i_{top} \neq i_{bot}$, we let $L_n(z) = b_{i_{bot},n}/a_{i_{top},n}z$, then
\begin{align*}
\lim_\omega L_n \circ f_n &= \lim_\omega \frac{a_{d,n}/a_{i_{top},n} z^d+ ... +a_{0,n}/a_{i_{top},n}}{b_{d,n}/b_{i_{bot},n} z^d+ ... +b_{0,n}/b_{i_{bot},n}}\\
&=\frac{z^{i_{top}} + ...}{z^{i_{bot}}+...}
\end{align*}
has degree $\geq 1$.
Since $r_n \to \infty$, by Proposition \ref{bdd},
$$
\lim_\omega d_{\Hyp^3}(\bm 0, L_n \circ \E f_n (\bm 0))/r_n = 0.
$$
Since $|b_{i_{bot}}| = |a_{i_{top}}| = 1$, by equations \ref{IsometryFormula} and \ref{DistanceFormula}, for any $\epsilon > 0$
$$
d_{\Hyp^3}(\bm 0, L_n(\bm 0)) < \epsilon r_n
$$ 
for all large $n$. 
Hence we have $F (x^0) = x^0$.

If $i_{top} = i_{bot}$, we let $L_n(z) = z - a_{i_{top},n}/b_{i_{bot},n}$ and consider $g_n = L_n \circ f_n$.
Since $|b_{i_{bot}}| = |a_{i_{top}}| = 1$, by equations \ref{IsometryFormula} and \ref{DistanceFormula}, for any $\epsilon > 0$
$$
d_{\Hyp^3}(\bm 0, L_n(\bm 0)) < \epsilon r_n
$$ 
for all large $n$.
Moreover, the map $\mathbf{g}$ represented by $g_n$ has non trivial reduction, and the indices $i_{top}$ and $i_{bot}$ for $g_n$ are different.
By applying the previous argument for $g_n$, we conclude that
$F (x^0) = x^0$.

Conversely, if $F (x^0) = x^0$, by naturality of the barycentric extension and Proposition \ref{bdd}, we can choose $L_n(z) = a_n z+ b_n$ with
$$
\lim_\omega d(\bm 0, L_n(\bm 0))/ r_n = 0
$$
so that $\lim_\omega L_n \circ f_n$ has degree $\geq 1$.
Therefore, by Proposition \ref{IsometricBijection}, $a_n$ and $b_n$ represents $a, b\in {^\rho\C}$ with $|a| = 1$ and $|b| < 1$.
Let $L(z) = az+ b\in \PSL_2({^\rho\C})$. 
Then $L \circ \mathbf{f}$ has non-trivial reduction.
Hence $\mathbf{f}$ has non-trivial reduction.
\end{proof}

We are now ready to prove a more precise version of Theorem \ref{EquivTheorem}:
\begin{theorem}\label{ConnectionToBerkovichDynamics}
Let $f_n \to \infty$ in $\Rat_d(\C)$. Let ${^r\Hyp^3}$ be the asymptotic cone with respect to the rescalings $r_n = \max_{y\in \E f_n^{-1}(\bm 0)}d_{\Hyp^3}(y, \bm 0)$, and ${^\rho\C}$ be the complexified Robinson's field with respect to $\rho_n = e^{-r_n}$.
Then $f_n$ represents a rational map $\mathbf{f}\in \Rat_d({^\rho\C})$, and
$$
\Phi\circ B = F \circ \Phi,
$$
where $F = \lim_\omega \E f_n$ is the limiting map on ${^r\Hyp^3}$, and $B$ is the Berkovich extension of $\mathbf{f}$ on $\Hyp_{Berk}({^\rho\C})$.

Conversely, if $\mathbf{f} \in \Rat_d({^\rho\C})$, and $F = \lim_\omega \E f_n$ for (any) sequence $f_n$ representing $\mathbf{f}$, we have
$$
\Phi\circ B = F \circ \Phi.
$$
\end{theorem}
\begin{proof}
Write $f_n(z) = \frac{a_{d,n}z^d + ... + a_{0,n}}{b_{d,n}z^d + ... + b_{0,n}}$ so the maximum of $\lim_{\omega} |a_{i,n}|/\log \rho_n$ and $\lim_{\omega} |b_{i,n}|/\log \rho_n$ is $1$.
Thus, the sequences of coefficients $a_{i,n}$, $b_{i,n}$ represent elements in ${^\rho\C}$. Let $\mathbf{f}$ be the corresponding rational map with coefficients in ${^\rho\C}$.
Note that a priori, the degree of $\mathbf{f}$ may be $\leq d$.

Let $\Phi: \Hyp_{Berk} \longrightarrow {^r\Hyp^3}$ be the map defined as above. Then $\Phi$ is an isometric bijection.
Given $x\in {^r\Hyp^3}$ represented by the sequence $M_n(\bm 0)$ where $M_n \in \PSL_2(\C)$.
Assume that $y = F(x)$ is represented by  $L_n(\bm 0)$ where $L_n \in \PSL_2(\C)$.
Let $M, L \in \PSL_2({^\rho\C})$ be represented by $M_n$ and $L_n$ respectively.
Then by naturality of the barycentric extension and Lemma \ref{ReductionAtGaussPoint}, the reduction
$$
L^{-1} \circ \mathbf{f} \circ M
$$
has degree $\geq 1$.
By Proposition \ref{ReductionGEQ1} and the definition of $\Phi$, $F(\Phi(x)) = \Phi(B(y))$.
Hence $\Phi$ is a conjugacy between $F$ and $B$.

Since $F$ has degree $d$ by Theorem \ref{DGL}, we conclude $B$ has degree $d$ as well.
Therefore, $\mathbf{f}$ has degree $d$, so $\mathbf{f} \in \Rat_d({^\rho\C})$.

The converse part follows by a similar argument.
\end{proof}

\subsection*{A version of holomorphic families.}
Let $f_t$ be a holomorphic family of rational maps of degree $d>1$ defined over the punctured unit disk $\Delta^* = \{t\in\C: 0 < |t| < 1\}$. We also assume that all the coefficients of $f_t$ extend to meromorphic functions on the unit disk $\Delta$.
We may view $\mathbf{f} = f_t$ as a rational map with coefficients in the field of formal Puiseux series $\PS$.
%%%%
%Theorem 1.7
%%%%
\begin{theorem}\label{ConnectionToBerkovichDynamicsFamily}
Let $\rho_n \to 0$ and $r_n = -\log|\rho_n|$. 
There is an isometric embedding
$$
\Phi: \Hyp_{Berk}(\PS)\hookrightarrow {^r\Hyp^3} \cong \Hyp_{Berk}({^\rho\C}).
$$

Moreover, if $\mathbf{f} = f_t$ is a holomorphic family of rational maps of degree $d>1$ defined over $\Delta^*$, then
$$
\Phi\circ B = F \circ \Phi,
$$
where $F = \lim_\omega \E f_{\rho_n}$ is the limiting map on ${^r\Hyp^3}$, and $B$ is the Berkovich extension of $\mathbf{f}$ on $\Hyp_{Berk}(\PS)$.
\end{theorem}
\begin{proof}
By Proposition \ref{PSEmbedding}, the field $\PS$ naturally embeds isometrically into ${^\rho\Com}$. 
Such an embedding gives an isometric embedding of $\Hyp_{Berk}(\PS)$ into $\Hyp_{Berk}({^\rho\Com})$ by the Berkovich classification theorem.
A rational map $\mathbf{f}$ with coefficients in $\PS$ can be thought of as a rational map with coefficients in ${^\rho\Com}$ via the embedding.
Its action on $\Hyp_{Berk}(\PS)$ naturally extends to $\Hyp_{Berk}({^\rho\Com})$.
The theorem now follows from Theorem \ref{ConnectionToBerkovichDynamics}.
\end{proof}

\section{Properties of the limiting map $F$}
In this section, using the equivalence result proved in Theorem \ref{EquivTheorem} and Theorem \ref{ConnectionToBerkovichDynamics}, we summarize some properties of $F$ that will be used later.

\subsection*{Expansion and local degrees.}
Recall that the {\em local degree} $\deg_x F$ is the degree of the reduction $\tilde{\mathbf{g}}$ of $\mathbf{g} = L^{-1}\circ \mathbf{f} \circ M$, where $x = [(M_n(\bm 0))]$, $F(x) = [(L_n(\bm 0))]$ and $M_n, L_n$ represent $M, L \in \PSL_2({^\rho\C})$. 
The local degree in the direction $\vec v\in T_x{^r\Hyp^3}$, denoted by $m_{\vec v} F$, is defined as the degree of $\tilde{\mathbf{g}}$ at $\vec v$.

The following theorem allows us to interpret the local degree quite concretely as a local expansion factor (see Proposition 3.1 in \cite{Rivera-Letelier05}, Theorem 9.26 in \cite{BakerRumely10} and Theorem 4.7 in \cite{Jonsson15}).
\begin{theorem}\label{CritVec}
Let $x\in {^r\Hyp^3}$, and $\vec{v} \in T_x{^r\Hyp^3}$.
\begin{enumerate}
\item For all sufficiently small segment $\gamma = [x,w]$ representing $\vec v$, $F$ maps $\gamma$ homeomorphically to $F(\gamma)$ and expands by a factor of $m_{\vec v} F$.
\item If $\vec w\in T_{F(x)}{^r\Hyp^3}$, and $\vec v_1,..., \vec v_k$ are the preimages of $\vec w$ in $T_x{^r\Hyp^3}$, then
$$
\sum_{i=1}^k m_{\vec v_i} F = \deg_x F.
$$
\end{enumerate}

If $\alpha:(0,\infty) \longrightarrow {^r\Hyp^3}$ is an end associated to $x\in \Proj^1_K$, then for all sufficiently large $K$,  $F$ maps $\alpha([K,\infty))$ homeomorphically to $F(\alpha([K,\infty)))$ and expands by a factor of $\deg_x \mathbf{f}$.
\end{theorem}

\subsection*{Approximating disks and annuli}
We say a sequence of simply connected domains $D_n\subset \C$ {\em approximates} an open disk $B(a, r)^- \subseteq {^\rho\C}$ if 
\begin{enumerate}
\item any sequence $z_n \in D_n$ represents a point in the closed disk $B(a,r)$;
\item any sequence $z_n \in \hat\C - D_n$ represents a point in $\Proj^1_{^\rho\C} - B(a,r)^-$.
\end{enumerate}

Similarly, we say a sequence of annuli $A_n \subseteq \C$ approximate an open annulus $A = B(a,R)^- - B(a,r) \subseteq {^\rho\C}$ if
\begin{enumerate}
\item any sequence $z_n \in A_n$ represents a point in $B(a,R)- B(a,r)^-$;
\item any sequence $z_n \in \hat\C - A_n$ represents a point in $(\Proj^1_{^\rho\C} - B(a,R)^-) \cup B(a,r)$.
\end{enumerate}

The following lemma gives the classification for approximating disks and annuli.
To distinguish the disks in $\C$ and ${^\rho\C}$, we use $B_\C(a, r)$ to represent an open disk centered at $a$ with radius $r$ in $\C$.
\begin{lem}\label{lem:cap}
Let $B(0,1)^- \subseteq {^\rho\C}$ be an open disk. 
A sequence of simply connected domains $D_n \subseteq \C$ approximates $B(0,1)^-$ if and only if
$$
B_\C(0, \rho_n^{-s}) \subseteq D_n \subseteq B_\C(0, \rho_n^{-S}) \, \, \text{ $\omega$-almost surely}
$$
for all $s < 0 < S$.

Let $A = B(0,R)^- - B(0,r) \subseteq {^\rho\C}$ be an open annulus.
Denote $A_\C(s, S) = B_\C(0, \rho_n^{-S}) - \overline{B_\C(0, \rho_n^{-s})} \subseteq \C$.
A sequence of annuli $A_n \subseteq  \C$ approximates $A$ if and only if
$$
A_\C(s_1, S_1) \subseteq A_n \subseteq A_\C(s_2, S_2)  \, \, \text{ $\omega$-almost surely}
$$
for all $s_2 < \log r < s_1$, $S_1 < \log R < S_2$.
\end{lem}
\begin{proof}
To simplify the notations, all the statements below are assumed to hold $\omega$-almost surely.
Assume $D_n$ approximates $B(0,1)^-$. Suppose for contradiction that $B_\C(0, \rho_n^{-s}) \not\subseteq D_n$. Let $z_n \in B_\C(0, \rho_n^{-s}) - D_n$. Then $(z_n)$ represents a point in $B(0, e^{-s}) \subseteq B(0, 1)^-$. Since $D_n$ approximates $B(0,1)^-$, $(z_n)$ represents a point in  $\Proj^1_{^\rho\C} - B(0, 1)^-$. This is a contradiction.
Thus, $B_\C(0, \rho_n^{-s}) \subseteq D_n$.
Similarly, $D_n \subseteq B_\C(0, \rho_n^{-S})$.

Conversely, assume $B_\C(0, \rho_n^{-s}) \subseteq D_n \subseteq B_\C(0, \rho_n^{-S})$ for all $s< 0< S$. Let $z_n \in D_n$. Then $z_n$ represents a point $z \in B(0, 1)$. Indeed, otherwise, let $S > 0$ with $e^S < |z|$. Then $z_n \notin B_\C(0, \rho_n^{-S})$ contradicting to the assumption. Similarly, any sequence $z_n\in \hat\C - D_n$ represents a point in $\Proj^1_{^\rho\C} - B(0, 1)^-$.
So $D_n$ approximates $B(0,1)^-$.

The proof for the annulus is similar.
\end{proof}

Recall that for $\vec v\in T_x{^r\Hyp^3}$, $U_{\vec v}$ is the component ${^r\Hyp^3} - \{x\}$ associated to $\vec v$.
By identifying ${^r\Hyp^3}$ with $\Hyp_{Berk}({^\rho\Com})$, $U_{\vec v}$ is identified with an open disk $D^-\subseteq \Proj^1_{^\rho\C}$.
We say a sequence of simply connected domains $D_n\subset \hat\C$ {\em approximates} $U_{\vec{v}}$ if after changing coordinates, $D^- \subseteq {^\rho\C}$ and $D_n$ approximates $D^-$.

Given two points $x, y\in {^r\Hyp^3}$, there is a unique component $U^{x,y}$ of ${^r\Hyp^3} - \{x,y\}$ with boundary $\{x, y\}$.
We can identify it with an open annulus $A \subseteq \Proj^1_{^\rho\C}$.
We define a sequence $A_n \subseteq \hat\C$ {\em approximates} $U^{x,y}$ in a similar way.

\begin{remark}
We remark that if $A_n$ approximate $U^{x,y}$, then the modulus $m(A_n)$ satisfies
$$
\lim_\omega 2\pi m(A_n)/r_n = d(x,y).
$$
In particular, the sequence of moduli $m(A_n)$ goes to infinity.
\end{remark}

\subsection*{Limiting dynamics on disks and annuli}
It is important to understand how the limiting map acts on $U_{\vec v}$ and $U^{x,y}$.
We first define a {\em critical end} as an end of ${^r\Hyp^3}$ associated with a critical point in $\Proj^1_{^\rho\C}$.
The {\em critical tree} $C$ is defined as the convex hull of the critical ends. 
Since the residual characteristic of ${^\rho\C}$ is $0$, any point not in the critical tree has local degree $1$ (see Lemma 4.12 in \cite{Jonsson15}, or \cite{Faber13, Faber13b}).
The following lemma follows from from Theorem 9.42, Theorem 9.46 in \cite{BakerRumely10} and the above observation.
\begin{lem}\label{IsomAwayCritTree}
\begin{enumerate}
\item Let $\vec v \in T_x{^r\Hyp^3}$, and $\vec w = D_x F(\vec v)$. 
Assume that $U_{\vec v}$ does not intersect the critical tree $C$.
Then $F$ is an isometric bijection from $U_{\vec v}$ to $U_{\vec w}$.
\item Let $x, y \in {^r\Hyp^3}$, and $x'=F(x)$, $y'=F(y)$.
Assume that $U^{x,y}$ contains no critical ends.
Then $F$ is a branched covering map from $U^{x,y}$ to $U^{x',y'}$.
\end{enumerate}
\end{lem}

Let $U^{x,y}$ and $U^{x',y'}$ be as in the above lemma. 
Since $U^{x',y'}$ contains no ends associated with critical values. 
We can choose a sequence $A'_n$ approximating $U^{x',y'}$ which contains no critical values of $f_n$.
Then $f_n^{-1}(A'_n) = A_{1,n} \cup ... \cup A_{k,n}$ is a union of annuli by Riemann-Hurwitz formula.
Note that if the sequence  $z_n \in \hat\C$ represents $z \in \Proj^1_{^\rho\C}$, then $f_n(z_n)$ represents $\mathbf{f}(z)$.
Thus, after changing the subindices, by Lemma \ref{lem:cap}, $A_n = A_{1,n}$ approximates $U^{x,y}$.
Therefore, by the Lemma \ref{IsomAwayCritTree}, we have
\begin{lem}\label{ApproxAnnuli}
Let $x, y \in {^r\Hyp^3}$, and $x'=F(x)$, $y'=F(y)$.
Assume that $U^{x,y}$ contains no critical ends, so $F: U^{x,y}\longrightarrow U^{x',y'}$ is a degree $e$ branched covering.
Then there exist sequences of annuli $A_n$ and $A_n'$ approximating $U^{x,y}$ and $U^{x',y'}$ such that
$$
f_n: A_n \longrightarrow A_n'
$$
is a degree $e$ covering $\omega$-almost surely.
\end{lem}

\subsection*{Extension on geodesic segments.}
The following lemma allows us to extend the tangent map to geodesic segments (see Lemma 9.38 in \cite{BakerRumely10}):
\begin{lem}\label{ExtendTangent}
Let $\vec v\in T_x{^r\Hyp^3}$, and $\vec w = D_xF (\vec v) \in T_y{^r\Hyp^3}$.
Let $y' \in U_{\vec w}$.
Then there exists $x'\in U_{\vec v}$ such that $F$ maps $[x,x']$ homeomorphically to $[y,y']$.
\end{lem}

\section{Marked length spectra and periodic ends}\label{MandL}
In this section, we shall illustrate how to use the limiting dynamics on $\R$-trees to study the marked length spectra of rational maps.

\subsection*{Markings and Length Spectra.}
Recall that a conjugacy class of a rational map $[f]$ is {\em hyperbolic} if any of the following equivalent definition holds (see Theorem 3.13. in \cite{McM94}):
\begin{enumerate}
\item The postcritical set $P(f)$ is disjoint from the Julia set $J(f)$.
\item There are no critical points or parabolic cycles in the Julia set.
\item Every critical point of $f$ tends to an attracting cycle under forward iteration.
\item There is a smooth conformal metric $\rho$ defined on a neighborhood of the Julia set such that $|f'(z)|_{\rho}>C>1$ for all $z\in J(f)$.
\item There is an integer $n > 0$ such that $f^n$ strictly expands the spherical metric on the Julia set.
\end{enumerate}
The space of hyperbolic rational maps is open in $\M_d$, and a connected component is called a {\em hyperbolic component}.
For each hyperbolic component $\HC$, there is a topological dynamical system
$$
\sigma: J\longrightarrow J,
$$
such that for any $[f]\in \HC$, there is a homeomorphism 
$$
\phi(f):J \longrightarrow J(f)
$$
which conjugates $\sigma$ and $f$.
A particular choice of such $\phi(f)$ will be called a {\em marking} of the Julia set.

Let $[f]\in \HC$ be a hyperbolic rational map with a marking $\phi: J \longrightarrow J(f)$.
Let $\mathscr{S}$ be the space of periodic cycles of the topological model 
$\sigma:J\longrightarrow J$.
We define the length of a periodic cycle $C\in \mathscr{S}$ for $[f]$ by
$$
L(C, [f]) = \log |(f^q)'(z)|,
$$
where $q = |C|$ and $z \in \phi(C)$. 
The collection $(L(C, [f]): C\in \mathscr{S}) \in \R_{+}^\mathscr{S}$ will be called the {\em marked length spectrum} of $[f]$.

\subsection*{Markings of periodic ends.}
Let $[f_n]\in \HC$ be a degenerating sequence with markings $\phi_n: J \longrightarrow J([f_n])$.
Recall that we have
$$
r_n:= r([f_n])= \inf_{x\in \Hyp^3} \max_{y\in \E f_n^{-1}(x)} d_{\Hyp^3}(x, y) \to \infty
$$
where $f_n$ is a representative of $[f_n]$.
We choose a sequence of representatives $f_n$ of $[f_n]$ so that
$$
\max_{y\in \E f_n^{-1}(\bm{0})} d_{\Hyp^3}(\bm{0}, y) \leq r_n +1.
$$
Let $F=\lim_\omega \E f_n$ be the limiting map on the asymptotic cone ${^r\Hyp^3}$ with rescalings $r_n$.
The sequence of markings $\phi_n: J \longrightarrow J([f_n])$ naturally gives a map $\phi: J \longrightarrow \Proj^1_{^\rho\C}$ by
$$
\phi(t) = [(\phi_n(t))].
$$

Note that if $t$ is a periodic point of period $q$, then
$f_n^q(\phi_n(t)) = \phi_n(t)$ for all $n$.
Hence, $\phi(t)$ is a periodic point of $\mathbf{f} \in \Rat_d({^\rho\C})$ represented by $f_n$.
Therefore, by Theorem \ref{EquivTheorem}, a periodic cycle $C\in \mathscr{S}$ is identified via $\phi$ with a cycle of periodic ends of ${^r\Hyp^3}$ for $F$.

\subsection*{Translation length on ends.}
Let $\alpha: [0,\infty) \longrightarrow {^r\Hyp^3}$ represent an end.
We will first show that the translation length
$$
L(\alpha, F) = \lim_{x_i \to \alpha} d(x_i, x^0) - d(F (x_i), x^0)
$$
is well defined.

If $\alpha$ is not a critical end, then by Theorem \ref{CritVec},
$F$ is an isometry on $\alpha([K,\infty))$ for a sufficiently large $K$.
Hence $d(\alpha(t), x^0) - d(\E (f_n) (\alpha(t)), x^0)$ is constant for $t\geq K$, so the translation length is well defined.

If $\alpha$ is a critical end, then by Theorem \ref{CritVec}, $F$ is expanding with derivative $e\in \N_{\geq 2}$ on $\alpha([K,\infty))$ for a sufficiently large $K$.
Hence, 
\begin{align*}
(d(\alpha(t), x^0) &- d(F (\alpha(t)), x^0)) \\
&- (d(\alpha(K), x^0) - d(F (\alpha(K)), x^0)) \\
&= (1-e)(t-K)
\end{align*}
for all $t\geq K$, so the translation length
$$
L(\alpha, F) = \lim_{x_i \to \alpha} d(x_i, x^0) - d(\E (f_n) (x_i), x^0) = -\infty.
$$

If $C = \{\alpha_1, ..., \alpha_q\}$ is a cycle of periodic ends, we define
$$
L(C, F) = \sum_{i=1}^q L(\alpha_i, F).
$$
We say a periodic end $C$ is attracting, indifferent or repelling if 
$L(C, F) < 0, =0$ or $>0$.

Using the equivalence result in Theorem \ref{EquivTheorem}, we give an algebraic proof of Theorem \ref{TL} (cf. Theorem 1.4 in \cite{L19}):
\begin{proof}[Proof of Theorem \ref{TL}]
By considering iterations if necessary, we may assume $C = \{t\} \in \mathscr{S}$ has period $1$.
By naturality, we may also assume $\phi_n(t) = 0$ for all $n$.
Let $\mathbf{f}\in \Rat_d(^\rho\C)$ be represented by $f_n$.
Note that $\mathbf{f}(0) = 0$.

We first claim that $0\in {^\rho\C}$ is not a critical point of $\mathbf{f}$.
Indeed, since $t\in J$, $|f_n'(\phi_n(t))| > 1$, so $|\mathbf{f}'(z)| \geq 1$.
Hence, we can write
$$
\mathbf{f}(z) = a z (1+ \mathbf{h}(z)),
$$
with $a = \mathbf{f}'(0)$ and $h(0) = 0$.

Let $M, L \in \PSL_2({^\rho\C})$ be $M(z) = kz$ and $L(z) = akz$.
Note that if $|k|$ is sufficiently small, 
$$
L^{-1} \circ \mathbf{f}\circ M(z) = z(1+ \mathbf{h}(kz))
$$
has non-constant reduction.
Hence, we conclude that $F$ sends $(M_n(\bm 0)) \in {^r\Hyp^3}$ to $(L_n(\bm 0))$ where $M_n$ and $L_n$ represent $M$ and $L$ respectively.
Therefore, the translation length is
$$
L(C, F) = \log |a|.
$$
By choosing representatives, we have
$$
\log |a| = -\nu(a) =  -\lim_\omega \log|f_n'(0)|/\log \rho_n = \lim_\omega L(C, [f_n])/ r([f_n]).
$$
Thus, the result follows.

To show the moreover part, we note that any periodic point $z_n$ not in the Julia set is non-repelling. The multiplier $|(f_n^q)'(z_n)|\leq 1$ has bounded norm.
Hence $|(\mathbf{f}^q)'(z)| \leq 1$.
But the total number of periodic points of period $q$ for $\mathbf{f}$ is the same as that of $f_n$, so all repelling periodic points of $\mathbf{f}$ are represented by periodic points in the Julia sets.
Therefore, every cycle of repelling periodic ends is represented by some $C\in \mathscr{S}$.
\end{proof}

Since the ultralimit of a sequence of real numbers belong to its limit set
$\lim_\omega x_n \in \overline{\{x_n\}}$, by choosing subsequences and a diagonal argument, we immediately have
\begin{cor}\label{LimitLengthSpec}
Let $[f_n]\in \HC$ be a degenerating sequence, and $F$ be defined as above. Then after possibly passing to subsequences, we have
$$
L(C, F) = \lim_{n\to\infty} L(C, [f_n])/r([f_n]).
$$
\end{cor}

%%%%%%%%%%%%%%
%%
%% Section 2
%%
%%%%%%%%%%%%%%

\section{Hyperbolic components with nested Julia sets}\label{HypComNestJ}
In this section, we will study a special type of hyperbolic components.
Recall that
a hyperbolic rational map $f\in \Rat_d(\Com)$ has {\em nested} Julia set $J = J(f)$ if
\begin{enumerate}
\item There are two points $p_1, p_2\in \hat\C$ such that any component of $J$ separates $p_1$ and $p_2$;
\item $J$ contains more than one component.
\end{enumerate}
A hyperbolic component $\HC$ is said to have nested Julia sets if the Julia set of any rational map in $\HC$ is nested.

We remark that $J$ has uncountably many components by taking preimages and their accumulation points.
The first example of hyperbolic rational map with nested Julia sets was introduced by McMullen in \cite{McM88}, where the Julia set is homeomorphic to a Cantor set times a circle.

\subsection*{Topological properties}
We begin by introducing some terminologies and deducing some topological properties of the Julia set.

Let $f$ be a hyperbolic rational map with nested Julia sets, and $J$ be its Julia set.
Since each component of $J$ separates $2$ points $p_1, p_2$, any Fatou component of $f$ is either simply connected, or is isomorphic to an annulus.
We will call a simply connected Fatou component a {\em disk Fatou component}, and an annulus Fatou component a {\em gap}.
A gap is called {\em critical} if it contains critical points of $f$. 

We say a Julia component $K$ of $J$ is
\begin{itemize}
\item {\em extremal} if $J- K$ is contained in a component of $\hat \C - K$;
\item {\em buried} if $K$ does not intersect the boundary of any Fatou component;
\item {\em unburied} if it is not buried.
\end{itemize}

We first verify the following basic topological properties.
\begin{lem}\label{TopologicalPropertyNestedJuliaSet}
Let $f$ be a hyperbolic rational map with nested Julia sets $J$. Then
\begin{enumerate}
\item The gaps are nested, backward invariant, and any gap is eventually mapped to a disk Fatou component.
\item Any non-extremal periodic Julia component $K$ is buried.
\item Any unburied Julia component $K$ is eventually mapped to an extremal Julia component by a degree $e = e(K)$ covering.
\item Any buried Julia component $K$ is a Jordan curve.
\item A critical gap is mapped to a disk Fatou component whose boundary is contained in an extremal Julia component.
\item The extremal Julia components are mapped to extremal ones.
\end{enumerate}
\end{lem}
\begin{proof}
Since the Julia set is nested, the gaps are nested. Since no disk Fatou component is mapped to a gap, they are backward invariant.
Let $U$ be a periodic Fatou component. Since $f$ is hyperbolic, $U$ is attracting. 
Thus, $U$ contains a critical point of the first return map $f^q$.
By Riemann-Hurwitz formula, $U$ is a disk Fatou component.
By no wandering domain theorem, every Fatou component is pre-periodic, so the statement (1) follows.

Let $K$ be a periodic Julia component of period $q$. By conformal surgery (see Proposition 6.9 in \cite{McM88}), there exists a post-critically finite hyperbolic rational map $g$ so that $f^q: K \longrightarrow K$ is conjugate to $g: J(g) \longrightarrow J(g)$.

If $K$ is not extremal, then exactly two components of $\hat \C - K$ intersects $J$.
Thus, the two corresponding components of $\hat \C - J(g)$ are totally invariant under $g$, so $g$ is conjugate to either $z^e$ or $\frac{1}{z^e}$ for some $e$. Hence $K$ is a Jordan curve.
Since $K$ is not extremal, $K$ does not intersect the boundary of a disk Fatou component.
The component $K$ does not intersect the boundary of a gap either, as otherwise, we would have a periodic gap. Thus $K$ is buried, and the statement (2) follows.

Since unburied Julia components are pre-periodic by no wandering domain theorem, and the image of an unburied Julia component is unburied, the statement (3) follows from (2).

By Theorem 1.2 in \cite{PilgrimTan00}, any wandering component of $J$ is a Jordan curve.
Since extremal components are not buried, and any non-extremal periodic Julia component is a Jordan curve, the statement (4) follows.

By Riemann-Hurwitz formula, critical gaps are mapped to disk Fatou components.
Let $U$ be a critical gap, and $K_1, K_2$ be two components of the Julia set $J$ containing $\partial U$.
Note $K = f(K_1) = f(K_2)$ is a Julia component.
We claim $K$ is an extremal Julia component.
Otherwise, exactly two components $V_1, V_2$ of $\hat\C-K$ intersect $J$.
Since $U$ is mapped to a disk Fatou component, the induced surjective map $f_*: I(K_1) \longrightarrow I(K)$ between ideal boundaries of $K_1$ and $K$ sends the ideal boundary corresponding to $U$ in $I(K_1)$ to $I(K) - I(V_1\cup V_2)$ (cf. \S 6 in \cite{McM88}). Thus there are at least 3 components of $\hat\C - K_1$ intersecting $J$, which is a contradiction and the statement (5) follows.

A similar argument also shows that the extremal Julia components are mapped to extremal ones, giving the statement (6).
\end{proof}

\subsection*{Shishikura tree for nested Julia sets}
We shall see that the dynamics on the gaps for a hyperbolic rational map $f$ with nested Julia sets is captured by the Shishikura tree.
The Shishikura tree was first introduced by Shishikura in \cite{Shishikura89} to study Herman rings.
We will briefly introduce the special case that we are interested in, and refer the readers to \cite{Shishikura89} for details and more general theories.

Let $A$ be an annulus of $\Com$ with modulus $M$. 
Then there is a conformal map unique up to post composing with rotation $\phi_A: A \longrightarrow \{z: 1<|z|<e^{2\pi M}\}$ sending the inner boundary to the inner one, and outer to the outer one.
We define
\begin{align*}
A[z] &:= \phi_A^{-1}(\{\zeta: |\zeta| = |\phi_A(z)|\}),\\
A(x,y) &=\{z\in A: A[z] \text{ separates $x$ and $y$}\}.
\end{align*}
Note that $A(x,y)$ is a sub-annuli of $A$.

Let $f$ be a hyperbolic rational map with nested Julia sets $J$ of degree $d$, and let $\mathscr{A}$ be the collection of gaps. 
Since the union of gaps is contained in the annulus bounded by the two extremal Julia components, which has finite moduli, by Gr\"otzch inequality, $\sum_{A\in \mathscr{A}} m(A) < \infty$.
We define a pseudo metric on $\hat\C$ by
$$
d(x,y) = \sum_{A\in \mathscr{A}} m(A(x,y)).
$$
In the usual fashion, we identify two points $x\sim y$ if $d(x,y) = 0$.

Note that if $x,y$ are in the same component of $\hat\C - \bigcup_{A\in \mathscr{A}} A$, then $d(x,y) = 0$.
More generally, the annulus $A$ is foliated by Jordan curves $A[z]$, and any two points $x, y\in A[z]$ have distance $0$.
Thus, the quotient $\hat\C / \sim$ is constructed by collapsing complementary components and leaves in the foliation.
Since the annuli in the collection $\mathscr{A}$ are nested, $\hat\C / \sim$ is isometric to a closed interval $I$, and we denote
$$
\pi: \hat\C \longrightarrow I
$$
as the projection map.

The dynamics of $f$ on $\hat\C$ determines an associated map on $I$ via
\begin{align*}
f_*:I &\longrightarrow I\\
x &\mapsto \pi\circ f(\partial \pi^{-1}(x))
\end{align*}
where $\partial \pi^{-1}(x)$ is the boundary of $\pi^{-1}(x)\subseteq \hat\C$.
By Theorem 3.6 in \cite{Shishikura89}, $f_*$ is well defined and continuous.
We will now prove some properties of the map $f_*$.

\begin{lem}\label{PropertyShishikuraTree}
Let $I = [a,b]$. Then there exist $a = a_1 < b_1 \leq a_2 < b_2 \leq... \leq a_k < b_k = b$ such that 
\begin{enumerate}
\item $f_*: [a_i, b_i] \longrightarrow I$ is a linear isometry with derivative $\pm d_i$;

\noindent moreover, $d_i \in \Z_{\geq 2}$ and the $\pm$ signs alternating;
\item $U$ is a critical gap if and only if $U = \pi^{-1}((b_i, a_{i+1}))$;

\noindent moreover, $f_*([b_i, a_{i+1}])\subseteq \{a,b\}$;
\item $d := \deg(f) = \sum_{i=1}^k d_i$, and $\sum_{i=1}^k 1/d_i <1$.
\end{enumerate}
\end{lem}
\begin{figure}[h]
\centering
\includegraphics[width=4.5cm]{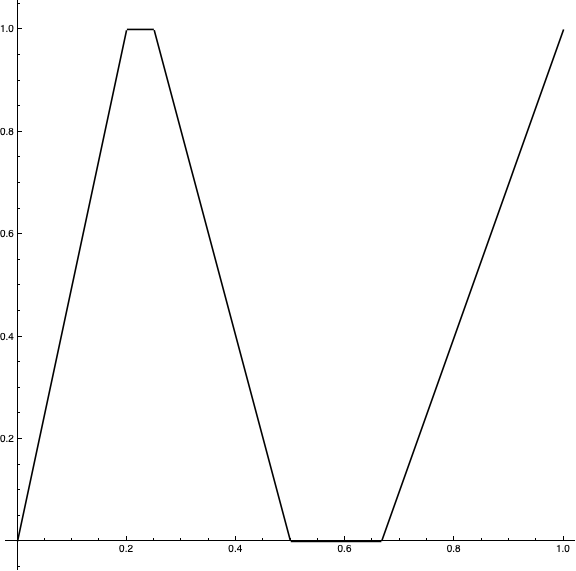}
\caption{An example of the graph of the map $f_*: I \longrightarrow I$. It has invariants $k=3$, $d_1 = 5, d_2 = 4, d_3 = 3$.}
\label{ShishikuraTreeExample}
\end{figure}
\begin{proof}
Consider the annulus $A = \pi^{-1}((a,b))$. 
The boundary $\partial A$ equals to the two extremal Julia components.
Since the Julia set is nested, each component $A_i$ of $f^{-1}(A)$ is an annulus.
Let $a_i, b_i$ be the projection of the boundary $\pi(\partial A_i)$.
We get $a_1<b_1\leq a_2 < b_2 \leq ...\leq a_k<b_k$.

Let $d_k = \deg(f:A_i \longrightarrow A)$. 
Since $A_i\subset A$, each $d_i \geq 2$.
Now restricting to $U \subseteq A_i$ and $U \in \mathscr{A}$, then $f(U) \in \mathscr{A}$. The map $f:U \longrightarrow f(U)$ is a degree $d_i$ covering, so $f_*$ restricting to $\pi(U)$ is linear with derivative $\pm d_i$.
This is true for any such $U$, so $f_*$ is a linear on $[a_i, b_i]$ with derivative $\pm d_i$.

Since the critical gap is mapped to a simply connected Fatou component whose boundary is contained in the extremal Julia component, and all the other gaps are mapped to other gaps, the property $(2)$ follows immediately.

For the last property, we note that the gaps are backward invariant. Hence if we pick an arbitrary gap, there are exactly $d$ preimages counted multiplicities. Therefore $d = \sum_{i=1}^k d_i$.

Note that $\sum_{i=1}^k 1/d_i \leq 1$.
If $\sum_{i=1}^k 1/d_i = 1$, then by the equality case of the Gr\"otzch inequality (see Theorem B.5 in \cite{McM94}), $A_i$ and $A_{i+1}$ share a Jordan curve boundary $\Gamma$.
Since $f(A_i) = f(A_{i+1}) = A$, we can choose simple closed curves $\gamma_i \subseteq A_i$ and $\gamma_{i+1} \subseteq A_{i+1}$ with $f(\gamma_i) = f(\gamma_{i+1})$.
Then $f$ has a critical point bounded between $\gamma_i$ and $\gamma_{i+1}$, so $f$ has a critical point on $\Gamma$.
This is a contradiction as $\Gamma \subseteq J$ and $f$ is hyperbolic, so $\sum_{i=1}^k 1/d_i <1$.
\end{proof}

Let $C =\bigcap_{i=1}^\infty f_*^{-n}([a,b]) \subseteq I$, and $P$ be the set of periodic points in $I$.  By Proposition \ref{PropertyShishikuraTree}, $C$ is a Cantor set and $P\subseteq C$.
\begin{cor}\label{CantorSetPeriodicSet}
The restriction $\pi$ on $J$ gives a surjective continuous map $\pi: J \longrightarrow C$ with connected fiber.
Moreover, the repelling periodic points of $f$ are contained in $\pi^{-1}(P)$.
\end{cor}
\begin{proof}
Since $\pi$ is a semi-conjugacy, it is easy to verify that $I-C = \pi(\cup_{A\in \mathscr{A}}A)$.
Hence $\pi:J \longrightarrow C$ is surjective and continuous.
The fiber $\pi^{-1}(t) \cap J$ is connected as otherwise, $\pi^{-1}(t)$ contains a gap which is a contradiction.
The moreover part follows directly from the fact $\pi$ is a semi-conjugacy. 
\end{proof}

Let $f_*: I = [a,b] \longrightarrow I$. Switching $a$ and $b$, and take the second iteration of $f$ if necessary, we may assume that $f_*(a) = a$.
Note if $k$ is even, then $f_*(b) = a$, and if $k$ is odd, then $f_*(b) = b$.

Recall that $g:U \longrightarrow V$ is called a {\em polynomial like map} if $g$ is a proper holomorphic map, and $\overline{U}\subseteq V$.
The degree of a polynomial like map is defined as the degree of the proper map.
The filled Julia set of the polynomial like map is defined as $K = \bigcap_{k=1}^\infty g^{-k}(V)$.
If $K$ is connected, then $g$ is quasiconformally conjugate (in fact, {\em hybrid conjugate}) to a polynomial $P$ of the same degree, which is unique up to affine conjugation.

For sufficiently small $\epsilon >0$, we have
$$
f_*([a,a+\epsilon)) = [a, a+d_1 \epsilon).
$$
Let $U = \pi^{-1}([a,a+\epsilon))$ and $V = \pi^{-1}([a, a+d_1 \epsilon))$.
Since $d_1 \geq 2$, $U,V$ are open sets with $\overline{U}\subseteq V$ and $f: U \longrightarrow V$ is proper of degree $d_1$.
Hence, $f: U \longrightarrow V$ is a polynomial like map, with connected filled Julia set $K = \pi^{-1}(a)$.
Let $P_a$ be the conjugate polynomial. 
Then $P_a$ is hyperbolic with connected Julia set as $f$ is hyperbolic.
Similarly, if $f_*(b) = b$, then we can associate a hyperbolic polynomial with connected Julia set $P_b$ to the end $b$.

Note that if $f$ varies in the hyperbolic component, $P_a$ (or $P_b$) also varies in the corresponding hyperbolic component of polynomials.
Hence, combining Lemma \ref{PropertyShishikuraTree}, we summarize the invariants that we can associate to a hyperbolic component with nested Julia sets in the following proposition.

\begin{prop}\label{NestJuliaSetInvariants}
Let $\HC$ be a hyperbolic component with nested Julia sets in $\M_d$, and $[f]\in 
\HC$. 
Taking the second iteration $f^2$ if necessary, we assume that $f$ fixes (at least) one of the two extremal Julia components.
We can associate the following set of invariants to $\HC$: 
\begin{enumerate}
\item A natural number $k \geq 2$, and a sequence $d_1,..., d_k$ such that $\sum d_i = d$, and $\sum 1/d_i < 1$, where these numbers are associated to $\HC$ as in Lemma \ref{PropertyShishikuraTree}.
\item If $k$ is even, a hyperbolic component $\HC_a\subseteq \Poly_{d_1}$ with connected Julia set;
\item If $k$ is odd, two hyperbolic components $\HC_a\subseteq \Poly_{d_1}$,  $\HC_b\subseteq \Poly_{d_k}$ with connected Julia set.
\end{enumerate}
\end{prop}

Next, we shall see that given any set of data as above, one can construct a hyperbolic component with nested Julia sets with that set of data as invariants.
We emphasize that this set of invariants, however, is not complete: different hyperbolic components may share the same invariant.
We refer the readers to \cite{L20} for a complete invariant.

\subsection*{Construction of hyperbolic component with nested Julia sets.}
In this subsection, we will use quasiconformal surgery to construct examples of nested Julia sets.
We start by reviewing some of the definitions.

Let $U\subseteq \hat \C$. Let $K \geq 1$, and set $k = \frac{K-1}{K+1}$. 
A map $\phi: U \longrightarrow \phi(U)$ is called {\em $K$-quasiconformal} if 
\begin{enumerate}
\item $\phi$ is a homeomorphism;
\item the partial derivatives $\partial_z \phi$ and $\partial_{\bar z} \phi$ exist in the sense of distributions and belong to $L^2_{loc}$ (i.e. are locally square integrable);
\item and satisfy $|\partial_{\bar z} \phi| < k |\partial_z \phi|$ in $L^2_{loc}$.
\end{enumerate}
A map $f: U \longrightarrow f(U)$ is called {\em $K$-quasiregular} if $f = g\circ \phi$, where $\phi$ is quasiconformal and $g$ is holomorphic.

The following fundamental lemma for quasiconformal surgery is known as the Shishikura's principle (see Lemma 1 in \cite{Shishikura87}):
\begin{prop}\label{ShishikuraPrinciple}
Let $g:\hat \C \longrightarrow \hat \C$ be a proper $K$-quasiregular map.
Suppose that there are disjoint open sets $E_i$ of $\hat\C$, quasiconformal maps $\Phi_i : E_i \longrightarrow E_i' \subseteq \hat\C$ ($i=1,..., m$) and integer $N \geq 0$, satisfying the following conditions:
\begin{enumerate}
\item $g(E) \subseteq E$, where $E = E_1 \cup ... \cup E_m$;
\item $\Phi \circ g \circ \Phi_i^{-1}$ is holomorphic in $E_i' = \Phi_i (E_i)$, where $\Phi: E \longrightarrow \hat\C$ is defined by $\Phi|_{E_i} = \Phi_i$;
\item $g_{\bar{z}} = 0$ a.e. on $\hat\C - g^{-N}(E)$.
\end{enumerate}
Then there exists a quasiconformal map $\varphi$ of $\hat\C$ such that $\varphi\circ g \circ \varphi^{-1}$ is a rational map. Moreover, $\varphi\circ \Phi_i^{-1}$ is conformal in $E_i'$ and $\varphi_{\bar{z}} = 0$ a.e. on $\hat \C - \bigcup_{n\geq 0} g^{-n}(E)$.
\end{prop}

We show every combinatorial invariant in Proposition \ref{NestJuliaSetInvariants} is realizable (cf. Theorem 1.1 and Theorem 1.2 in \cite{L20}):

\begin{theorem}\label{NestedMatingPoly}
Given the set of data as in Proposition \ref{NestJuliaSetInvariants}, there is a hyperbolic rational map $f$ with nested Julia sets having the set of data as its invariants.
\end{theorem}
\begin{proof}
Let $k \geq 2$, and a sequence $d_1,..., d_k$ with $\sum 1/d_i < 1$.

We first assume $k$ is an even number.
Let $P$ be a hyperbolic polynomial with connected Julia set.
We normalize $P$, and choose $R, \epsilon > 0$ so that
\begin{enumerate}
\item $\overline{B(0,1)}$ and $\hat\C-B(0,R)$ are contained in the bounded and unbounded Fatou components of $P$ respectively.
\item We have
$$
P(\partial B(0, R^{\epsilon+1/d_1})) \subseteq \hat\C-\overline{B(0,R)}.
$$ 
\item For $i = 2,..., k-1$, $A_i \subseteq B(0, R^{1-\epsilon -1/d_k}) - B(0, R^{\epsilon+1/d_1})$ is a round annulus centered at $0$ with modulus $\frac{\log R}{2d_i\pi}$.
The annuli $A_i$ is contained in the bounded component of $\Com - A_{i+1}$ with disjoint closures.
\end{enumerate}

By conjugating $P$ with an affine map and choosing $R$ large enough, the condition (1) can be achieved.
Given a sufficiently small $\epsilon>0$, the condition $(3)$ can be achieved for any $R$ as $\sum 1/d_i < 1$.
Fix $\epsilon > 0$ so that (3) is satisfied, by choosing $R$ large enough, the condition $(2)$ is achieved, as $P$ is a polynomial with degree $d_i$.

By condition (2), $P^{-1}(B(0,R)) \subseteq B(0, R^{\epsilon+1/d_1})$.
We first define 
$$
 g(z)=
\begin{cases}
P(z) & \text{ on } P^{-1}(B(0,R)),\\
C_i z^{(-1)^{i+1}d_i} & \text{ on } A_i, i =2,..., k-1,,\\
C_k z^{-d_k} & \text{ on } \hat\C - \overline{B(0, R^{1-1/d_k})},
\end{cases}
$$
where $C_i > 0$, $i < k$ is chosen so that $g(A_i) = B(0,R) - \overline{B(0,1)}$,
and $C_k>0$ is chosen so that $g(\partial B(0, R^{1-1/d^k})) = \partial B(0, R)$.
The construction is possible as $A_i$ is a round annulus centered at $0$ with modulus $\frac{\log R}{2d_i\pi}$.

By our construction, $g(\hat\C - \overline{B(0, R)}) = B(0, 1)$.
Let 
$$
U := P^{-1}(B(0,R)) \cup (\bigcup_{i=2}^{k-1}A_i) \cup(\hat\C - \overline{B(0, R^{1-1/d_k})}),
$$
be the domain of $g$.
Then $g$ extends continuously to $\overline U$.

Let $E_1,..., E_{k-1}$ be components of $\hat\C - \overline{U}$ with the natural ordering.
Then each $E_i$ is an annulus, and $g(\partial E_i) = \partial B(0,1)$ or $g(\partial E_i) = \partial B(0,R)$ depending $i$ is even or odd.
We extend $g$ to a proper quasiregular map $g: E_i \longrightarrow B(0,1)$ or $g: E_i \longrightarrow \hat\C - \overline{B(0,R)}$ respectively.
Note that $g$ is a degree $d_i$ and $d_{i+1}$ covering on the two boundary components of $\partial E_i$.
Thus, by the Riemann-Hurwitz formula, $E_i$ contains $d_i+ d_{i+1}$ critical points counted multiplicities.
In this way, we constructed a quasiregular map $g: \hat\C \longrightarrow \hat\C$.

Let $E_0 = \bigcup_{i=0}^\infty P^i(B(0,1))$ and $E_k = \hat\C - \overline{B(0, R)}$, and $E = \bigcup_{i=0}^k E_i$.
Note that since $P$ is hyperbolic with connected Julia set, and $\overline{B(0,1)}$ is contained in a bounded Fatou component of $P$, $E_0$ is disjoint from $E_j$ for all $j=1,..., k$.
By construction, $g(E) \subseteq E$.

Define $\Phi_0$ and $\Phi_k$ as the identity map.
Using the quasiregular map $g$, we can pull back the standard complex structure on $E_0$ and $E_k$ to $\mu_i$ on $E_i$, and let $\Phi_i$ be the corresponding quasiconformal map for $\mu_i$.
By construction, $\Phi \circ g \circ \Phi_i^{-1}$ is holomorphic.
Note that $g$ is holomorphic on $\hat \C - E$. Thus we can apply Proposition \ref{ShishikuraPrinciple} and conclude that $g$ is quasiconformally conjugate to a rational map $f$.

The map $f$ is hyperbolic as the critical points of $g$ are contained either in the bounded Fatou component of $P$ or one of $E_i$'s. 
Each $E_i$ is mapped to a bounded Fatou component of $P$ under the second iteration.
It can also be easily verified that the Julia set $f$ is nested, and it has the invariants the set of invariants
$k, d_1,..., d_k$ and $P$.

Let $k$ be an odd number, and $P, Q$ be two hyperbolic polynomials with connected Julia set.
Denote $M_R(z) = \frac{R}{z}$, $R > 0$. We normalize $P, Q$, and choose $R, \epsilon > 0$ so that
\begin{enumerate}
\item $\overline{B(0,1)}, \hat\C - B(0,R)$ are contained in Fatou components of $P$ and of $M_R \circ Q \circ M_R^{-1}$.
\item We have
$$
P(\partial B(0, R^{\epsilon+1/d_1})) \subseteq \hat\C-\overline{B(0,R)}
$$ 
and 
$$
M_R \circ Q \circ M_R^{-1}(\partial B(0, R^{1-\epsilon -1/d_k})) \subseteq B(0, 1).
$$
\item For $i = 2,..., k-1$, $A_i \subseteq B(0, R^{1-\epsilon -1/d_k}) - B(0, R^{\epsilon+1/d_1})$ is a round annulus centered at $0$ with modulus $\frac{\log R}{2d_i\pi}$.
The annuli $A_i$ is contained in the bounded component of $\Com - A_{i+1}$ with disjoint closures.
\end{enumerate}
A similar argument shows that we can achieve the three conditions. Define
$$
 g(z)=
\begin{cases}
P(z) & \text{ on } P^{-1}(B(0,R)),\\
C_i z^{(-1)^{i+1}d_i} & \text{ on } A_i, i =2,..., k-1,\\
M_R \circ Q \circ M_R^{-1}(z) & \text{ on } (M_R \circ Q \circ M_R^{-1})^{-1}(\hat\C - B(0,1)),
\end{cases}
$$
where $C_i > 0$ is chosen so that $F(A_i) = B(0, R) - \overline{B(0,1)}$.
Then a similar interpolation construction gives a hyperbolic rational map with nested Julia sets associated with this set of invariants.
\end{proof}

\begin{figure}[h!]
\centering
\includegraphics[width=5cm]{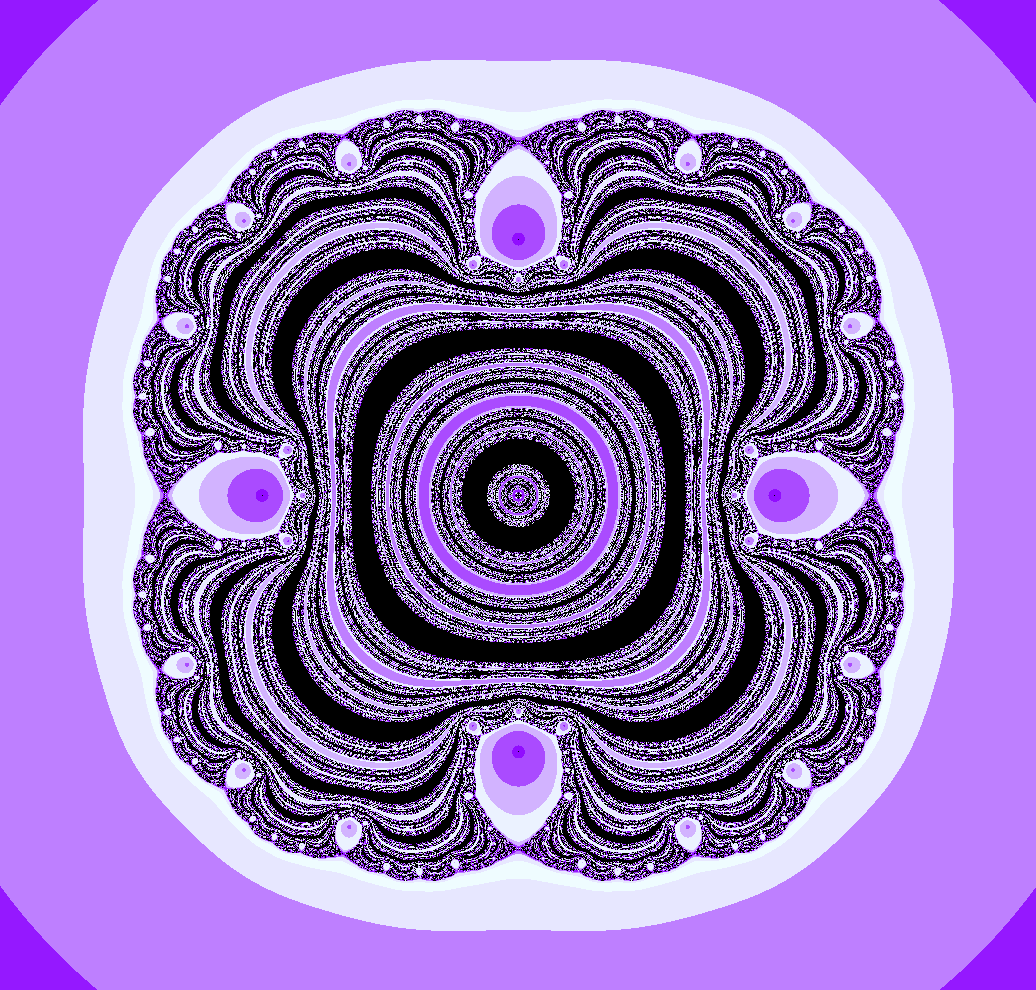}
\includegraphics[width=5cm]{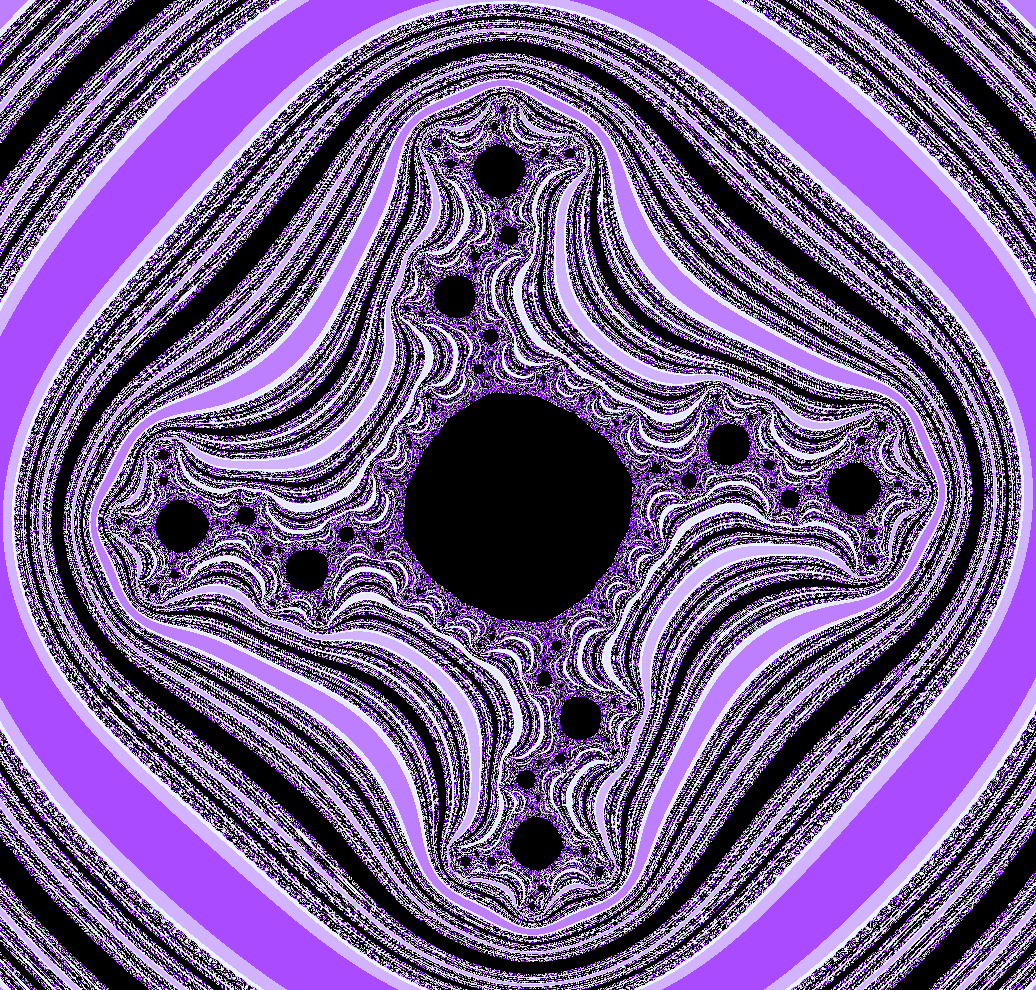}
\includegraphics[width=5cm]{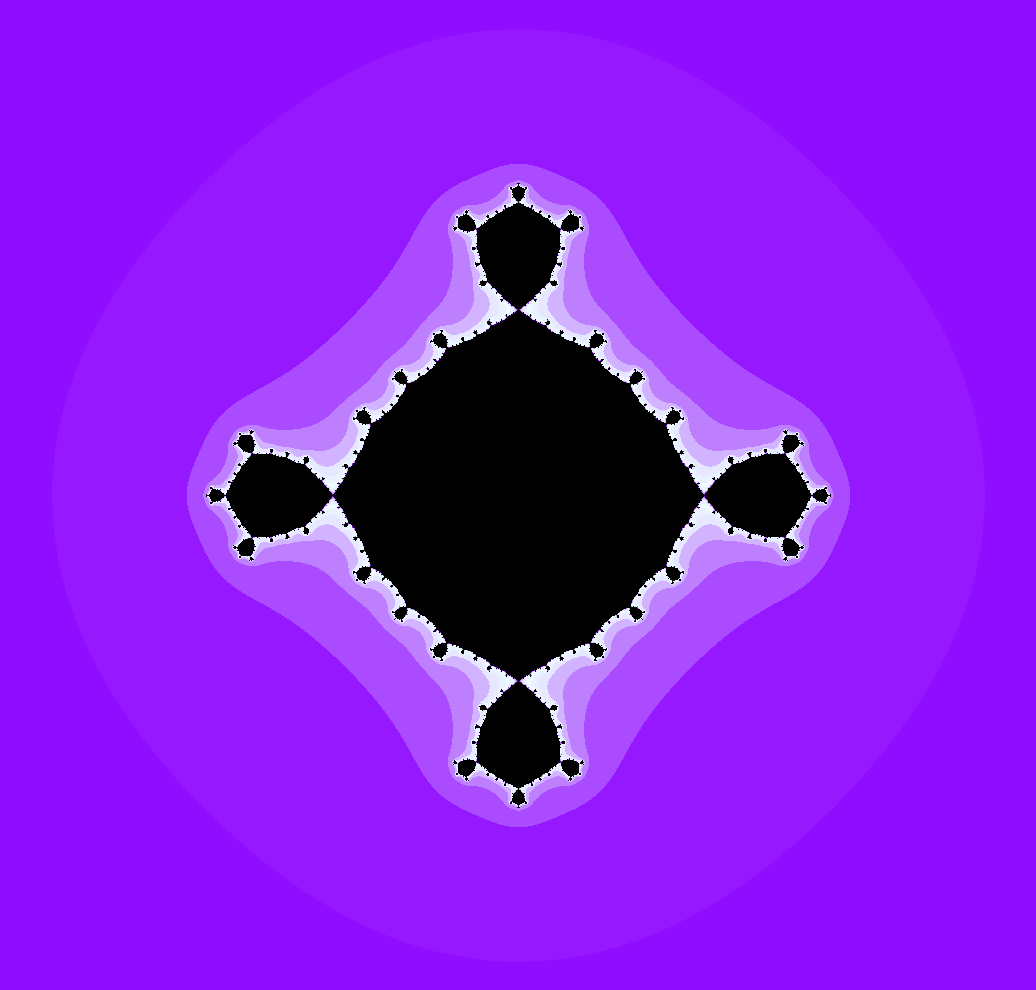}
\includegraphics[width=5cm]{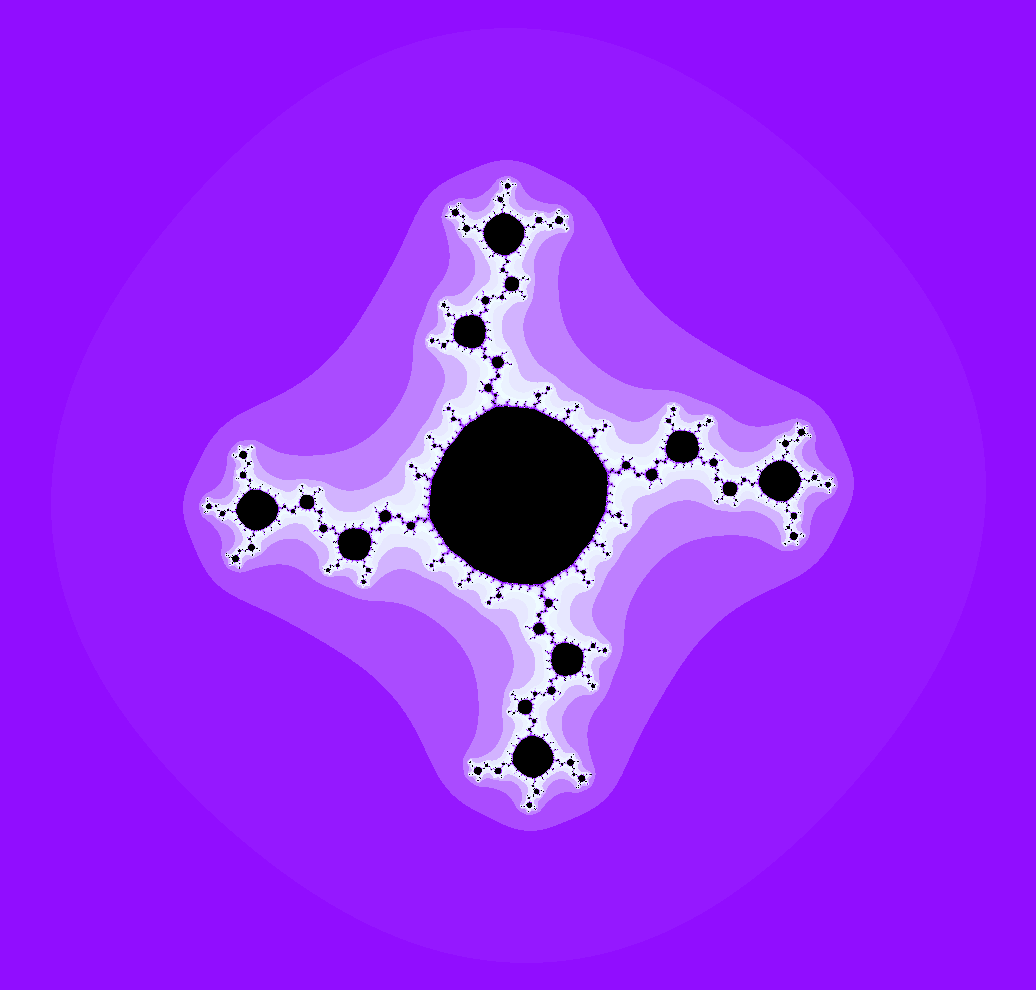}
\caption{An example of the Julia set of `nested mating' of $z^4-1$ (see lower left figure) and $z^4-1.10658-0.24848i$ (see lower right figure). The outermost Julia component (see upper left figure) is an inverted copy the Julia set of $z^4-1$, while the innermost Julia component (see upper right figure) is a copy of the Julia set of $z^4-1.10658-0.24848i$. The two component is separated by a Cantor set of nontrivial continuum, with countably many being coverings of either the Julia set of $z^4-1$ or $z^4-1.10658-0.24848i$.}
\end{figure}

\begin{figure}[h!]
\centering
\includegraphics[width=5cm]{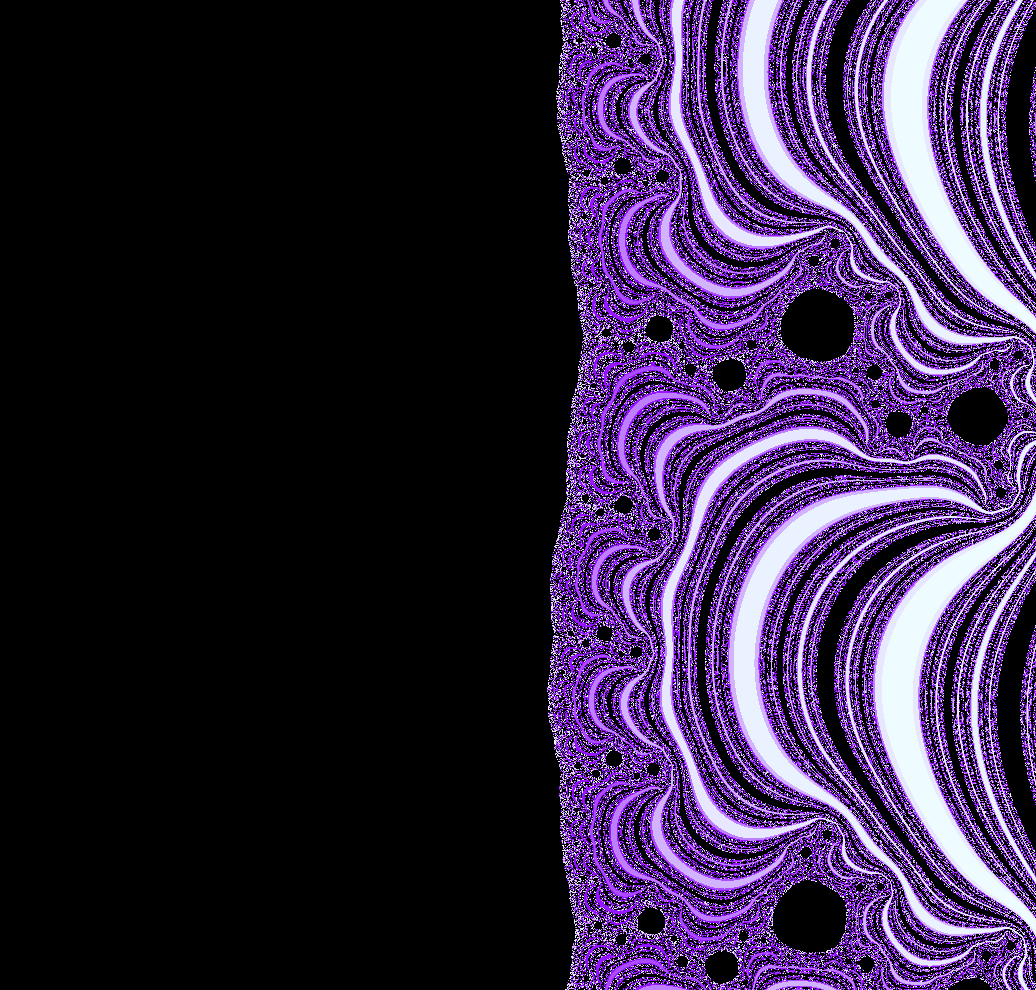}
\includegraphics[width=5cm]{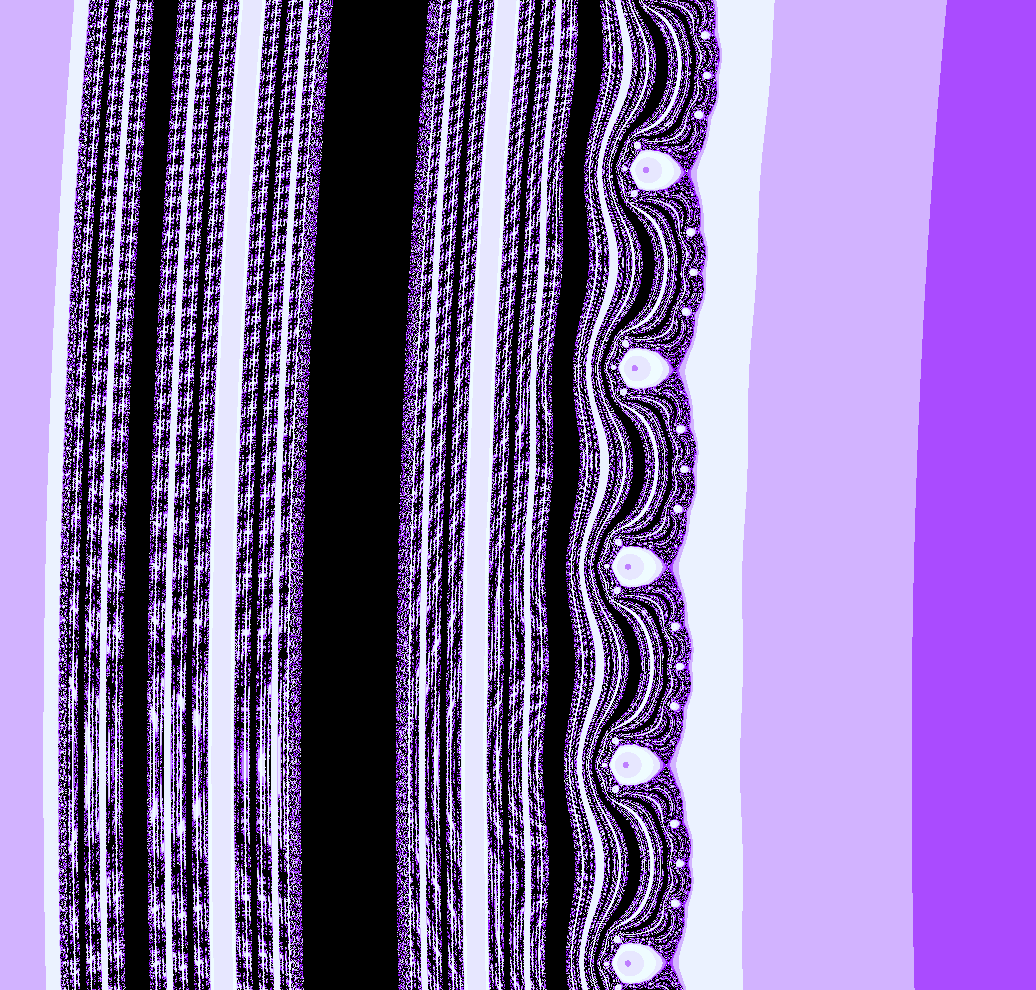}
\caption{A zoom near the boundary of the black region (see left figure) shows that it is a covering of the Julia set of $z^4-1.10658-0.24848i$, while a zoom near the boundary of the pink region (see right figure) shows that it is a covering of the Julia set of $z^4-1$.}
\end{figure}

We remark that in \cite{QiuYangYin15}, Qiu, Yang and Yin have a similar result for rational maps with Cantor set of circles.

%%%%%%%%%%%%%%
%%
%% Section 3
%%
%%%%%%%%%%%%%%

\section{Equivalence of bounded escape and nested Julia sets}\label{ProofOfCantorCircle}
Recall that a hyperbolic component $\HC$ is said to admit {\em bounded escape} if there exists a sequence $[f_n] \in \HC$ with markings $\phi_n$ so that
\begin{enumerate}
\item $[f_n]$ is degenerating;
\item For any periodic cycle $C\in \mathscr{S}$ of the topological model $\sigma: J \longrightarrow J$, the sequence of lengths $L(C, [f_n])$ is bounded.
\end{enumerate}
Note that by definition, the hyperbolic component $\HC$ is, in particular, unbounded.
In this section, we shall prove that a hyperbolic component $\HC$ has nested Julia sets if and only it admits bounded escape.

\subsection{Nested Julia sets $\implies$ bounded escape}
In this subsection, we shall prove the following using quasiconformal stretching.
\begin{theorem}\label{NestedJuliaSetImpliesBoundedEscape}
Let $\HC$ be a hyperbolic component with nested Julia sets. Then it admits bounded escape.
\end{theorem}
\subsection*{Linear stretch of an annulus}
Let $s< 1$, and let $A(s)$ denote the round annulus $B(0,1) - \overline{B(0,s)}$.
Given two round annuli $A(s)$ and $A(s')$, we define a {\em linear stretch} between them as
\begin{align*}
\psi_{s\to s'} : A(s) &\longrightarrow A(s')\\
(\rho,\theta) &\mapsto (\frac{1-s'}{1-s}\rho+\frac{s'-s}{1-s}, \theta)
\end{align*}
where $(\rho,\theta)$ is the polar coordinate.
We can extend the map continuously to unit disks $\psi_{s\to s'}: B(0,1) \longrightarrow B(0,1)$ by setting $\psi_{s\to s'}(z) = \frac{s'}{s}z$ on $\overline{B(0,s)}$.
Note that $\psi_{s\to s'}$ is $K$-quasiconformal, where $K = \max\{\frac{\log s}{\log s'}, \frac{\log s'}{\log s}\}$.

\subsection*{Stretch of a hyperbolic rational map}
Let $f$ be a hyperbolic rational map with nested Julia set, and let $A_1, ..., A_k$ be critical gaps of $f$. Then $D_i = f(A_i)$ is a disk Fatou component. Note that $D_i$ may equal to $D_j$.
Let $\phi_i: B(0,1) \longrightarrow D_i$ be a Riemann mapping.
Since $f$ is hyperbolic, there exist $s < 1$ so that $\phi_i(A(s))$ contains no post-critical points for all $i$.
We define $g_n: \hat \C \longrightarrow \hat\C$ by replacing $f$ on $A_i$, $i=1,..., k$ to
$$
g_n|_{A_i}:= \phi_i \circ \psi_{s \to e^{-n}} \circ \phi_i^{-1} \circ f|_{A_i}.
$$ 
Note that for all sufficiently large $n$, $g_n$ is $n/\log s$-quasiregular.

Let $E_{i,j} = f^j(A_i)$, and $E =  \bigcup_{j\geq 0} \bigcup_{i=1}^k E_{i,j}$. Since $A_i$ is pre-periodic, there are only finitely many components in $E$.
Define $\Phi_{i,j,n}$ as identity if $j \neq 0$.
Using the quasiregular map $g_n$ to pull back the standard complex structure on $E_{i,1}$ to $\mu_{i, 0, n}$ on $E_{i,0}$, and let $\Phi_{i,0,n}$ be the corresponding quasiconformal map for $\mu_{i,0,n}$.
Then by Proposition \ref{ShishikuraPrinciple}, $g_n$ is quasiconformally conjugate to a rational map $f_n$. We call $f_n$ a {\em quasiconformal stretch} of $f$.
Note that the modulus of $\Phi_{i,0,n}(A_i) = \Phi_{i,0,n}(E_{i,0})$ is going to infinity, so the modulus of any critical annulus of $f_n$ is going to infinity.

By construction, each $f_n$ is quasiconformally conjugate to $f$, so $[f_n] \in \HC$.
Let $A$ be a gap of $f$. Since $A$ is eventually mapped to a critical gap, the modulus of the corresponding annulus $A_n$ of $f_n$ is going to infinity.
This implies that $[f_n]$ is degenerating.
Indeed, otherwise, we can choose representatives and passing to a subsequence so that $f_n$ converges to a rational map $g$ of the same degree.
Since the modulus of $A_n$ is going to infinity, at least one component of $\hat\C - A_n$ is shrinking to a point. So infinitely many pre-periodic points collapse to a single point, giving a contradiction.

We are now ready to prove Theorem \ref{NestedJuliaSetImpliesBoundedEscape}.
\begin{proof}[Proof of Theorem \ref{NestedJuliaSetImpliesBoundedEscape}]
Let $[f]\in \HC$ be such a hyperbolic rational map.
Let $f_n$ be a quasiconformal stretch of $f$, with markings $\phi_n$.
Let $f_*: I = [a,b] \longrightarrow [a,b]$ be the corresponding map on the Shishikura tree.
Let $x$ be a periodic point of $f$ of period $p$.
We consider two cases:

Case (1): $x$ is on the extremal Julia component $K$.
After passing to a second iterate, we may assume $K$ is fixed by $f$ and $\pi(K) = a \in I$.
Choose $\epsilon > 0$ so that $[a, a+ d_1 \epsilon) \subseteq [a_1, b_1]$.
Let $U = \pi^{-1}([a, a+ \epsilon))$, $V = \pi^{-1}([a, a+ d_1 \epsilon))$.
Then $f: U \longrightarrow V$ is a polynomial like map.

Let $U_n = \phi_n(U)$ and $V_n = \phi(V)$. 
Then $f_n:U_n \longrightarrow V_n$ is a polynomial like map.
By construction of the quasiconformal stretch, we have $m(V_n - U_n) \to \infty$.
Since polynomial like map of degree $e$ $f:U\longrightarrow V$ with $m(V - U)$ bounded below is compact up to affine conjugacy (see Theorem 5.8 in \cite{McM94}), so $f_n$ converges compactly to $f_\infty : U_\infty \longrightarrow V_\infty$ of the same degree, so the multipliers of $\phi_n(x)$ stay bounded.

Case(2): $x$ is on the non-extremal Julia component $K$.
Then by Lemma \ref{TopologicalPropertyNestedJuliaSet}, $K$ is a Jordan curve and is buried.
Let $\pi(K) = t \in (a_i, b_i) \subseteq I$.
Let $e$ be the norm of the derivative of $(f_*)^p$ at $t$.
Choose $\epsilon > 0$ so that $(t - e \epsilon, t+ e \epsilon) \subseteq (a_i, b_i)$.
Let $U = \pi^{-1} ((t - \epsilon, t+ \epsilon))$ and $V = \pi^{-1}((t - e \epsilon, t+ e \epsilon))$.
Then $U \subseteq V$ and $f^p: U \longrightarrow V$ is a degree $e$ covering map.

Let $U_n = \phi_n(U)$ and $V_n = \phi_n(V)$. 
By construction of the quasiconformal stretch, the moduli of the two annuli $U_n - \phi_n(K)$ go to infinity.
We normalize so that $V_n$ separates $0, \infty$, and $\phi_n(x) = 1$.
Let $A_n$ be an annuli component of $U_n - \phi_n(K)$.
Without loss of generality, we may assume that $A_n$ is in the unbounded component of $\C - \phi_n(K)$.
Then $A_n$ separates $\{0,1\}$ from $\infty$.
Since $m(A_n) \to \infty$, $\partial A_n - \phi_n(K) \to \infty$ (see Theorem 4.7 in \cite{A73}).
Similarly, the other component of $\partial U_n$ converges to $0$.
Hence, both $U_n$ and $V_n$ converges to $\Com-\{0\}$.
After passing to a subsequence, $f_n^p$ converges compactly on $\Com-\{0\}$ to a non-constant rational map.
So the multiplier at $\phi_n(x)$ is bounded in the subsequence.

Since the set of periodic cycles is countable, a diagonal argument allows us to pass to a subsequence $f_{n_k}$ and conclude $\HC$ admits bounded escape.
\end{proof}

%%%%%%%%%%%%%%
%%
%% Section 4
%%
%%%%%%%%%%%%%%

\subsection{Bounded escape $\implies$ nested Julia sets}
We shall now prove the converse of Theorem \ref{NestedJuliaSetImpliesBoundedEscape}.
Let $\HC\subseteq \M_d$ be a hyperbolic component that admits bounded escape.
Let $\sigma :J \longrightarrow J$ be the topological model for the actions on the Julia set for $\HC$,
and $\mathscr{S}$ be the set of periodic cycles of $\sigma$.
Then we have a sequence $[f_n] \in \HC$ with markings $\phi_n$ which is degenerating, and for any $C\in \mathscr{S}$, $L(C, f_n)$ stay bounded.
Recall that in \S \ref{MandL}, we have constructed a limiting dynamics
$$
F : {^r\Hyp^3} \longrightarrow {^r\Hyp^3},
$$
with the rescalings $r_n = r([f_n])$.
The markings provide a marking on the end of the tree (see \S \ref{MandL}), so each periodic cycle $C\in \mathscr{S}$ represents a cycle of periodic ends for $F$ on ${^r\Hyp^3}$.

The proof consists of two steps. We first classify the limiting dynamics that can appear.
We then use the limiting dynamics to deduce the topological properties on the Julia set.

\subsection*{Classification of limiting map $F$}
Recall a periodic end $\alpha$ is {\em repelling} if $L(\alpha, F) > 0$.
Suppose $F$ has a cycle of repelling periodic ends, then by Theorem \ref{TL}, it is represented by $C\in \mathscr{S}$ with $\lim_\omega L(C, f_n) = \infty$, which is a contradiction.
Hence, $F$ has no repelling periodic ends.
We shall first classify those limiting dynamics with no repelling periodic ends.

\begin{lem}\label{RepellingEnd}
Let $F: {^r\Hyp^3} \longrightarrow {^r\Hyp^3}$, and $x_0, x_1 \in {^r\Hyp^3}$.
Let $\vec{v_0} \in T_{x_0}{^r\Hyp^3}$ associated to $x_1$, and $\vec{v_1} \in T_{x_1}{^r\Hyp^3}$ such that
\begin{enumerate}
\item $F (x_1) = x_0$;
\item $D_{x_1}F(\vec{v_1}) = \vec{v_0}$;
\item $U_{\vec{v_1}}$ does not intersect the critical tree nor contain $x_0$.
\end{enumerate}
Then $F$ has a repelling fixed end.
\end{lem}
\begin{proof}
Since $U_{\vec{v_1}}$ does not intersect the critical tree, $F$ is an isometric bijection from $U_{\vec{v_1}}$ to its image $U_{\vec{v_0}}$.
Since $U_{\vec{v_1}}$ does not intersect $x_0$, $U_{\vec{v_1}}\subseteq U_{\vec{v_0}}$.

Let $x_2$ be the preimage of $x_1$ in  $U_{\vec{v_1}}$. Then $d(x_0,x_1) = d(x_1,x_2)$, and 
$x_2 \in U_{\vec{v_0}}$.
Hence, we can define $x_n$ inductively by taking the preimage of $x_{n-1}$ in $U_{\vec{v_1}}$.
The union of the geodesic segments $\alpha:=\cup_{k=0}^\infty [x_k, x_{k+1}]$ is an end which is fixed by $F$.
It is repelling as $L(\alpha, F) = d(x_0, x_1) > 0$.
\end{proof}

The following lemma follows from our Theorem \ref{ConnectionToBerkovichDynamics} and Theorem 10.83 in \cite{BakerRumely10} (see also Proposition 9.3 in \cite{Rivera-Letelier05} and Lemma 6.2 in \cite{Rivera-Letelier03}). %For completeness, we produce a proof here as well.
\begin{lem}\label{FixedPointMultiplicityGEQ2}
Assume the limiting map $F: {^r\Hyp^3} \longrightarrow {^r\Hyp^3}$ has no repelling periodic ends.
Then it has a fixed point $x\in {^r\Hyp^3}$ of local degree $\deg_x F \geq 2$.
\end{lem}

\begin{prop}
Assume the limiting map $F : {^r\Hyp^3} \longrightarrow {^r\Hyp^3}$ has no repelling periodic ends.
Let $x\in {^r\Hyp^3}$ be a fixed point of multiplicity $\geq 2$ (which exists by Lemma \ref{FixedPointMultiplicityGEQ2}). Then the set
$$
P = \bigcup_{i=0}^\infty F^{-i}(x)
$$
is contained in a geodesic segment.
\end{prop}
\begin{proof}
By Theorem \ref{DGLC}, there is no totally invariant point. 
Hence, the preimage of $x$ contains more than $1$ point.
Note it suffices to show $P$ is contained in a line.
Indeed, if we prove this, and $P$ escapes to one end, then replace $F$ by its second iterate if necessary, we get a repelling fixed end which is a contradiction.

We will now argue by contradiction to prove $P$ is contained in a line.
Suppose not. Then there are two points $y, y'$ which are eventually mapped to $x$, and the convex hull of $x, y, y'$ forms a `tripod'. 
Replacing $F$ by its iterate if necessary, we may assume that 
$$
F(y) = F(y') = x.
$$ 
Let $\vec{v}$ be the tangent vector at $x$ associated to $y$.
There are two cases to consider:

Case $(1)$: The preimage of $\vec{v}$ under $D_{x}F$ in $T_x {^r\Hyp^3}$ is infinite.
Then we can construct a `fan' as follows (see Figure \ref{Fan}).
Using Lemma \ref{ExtendTangent}, we construct $z_0 = y, z_1,..., z_n,...$ inductively so that $F$ sends $[x, z_{n+1}]$ homeomorphically to $[x, z_n]$.
Let $\vec{w}_k$ denote the tangent vector at $x$ associated to $z_k$.
Let $\vec{u}_{0,k}$ be the tangent vectors at $z_0 = y$ which is mapped to $\vec{w}_k$ 
(there might be many such vectors, if that's the case, we just choose one).

Inductively, we let $\vec{u}_{n,k}$ be vector at $z_n$ which is mapped to $\vec{u}_{n-1, k}$.
Note that the components $U_{\vec{u}_{n,k}}$ are all disjoint. 
Since the critical tree for $F$ is a finite tree, there is a $K$, such that for all $k\geq K$ and all $n$, the component $U_{\vec{u}_{n,k}}$ does not intersect the critical locus.
Since $\vec{u}_{n,K}$ is mapped to $\vec{u}_{n-1,K}$, 
$F^{K+1}$ is an isometric bijection from $U_{\vec{u}_{K,K}}$ to its image $U_{\vec{w}_K}$ by Lemma \ref{IsomAwayCritTree}.
Since the critical tree intersects $[x,y]$ (as $F$ is not injective on $[x,y]$), $x \notin U_{\vec{u}_{K,K}}$.
Now by Lemma \ref{RepellingEnd}, we conclude that there is a repelling periodic end of period $K+1$, which is a contradiction.

\begin{figure}[h!]
\centering
\includegraphics[width=10cm]{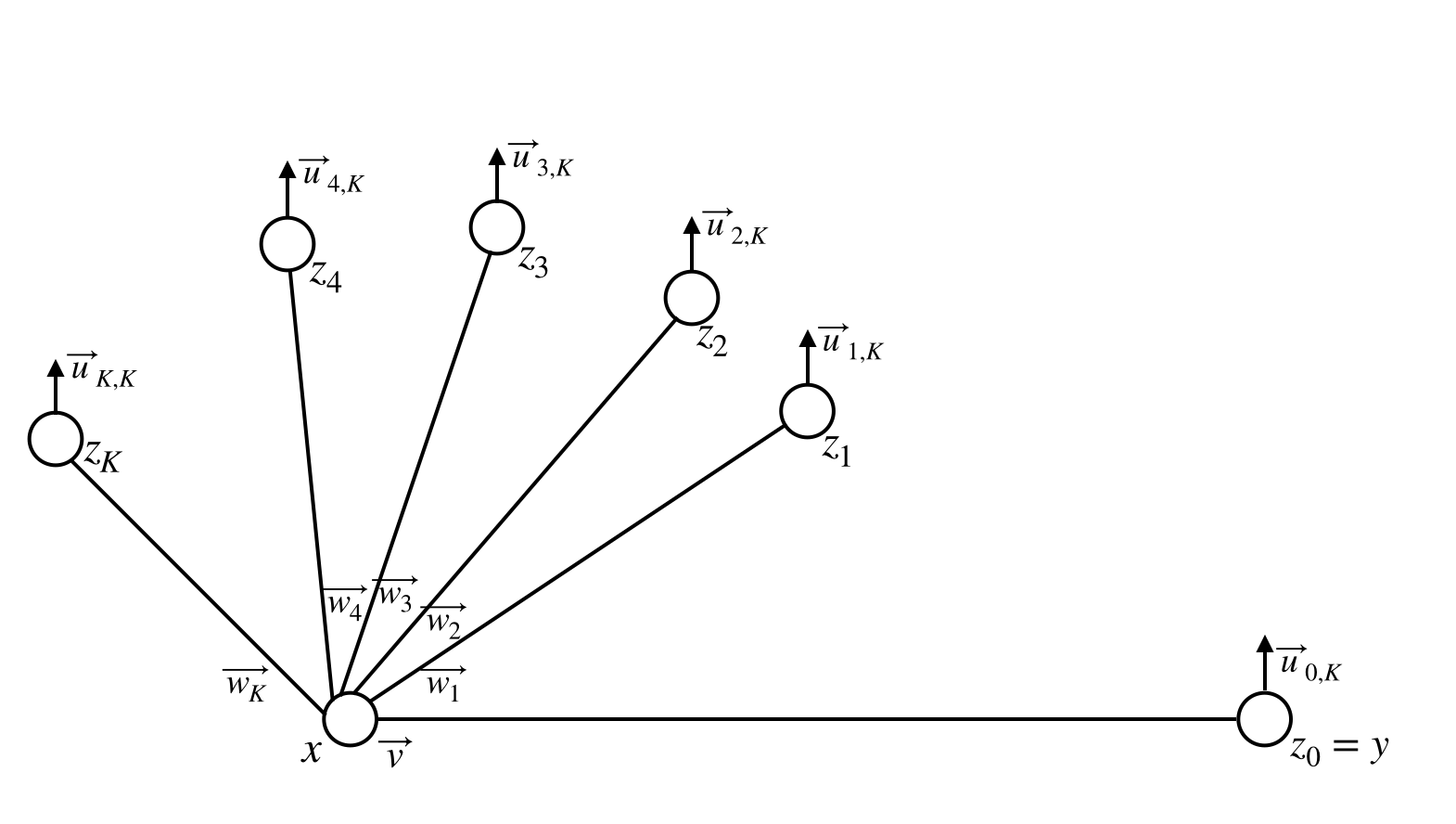}
\caption{The `fan' for Case $(1)$.}
\label{Fan}
\end{figure}

Case $(2)$: The preimage of $\vec{v}$ under $D_{x}F$ in $T_x {^r\Hyp^3}$ is finite.
We remark that the idea is similar to the previous case, but the notations and indices become more complicated.
By Theorem \ref{CritVec}, replacing $\E(f_n)$ by its second iterate if necessary, we may assume $\vec{v}$ is totally invariant under $D_{x}F$.
Hence $F$ is (locally) expanding in the direction $\vec{v}$ by Theorem \ref{CritVec}.
Let $p$ be the branched point of the tripod $\hull(x,y,y')\subseteq {^r\Hyp^3}$.
Using Lemma \ref{ExtendTangent}, we choose $q\in U_{\vec v}$ so that $F$ maps $[x,q]$ homeomorphically to $[x,p]$.
We also choose $z_1, z_1'$ so that $F$ maps $[q, z_1]$ and $[q, z_1']$ homeomorphically to $[p, y]$ and $[p, y']$ respectively.

We claim that at least one of $[x, z_1]$ and $[x, z_1']$ branches off in $[x,p]$.
Indeed, since $F$ is (locally) expanding in the direction $\vec{v}$, $d(x, q) < d(x, p)$, so at least one of $[x, z_1]$ and $ [x,z_1']$ must branch off $[x,p]$ (note that it may even happen before $q$).
We assume $[x, z_1]$ branches off $[x,p]$.% before $[x, z_1']$.

We denote $z_0 = y$, $q_1$ as the branched point of $\hull(x,z_0,z_1)$.
Note that $[x,F(q_1)]\subseteq [x,p] \subseteq [x, z_0]$.
We construct a generalized `fan' using Lemma \ref{ExtendTangent} as follows (see Figure \ref{GeneralizedFan}).

First, by inductively taking preimages of $F:[x,q_1] \longrightarrow [x,F(q_1)]$, we construct $q_1,...q_n,...$ so that $F$ sends $[x, q_{n+1}]$ homeomorphically to $[x, q_n]$.
Since $F$ is (locally) expanding in the direction $\vec{v}$ and $[x,F(q_1)]\subset [x,p]$, $q_n \in [x,p]$ and $d(x, q_n) < d(x, q_{n-1})$ and $d(x, q_n) \to 0$.

We construct $z_0 = y, z_1,..., z_n,...$ inductively so that $F$ sends $[q_{n+1}, z_{n+1}]$ homeomorphically to $[q_n, z_n]$.
Let $\vec{w}_k$ denote the tangent vector at $q_k$ associated to $z_k$.
We define $q_{0,k}$ so that $F$ maps $[z_0, q_{0,k}]$ homeomorphically to $[x, q_k]$.
Let $\vec{u}_{0,k}$ be a tangent vector at $q_{0,k}$ which is mapped $\vec{w}_k$.

Since the critical tree is finite, for large $n$, the component $U_{\vec{w}_n}$ does not intersect the critical locus.
Thus $F$ is eventually an isometry between $[q_n, z_n]$ and $[q_{n-1}, z_{n-1}]$, so there exists $\epsilon > 0$ so that the length of $[q_n, z_n]$ is at least $\epsilon$ for all $n$.
Since $d(x, q_k) \to 0$, for sufficiently large $k$, we can construct $q_{1,k} \in U_{\vec{w}_1}$ so that $F$ maps $[z_1, q_{1,k}]$ homeomorphically to $[z_0, q_{0,k}]$. 
Inductively, for sufficiently large $k$, we define $q_{n,k} \in U_{\vec{w}_n}$ so that $[z_n, q_{n,k}]$ is mapped homeomorphically to $[z_{n-1}, q_{n-1,k}]$.
We may assume that $q_{n,k} \in [z_n, q_{n, k-1}]$.
Let $\vec{u}_{n,k}$ be a tangent vector at $q_{n,k}$ which is mapped $\vec{u}_{n-1,k}$.

The argument is now similar to Case $(1)$.
Note that the components $U_{\vec{u}_{n,k}}$ are disjoint. 
Since the critical tree for $F$ is a finite tree, there is a $K$, such that for all $k\geq K$ and all $n$, the component $U_{\vec{u}_{n,k}}$ does not intersect the critical locus.
Since $\vec{u}_{n,K}$ is mapped to $\vec{u}_{n-1,K}$, 
$F^{K+1}$ is an isometry from $U_{\vec{u}_{K,K}}$ to its image $U_{\vec{w}_K}$.
Since the critical tree intersect $[x,y]$, and $q_{K,K} \in U_{\vec{w}_K}$, so $q_K \notin U_{\vec{u}_{K,K}}$.
Now by Lemma \ref{RepellingEnd}, we conclude that there exists a repelling periodic end of period $K+1$, which is a contradiction.
\begin{figure}[h!]
\centering
\includegraphics[width=10cm]{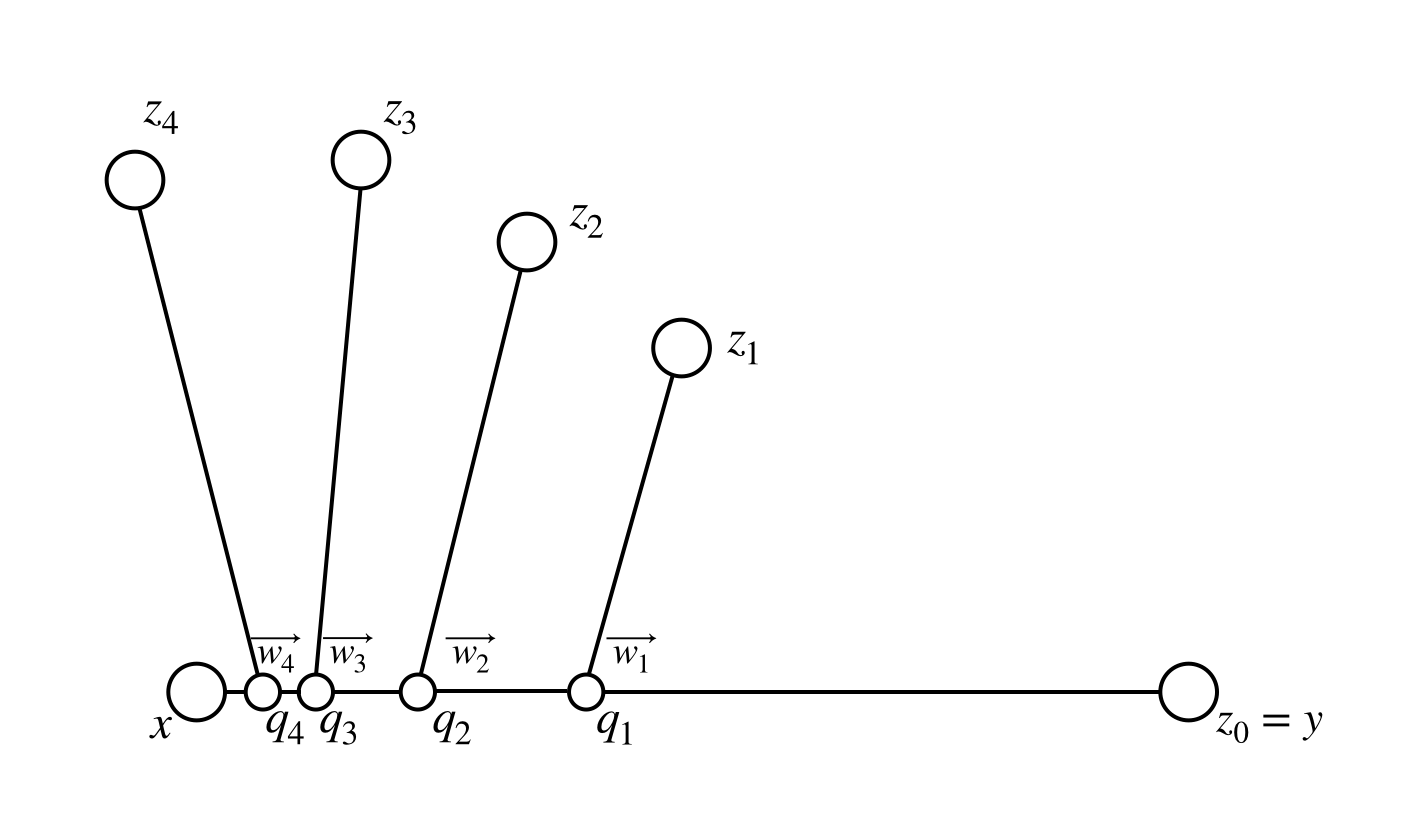}
\caption{The generalized `fan' for Case $(2)$.}
\label{GeneralizedFan}
\end{figure}
\end{proof}

Let $I = [a,b]$ be the smallest geodesic segment containing $P$. 
Then $F$ sends the boundary $\{a, b\}$ to the boundary $\{a, b\}$. As otherwise, we can find a point in $P$ with preimage outside of $[a,b]$, which is a contradiction.

Let $J\subseteq I$ be the closure of a component of $F^{-1}(\Int(I)) \cap I$. 
Then $F$ maps $J$ homeomorphically to $I$.
Indeed, if the map is not injective, then there is a point $t\in \Int(J)$ with tangent vectors $\vec{v}_1, \vec{v}_2$ at $t$ associated to $a$ and $b$ respectively so that 
$$
D_t F (\vec{v}_1) = D_t F (\vec{v}_2).
$$
But $D_t F$ is surjective by Theorem \ref{CritVec}, so there is a tangent vector $\vec{v}$ which is mapped to the tangent vector at $F(t)$ associated to either $a$ or $b$. 
This means that $P$ intersects non-trivially with $U_{\vec{v}}$, which is a contradiction.
The map $F$ is surjective by a similar argument: if $F$ is not surjective and let $I'\subseteq I$ be the image.
Note that $I'$ is a subinterval of $I$, so $P$ intersects $I-I'$ non-trivially.
By Theorem \ref{CritVec} and Lemma \ref{ExtendTangent}, the preimage of $I-I'$ is not contained in $I$, which is a contradiction.

We also note that $F$ has constant derivative on $J$. Indeed, if not, then we can find a point $t\in J$ with tangent vectors $\vec{v}_1, \vec{v}_2$ at $t$ associated to $a$ and $b$ respectively so that the local degrees at $\vec{v}_1$ and $\vec{v}_2$ are different.
Then applying Theorem \ref{CritVec}, one of $D_t F (\vec{v}_i)$ has a preimage $\vec{v}$ in $T_t {^r\Hyp^3}$ other than $\vec{v}_i$. 
Then there is a point of $P$ in $U_{\vec{v}}$ by Lemma \ref{ExtendTangent}, which is a contradiction.

By looking at the local degrees at the preimages of the point $x$, we conclude the sum of the derivatives on different components $J$ equals $d$.

To summarize, we have the following proposition, which describes the limiting dynamics with no repelling periodic ends.

\begin{prop}\label{PropertyLimitingMap}
Assume the limiting map $F : {^r\Hyp^3} \longrightarrow {^r\Hyp^3}$ has no repelling periodic ends.
Let $x \in {^r\Hyp^3}$ be a fixed point with multiplicity $\geq 2$ and $P = \cup_{i=0}^{\infty} F^{-i}(x)$.
Let $I = [a,b]$ be the smallest geodesic segment that contains $P$. Then there exist $a = a_1 < b_1 \leq a_2 < b_2 \leq.. \leq a_k < b_k = b$ such that 
\begin{enumerate}
\item $F: [a_i, b_i] \longrightarrow I$ is a linear isometry with derivative $\pm d_i$ and $d_i \in \Z_{\geq 2}$ and the $\pm$ sign alternating;
\item $d = \sum_{i=1}^k d_i$.
\end{enumerate}
\end{prop}

An immediate corollary of the above proposition is the following:
\begin{cor}\label{cor:nim}
Let $t\in I = [a,b]$ which is mapped into $(a,b)$. Then $U_{\vec{v}}$ contains no critical ends for all $\vec{v}$ at $t$ not associated to $a$ or $b$.
\end{cor}
\begin{proof}
Assume $t \in (a_i, b_i) \subseteq I$, so the norm of the derivative of $F$ at $t$ is $d_i$.
We first note that the local degree of $D_tF$ at $\vec{v}$ is $1$.
Otherwise, by counting critical points of $D_t F$, we conclude that the degree of $D_tF$ is $> d_i$.
Let $\vec{w}_a, \vec{w}_b \in T_t{^r\Hyp^3}$ and $\vec{w}'_a, \vec{w}'_b \in T_{F(t)}{^r\Hyp^3}$ correspond to $a, b$. Then by Theorem \ref{CritVec}, there exists $\vec{w} \neq \vec{w}_a, \vec{w}_b$ so that $D_tF(\vec{w}) = \vec{w}'_a$.
By Lemma \ref{ExtendTangent}, there exists $x\in U_{\vec{w}}$ mapped to $a$. 
This is a contradiction as $I$ is backward invariant.

Suppose $U_{\vec{v}}$ contains a critical ends. Since the local degree of $D_tF$ at $\vec{v}$ is $1$, Proposition 9.41 in \cite{BakerRumely10} gives that $F(U_{\vec{v}}) = {^r\Hyp^3}$. Again, this is a contradiction as $I$ is backward invariant.
\end{proof}

\begin{rmk}
We remark that we did not use the fact that $f_n$ comes from the same hyperbolic component yet.
In our communication with Favre, a similar classification also appears in an unpublished manuscript by Favre and Rivera-Letelier. We would refer to \cite{FR10} where many such examples are studied.

We remark the similarities and the distinctions of the classification with the induced map on the Shishikura's tree.
In Proposition \ref{PropertyLimitingMap}, it is possible $\sum_{i=1}^k 1/d_i = 1$, which cannot occur in Lemma \ref{PropertyShishikuraTree}.
We shall see that if we assume $[f_n]$ are contained a hyperbolic component $\HC$, then $\sum_{i=1}^k 1/d_i < 1$.
\end{rmk}

\subsection*{From limiting map $F$ to nested Julia sets}
We start with the following lemma. The proof uses standard expansion and pullback argument. 
For completeness, we sketch the proof and refer the readers to \S 5 in \cite{PilgrimTan00} for a more general setup.
\begin{lem}\label{lem:sal}
Let $A_1, A_0 \subseteq \C$ be two annuli with  $\overline{A_1} \subseteq A_0$, and let $g: A_1 \longrightarrow A_0$ be a degree $e$ covering with $e\geq 2$. Then the non-escaping set $J = \bigcap_{k=0}^\infty \overline{g^{-k}(A_0)}$ is a Jordan curve.
\end{lem}
\begin{proof}
Let $A_n = g^{-n}(A_0)$, and $A_n^\pm$ be two annuli component of $A_n - A_{n-1}$ so that $A_n^-$ is in the bounded component of $\C - A_n^+$.
By Schwarz lemma, with respect to the hyperbolic metric on $A_1$ and $A_0$ in the domain and the codomain, the inclusion map $i: A_1 \longrightarrow A_0$ is a contraction, and the covering map $g: A_1 \longrightarrow A_0$ is an isometry.
Thus, the map $g: A_1 \subseteq A_0 \longrightarrow A_0$ is expanding with respect to hyperbolic metric $\rho$ on $A_0$.
Since $A_2$ has compact closure in $A_1$, $g$ expands $\rho$ a definite factor $\lambda>1$ on $A_2$.

Let $\delta_n^{\pm}$ be the component of $\partial A_n$ so that $\delta_n^-$ is bounded by $\delta_n^+$.
By expansion, there exists a constant $M$ so that $\max_{y\in \delta_n^+} \min_{x\in J} d_\rho(x,y) \leq M/\lambda^n$ and $\max_{x\in J} \min_{y\in \delta_n^+} d_\rho(x,y) \leq M/\lambda^n$ (see Lemma 2 in \S 5 of \cite{PilgrimTan00}), and similarly for $\delta_n^-$.
Thus, $\hat\C - J$ has exactly two components $U_1, U_2$, and $\partial U_1 = \partial U_2 = J$.

Let $\gamma_1^\pm$ be smooth simple closed curves that generate the fundamental group of $A_1^\pm$, with parameterization $\zeta_1^\pm: S^1 \longrightarrow \gamma_1^\pm$.
Inductively, let $\gamma_n^\pm \subseteq A_n^\pm$ be the lift of $\gamma_{n-1}^\pm$.
Choose a homotopy between $\gamma_2^\pm$ and $\gamma_1^\pm$.
By lifting this homotopy inductively, we construct parameterizations $\zeta_n^\pm: S^1 \longrightarrow \gamma_n^\pm \subseteq A_0$ (see Lemma 0 in \S 5 of \cite{PilgrimTan00}).
By expansion, $\zeta_n^\pm$ converges uniformly to continuous maps $\zeta_\infty^\pm:S^1 \longrightarrow A_0$ (see Lemma 1 in \S 5 of \cite{PilgrimTan00}).
Since $\gamma_n^\pm$ is squeezed between $\delta_n^\pm$, the image $\zeta_\infty^\pm(S_1) = J$. In particular, $J$ is locally connected.
By Lemma 5.1 in \cite{PilgrimTan00}, $J$ is a Jordan curve.
\end{proof}

Let  $[f_n] \in \HC$ with markings $\phi_n$ that gives the bounded escape.
Let $F: {^r\Hyp^3} \longrightarrow {^r\Hyp^3}$ be the limiting map.
Let $I = [a, b]$ be the geodesic segment as in Proposition \ref{PropertyLimitingMap}, and $x$ be a fixed point in $(a,b)$.
Then there exists an open set $U^{x-t, x+t}\subseteq {^r\Hyp^3}$ with boundary points $x-t$ and $x+t$ which is mapped to $U^{x-et, x+et}$ for some integer $e \geq 2$.
By Corollary \ref{cor:nim}, $U^{x-t, x+t}$ contains no critical ends.

Let $U_n, V_n$ be sequences of annuli approximating $U^{x-t, x+t}, U^{x-et, x+et}$ as in Lemma \ref{ApproxAnnuli}.
Then $f_n:U_n \longrightarrow V_n$ is a degree $e$ covering $\omega$-almost surely.
Since $U^{x-t, x+t} \subseteq U^{x-et, x+et}$, by Lemma \ref{lem:cap}, $\overline U_n\subseteq V_n$ $\omega$-almost surely.

Let $N$ be in the $\omega$-big set so that the above holds.
Define the non-escaping set of $f_N : U_N \longrightarrow V_N$ by
$$
K = \bigcap_{k=0}^\infty \overline{f_N^{-k}(V_N)}.
$$
Then $K$ is a Jordan curve by Lemma \ref{lem:sal}.

Since all $f_n$ comes from a single hyperbolic component, we will abuse notations and regard $K$ as in the topological model $\sigma:J \longrightarrow J$ of the Julia set.
The realization of $K$ in $J_n = J(f_n)$ will be denoted by $\phi_n(K)$.

Let $\mathcal{K}_m: = \cup_{i=0}^m \sigma^{-i}(K)$ and $\mathcal{K} := \cup_{m=0}^\infty \mathcal{K}_m$.
\begin{lem}
$\mathcal{K}$ is a nested set of circles.
\end{lem}
\begin{proof}
Let $\vec{v} \in T_a{^r\Hyp^3}$ and $\vec{w} \in T_b{^r\Hyp^3}$ so that $U_{\vec{v}}$ and $U_{\vec{w}}$ are disjoint from $I = [a,b]$.
Let $D_{\vec{v},n}$ and $D_{\vec{w},n}$ be domains approximating $U_{\vec{v}}$ and $U_{\vec{w}}$ respectively.
Choose $z_{\vec{v},n} \in D_{\vec{v}, n}$ and $z_{\vec{w},n} \in D_{\vec{w},n}$.

We will now argue by induction:
we assume that $\mathcal{K}_m$ is a nested set of circles.
Let $P_m := \cup_{i=0}^m F^{-i}(x)$. 
If $y\in P_{m+1}-P_m$ with $F(y) = w \in P_m$, then there exists $t$ and integer $e \geq 2$ so that $F$ maps $U^{y-t, y+t}\subset  {^r\Hyp^3}$ to $U^{w-et, w+et}$.
We may choose $t$ small enough so that $U^{y-t, y+t} \cap P_m = \emptyset$.
By Corollary \ref{cor:nim}, $U^{y-t, y+t}$ contains no critical ends.
Let $U_n$ and $V_n$ be sequences approximating $U^{y-t, y+t}$ and $U^{w-et, w+et}$ as in Lemma \ref{ApproxAnnuli}.
Then
$$
f_n:U_n \longrightarrow V_n
$$
is a degree $e$ covering $\omega$-almost surely.
Since $U^{y-t, y+t} \cap P_n = \emptyset$, by Lemma \ref{lem:cap}
$U_n \cap \phi_n(\mathcal{K}_m) = \emptyset$
$\omega$-almost surely.
Since $U^{y-t, y+t}$ separates $D_{\vec{v},n}$ and $D_{\vec{w},n}$, by Lemma \ref{lem:cap}, $U_n$ separates $z_{\vec{v},n}, z_{\vec{w}, n}$ $\omega$-almost surely.
Similarly, any circle in $\mathcal{K}_m$ separates $z_{\vec{v},n}, z_{\vec{w}, n}$ $\omega$-almost surely.

Let $N$ be in the $\omega$-big set that satisfies the above properties.
Let $\phi_N(C)$ be the component of $\phi_N^{-(m+1)}(K)$ in $U_N$. 
Then $\phi_N(C) \cup \phi_N(\mathcal{K}_m)$ is still a nested set of circles, so $C \cup \mathcal{K}_m$ is a nested set of circles.
We can now add more $m+1$-th preimages of $K$ into $C \cup \mathcal{K}_m$ in a similar way.
Therefore, by induction, $\mathcal{K}$ is a nested set of circles.
\end{proof}

From the construction above, we also have
\begin{prop}\label{PiSemiConjugacy}
The natural ordering on the nested set $\mathcal{K}$ is compatible with the linear ordering on $P= \cup_{n=0}^\infty P_n$, and the map $\pi$ sending a component $C$ of $\mathcal{K}$ to the associated point in $P$ is a semi-conjugacy.
\end{prop}

We are now ready to prove:
\begin{theorem}\label{BoundedEscapeImpliesNestedJuliaSet}
Let $\HC$ be a hyperbolic component. If $[f_n] \in \HC$ is degenerating with markings $\phi_n$ such that 
$$
F : {^r\Hyp^3} \longrightarrow {^r\Hyp^3}
$$ 
has no repelling periodic ends, then $\HC$ has nested Julia sets.
\end{theorem}
\begin{proof}
Let $[f]\in \HC$. Abusing the notation, we assume the topological model of the action on the Julia set is given by $f:J \longrightarrow J$.

First, we will show that the Julia set $J$ is disconnected. 
To show this, we will show $\sum_{i=1}^k 1/d_i < 1$, where $d_i$ is defined as in Proposition \ref{PropertyLimitingMap}.
Replacing $F$ by its second iterates and switch the role of $a$ and $b$ if necessary, we may assume $a$ is fixed by $F$.

Let $p_n, q_n \in P$ with $p_n \to a$, $q_n \to b$, and $C_n = \pi^{-1}(p_n), D_n = \pi^{-1}(q_n)$.
We define $A_n$ to be the annulus bounded by $C_n$ and $D_n$, then
$$
A = \cup_{n=1}^\infty A_n
$$ 
is again an annulus.
Let $p_{i,n}$ and $q_{i,n}$ be the $i$-th in the linear ordering on $[a,b]$ of the preimages of $p_n$ and $q_n$, and $C_{i,n} = \pi^{-1}(p_{i,n})$ and $D_{i,n} = \pi^{-1}(q_{i,n})$ respectively.
Let $A_{i,n}$ be the annulus bounded by $C_{i,n}$ and $D_{i,n}$, and
$$
A_i = \cup_{n=1}^\infty A_{i,n}.
$$

For each $n$, since $[p_{i,n}, q_{i,n}] \subseteq (p_m, q_m)$ for sufficiently large $m$, $A_{i,n} \subseteq A$.
Thus $A_i\subseteq A$ and the inclusion map is an isomorphism on fundamental groups.
Since $\pi$ is a semi-conjugacy by Proposition \ref{PiSemiConjugacy}, $A_i$ is mapped to $A$ as a degree $d_i$ covering, so $m(A_i) = m(A) / d_i$.
If $\sum_{i=1}^k 1/d_i = 1$, then by the equality case of the Gr\"otzch inequality (see Theorem B.5 in \cite{McM94}), $A_i$ and $A_{i+1}$ shares a Jordan curve boundary. 
Since $f(A_i) = f(A_{i+1}) = A$, $f$ has a critical point on this boundary, which is a contradiction as this boundary is in the Julia set, and $f$ is hyperbolic.

Since $\sum_{i=1}^k 1/d_i < 1$, the non-escaping set $\bigcap_{i=0}^\infty F^{-i}(I)$ is a Cantor set. Since $P \subseteq \bigcap_{i=0}^\infty F^{-i}(I)$, $P$ is not dense in $I$, so $J = \overline{\mathcal{K}}$ is not connected.

Now we will prove every component separates two points.
Let $p_n$ be close to $a$, and $C_n = \pi^{-1}(p_n)$.
Let $U_n$ be bounded by $C_n$ containing the Julia component associated with $a$.
Then $f:U_n \longrightarrow f(U_n)$ is a polynomial like map.
Since the backward orbits of $C_n$ under the polynomial like restriction are all Jordan curves, the non-escaping set of $f:U_n \longrightarrow f(U_n)$ is connected.

Thus, the extremal Julia component $K_a$ associated with $a$ is the Julia set associated to this polynomial like restriction.
Since $f$ is hyperbolic, there exists a Fatou component $U_1$ of $f$ whose boundary is contained in $K_a$.
Similarly, we can find a Fatou component $U_2$ of $f$ whose boundary is contained in $K_b$.

Let $p_1 \in U_1, p_2\in U_2$ be two points.
Then any Jordan curve in $\mathcal{K}$ separates $p_1, p_2$.
Since $J = \overline{\mathcal{K}}$, every component of $J$ separates $p_1, p_2$.
\end{proof}

\begin{rmk}
If we have a degenerating sequence of flexible Latt\`es maps of degree $d^2$, the limiting map $F$ also provides an example with no repelling periodic ends.
In this case, we have $k = d$ and each $d_i = d$ (so $\sum_{i=1}^k 1/d_i = 1$).
Furthermore, the nested set of circle $\mathcal{K}$ is dense in $\hat\C$ giving $J = \hat\C$.
\end{rmk}

Now Theorem \ref{CantorCircle} is a straight forward consequence of the above results:
\begin{proof}[Proof of Theorem \ref{CantorCircle}]
Combining Theorem \ref{NestedJuliaSetImpliesBoundedEscape} and Theorem \ref{BoundedEscapeImpliesNestedJuliaSet}, we get Theorem \ref{CantorCircle}.
\end{proof}

We also have the following theorem, which gives quantitative control on how periodic cycles escaping to infinity if $\HC$ is {\em not} nested.
\begin{theorem}\label{LimitingLengthSpectrum}
Let $H$ be a hyperbolic component which does not have nested Julia sets, and let $[f_n]\in \HC$ be a degenerating sequence with markings $\phi_n$. Then there exists some $C\in \mathscr{S}$ with $\lim_\omega L(C, [f_n])/r([f_n]) \neq 0$.
\end{theorem}
\begin{proof}
Combining Theorem \ref{BoundedEscapeImpliesNestedJuliaSet} and Theorem \ref{TL}, we get the result.
\end{proof}

\subsection*{Open questions}
We conclude our discussion on the length spectrum with the following open questions:
\begin{question}\label{Q1}
Let $\HC$ be a hyperbolic component admitting bounded escape. 
Do all degenerating sequences in this hyperbolic component $\HC$ have bounded multipliers?
\end{question}

Let $M(f_n)$ be the modulus of the annulus bounded by the two extremal Julia components.
If $M(f_n)$ are bounded from below, then a similar argument as in the proof of Theorem \ref{NestedJuliaSetImpliesBoundedEscape} can be used to show that the lengths of a periodic cycle $C$ stay bounded.
Therefore, to get an unbounded length spectrum, $M(f_n)$ has to tend to $0$.
Conversely, if $[f_n]$ is degenerating and has bounded length spectrum, the proof of Theorem \ref{BoundedEscapeImpliesNestedJuliaSet} implies that $M(f_n)$ tends to $\infty$.
Therefore, the above is equivalent to the following:
\begin{question}\label{Q2}
Does there exist a positive lower bound for $M(f)$ in $\HC$?
\end{question}

%%%%%%%%%%%%
% REFERENCES
%%%%%%%%%%%%


\begin{thebibliography}{McM09a}

\bibitem[AM77]{AM77}
William Abikoff and Bernard Maskit.
\newblock Geometric decompositions of Kleinian groups.
\newblock {\em Amer. J. Math.}, 99(4):687--697, 1977.

\bibitem[Ahl73]{A73}
Lars Ahlfors.
\newblock {\em Conformal invariants: topics in geometric function theory.}
\newblock McGraw Hill, 1973.

\bibitem[Arf17]{Arfeux17}
Matthieu Arfeux.
\newblock Dynamics on trees of spheres.
\newblock {\em J. Lond. Math. Soc. (2)}, 95(1):177--202,
  2017.

\bibitem[Bea91]{Beardon91}
Alan Beardon.
\newblock {\em Iteration of rational functions}.
\newblock Graduate texts in mathematics; 132. Springer-Verlag, New York, 1991.

\bibitem[Bes88]{Bestvina88}
Mladen Bestvina.
\newblock Degenerations of the hyperbolic space.
\newblock {\em Duke Math. J.}, 56(1):143--161, 1988.

\bibitem[Bes01]{Bestvina01}
Mladen Bestvina.
\newblock $\mathbb{R}$-trees in topology, geometry, and group theory.
\newblock In R.J. Daverman and R.B. Sher, editors, {\em Handbook of Geometric
  Topology}, pages 55 -- 91. North-Holland, Amsterdam, 2001.

\bibitem[BF14]{BrannerFagella14}
Bodil Branner and N\'uria Fagella.
\newblock {\em Quasiconformal Surgery in Holomorphic Dynamics}.
\newblock Cambridge Studies in Advanced Mathematics. Cambridge University
  Press, 2014.

\bibitem[BR10]{BakerRumely10}
Matthew Baker and Robert~S. Rumely.
\newblock {\em Potential Theory and Dynamics on the Berkovich Projective Line}.
\newblock Mathematical surveys and monographs. American Mathematical Soc.,
  2010.

\bibitem[Chi91]{Chiswell91}
Ian Chiswell.
\newblock Non-standard analysis and the morgan-shalen compactification.
\newblock {\em Q. J. Math.}, 42(1):257--270, 1991.

\bibitem[DE86]{DouadyEarle86}
Adrien Douady and Clifford~J. Earle.
\newblock Conformally natural extension of homeomorphisms of the circle.
\newblock {\em Acta Math.}, 157:23--48, 1986.

\bibitem[DeM05]{DeM05}
Laura DeMarco.
\newblock Iteration at the boundary of the space of rational maps.
\newblock {\em Duke Math. J.}, 130(1):169--197, 2005.

\bibitem[DeM07]{DeM07}
Laura DeMarco.
\newblock The moduli space of quadratic rational maps.
\newblock {\em J. Amer. Math. Soc.}, 20(2):321--355,
  2007.

\bibitem[dFM09]{deFM09}
Tommaso de~Fernex and Mircea Musta\c{t}\u{a}.
\newblock Limits of log canonical thresholds.
\newblock {\em Ann. Sci. Norm. Sup\'er. (4)}, 42(3):491--515, 2009.

\bibitem[DM08]{DeM08}
Laura DeMarco and Curtis McMullen.
\newblock Trees and the dynamics of polynomials.
\newblock {\em Ann. Sci. Norm. Sup\'er. (4)}, 41(3):337--383, 2008.

\bibitem[Eps00]{Epstein00}
Adam Epstein.
\newblock Bounded hyperbolic components of quadratic rational maps.
\newblock {\em Ergodic Theory Dynam. Systems}, 20(3):727--748, 2000.

\bibitem[Fab13a]{Faber13}
Xander Faber.
\newblock Topology and geometry of the Berkovich ramification locus for
  rational functions, I.
\newblock {\em Manuscripta Math.}, 142(3):439--474, 2013.

\bibitem[Fab13b]{Faber13b}
Xander Faber.
\newblock Topology and geometry of the Berkovich ramification locus for
  rational functions, II.
\newblock {\em Math. Ann.}, 356(3):819--844, 2013.

\bibitem[FRL10]{FR10}
Charles Favre and Juan Rivera-Letelier.
\newblock Th\'eorie ergodique des fractions rationnelles sur un corps
  ultram\'etrique.
\newblock {\em Proc. Lond. Math. Soc. (3)},
  100(1):116--154, 2010.

\bibitem[Gro92]{Gromov92}
Misha Gromov.
\newblock {\em Asymptotic invariants of infinite groups. Geometric group
  theory. Volume 2}.
\newblock London Math. Society Lecture Notes, 182. Cambridge Univ. Press.,
  1992.

\bibitem[Hub06]{Hubbard06}
John Hubbard.
\newblock {\em Teichm{\"u}ller Theory and Applications to Geometry, Topology,
  and Dynamics: Teichm{\"u}ller theory}.
\newblock Teichm{\"u}ller Theory and Applications to Geometry, Topology, and
  Dynamics. Matrix Editions, 2006.

\bibitem[Jon15]{Jonsson15}
Mattias Jonsson.
\newblock {\em Dynamics on Berkovich Spaces in Low Dimensions}, pages 205--366.
\newblock Springer International Publishing, Cham, 2015.

\bibitem[Kiw15]{Kiwi15}
Jan Kiwi.
\newblock Rescaling limits of complex rational maps.
\newblock {\em Duke Math. J.}, 164(7):1437--1470, 05 2015.

\bibitem[KL95]{KapovichLeeb95}
Michael Kapovich and Bernhard Leeb.
\newblock On asymptotic cones and quasi-isometry classes of fundamental groups
  of 3-manifolds.
\newblock {\em Geom. Funct. Anal.}, 5(3):582--603,
  1995.

\bibitem[LR75]{LightstoneRobinson75}
A.H. Lightstone and Abraham Robinson.
\newblock {\em Nonarchimedean Fields and Asymptotic Expansions}.
\newblock North-Holland Mathematical Library. Elsevier, 1975.

\bibitem[Luo19]{L19}
Yusheng Luo.
\newblock Limits of rational maps, $\mathbb{R}$-trees and barycentric
  extension.
\newblock Preprint. arXiv:1905.00915, 2019.

\bibitem[Luo20]{L20}
Yusheng Luo.
\newblock On hyperbolic rational maps with finitely connected Fatou sets.
\newblock {\em J. Lond. Math. Soc. (2)} to apear, arXiv:2004.02797, 2020.

\bibitem[McM88]{McM88}
Curtis McMullen.
\newblock Automorphisms of rational maps.
\newblock In {\em Holomorphic Functions and Moduli I}, volume~10 of {\em
  Mathematical Sciences Research Institute Publications}, pages 31--60.
  Springer, 1988.

\bibitem[McM94]{McM94}
Curtis McMullen.
\newblock {\em Complex Dynamics and Renormalization (AM-135)}.
\newblock Princeton University Press, 1994.

\bibitem[McM08]{McM08}
Curtis McMullen.
\newblock Thermodynamics, dimension and the {Weil-Petersson} metric.
\newblock {\em Invent. Math.}, 173(2):365--425, 2008.

\bibitem[McM09a]{McM09b}
Curtis McMullen.
\newblock A compactification of the space of expanding maps on the circle.
\newblock {\em Geom. Funct. Anal.}, 18(6):2101--2119, 2009.

\bibitem[McM09b]{McM09}
Curtis McMullen.
\newblock Ribbon $\mathbb{R}$-trees and holomorphic dynamics on the unit disk.
\newblock {\em J. Topol.}, 2(1):23--76, 2009.

\bibitem[McM10]{McM10}
Curtis McMullen.
\newblock Dynamics on the unit disk : short geodesics and simple cycles.
\newblock {\em Comment. Math. Helv.}, 85(4): 723--749, 2010.

\bibitem[MS84]{MorganShalen84}
John Morgan and Peter Shalen.
\newblock Valuations, trees, and degenerations of hyperbolic structures, {I}.
\newblock {\em Ann. of Math. (2)}, 120(3):401--476, 1984.

\bibitem[MS98]{McMS98}
Curtis McMullen and Dennis Sullivan.
\newblock Quasiconformal homeomorphisms and dynamics III. the Teichm\"uller space
  of a holomorphic dynamical system.
\newblock {\em Adv. Math.}, 135(2):351 -- 395, 1998.

\bibitem[NP20]{NP18}
Hongming Nie and Kevin Pilgrim.
\newblock Boundedness of hyperbolic components of newton maps.
\newblock {\em Israel J. Math.}, 238:831 -- 869, 2020.

\bibitem[Ota96]{O96}
Jean-Pierre Otal.
\newblock Le th\'eor\`eme d'hyperbolisation pour les vari\'et\'es fibr\'ees de
  dimension 3.
\newblock {\em Asterisque, Soci\'et\'e Math\'ematique de France}, (235), 1996.

\bibitem[Pau88]{Paulin88}
Fr{\'e}d{\'e}ric Paulin.
\newblock Topologie de {Gromov} {\'e}quivariante, structures hyperboliques et
  arbres r{\'e}els.
\newblock {\em Invent. Math.}, 94(1):53--80, 1988.

\bibitem[Pet11]{Petersen11}
Carsten Petersen.
\newblock Conformally natural extensions revisited.
\newblock arXiv:1102.1470, 2011.

\bibitem[PT00]{PilgrimTan00}
Kevin Pilgrim and Lei Tan.
\newblock Rational maps with disconnected Julia set.
\newblock {\em Asterisque, Soci\'et\'e Math\'ematique de France}, 261:349--384,
  01 2000.

\bibitem[QYY15]{QiuYangYin15}
Weiyuan Qiu, Fei Yang, and Yongcheng Yin.
\newblock Rational maps whose julia sets are cantor circles.
\newblock {\em Ergodic Theory Dynam. Systems}, 35(2):499--529, 2015.

\bibitem[RL03]{Rivera-Letelier03}
Juan Rivera-Letelier.
\newblock Espace hyperbolique $p$-adique et dynamique des fonctions
  rationnelles.
\newblock {\em Compos. Math.}, 138(2):199--231, 2003.

\bibitem[RL05]{Rivera-Letelier05}
Juan Rivera-Letelier.
\newblock Points p\'eriodiques des fonctions rationnelles dans l'espace
  hyperbolique $p$-adique.
\newblock {\em Comment. Math. Helv.}, 80(3):593--629, 2005.

\bibitem[Roe03]{Roe03}
John Roe.
\newblock {\em Lectures on Coarse Geometry}.
\newblock American Mathematical Society, Providence, 2003.

\bibitem[Shi87]{Shishikura87}
Mitsuhiro Shishikura.
\newblock On the quasiconformal surgery of rational functions.
\newblock {\em Ann. Sci. \'Ec. Norm. Sup\'er. (4)}, 20(1):1--29, 1987.

\bibitem[Shi89]{Shishikura89}
Mitshuhiro Shishikura.
\newblock Trees associated with the configuration of Herman rings.
\newblock {\em Ergodic Theory Dynam. Systems}, 9(3):543--560, 1989.

\bibitem[Sti93]{Stimson93}
James Stimson.
\newblock {\em Degree two rational maps with a periodic critical point.}
\newblock PhD thesis, University of Liverpool, 1993.

\end{thebibliography}
\end{document}